\documentclass{article}
\usepackage{arxiv}

\usepackage[utf8]{inputenc} 
\usepackage[T1]{fontenc}    
\usepackage{hyperref}       
\usepackage{url}            
\usepackage{booktabs}       
\usepackage{amsfonts}       
\usepackage{nicefrac}       
\usepackage{microtype}      
\usepackage{doi}

\usepackage{amsthm,amsmath,amssymb}
\usepackage{graphicx,xcolor}
\usepackage[normalem]{ulem}
\usepackage{setspace}
\usepackage{accents}
\usepackage{algpseudocode}
\usepackage{tikz}
\usetikzlibrary{calc}
\usetikzlibrary{arrows}
\usepackage{multicol}
\usepackage{cite}
\usepackage{comment}
\usepackage{ulem}
\newcommand{\dout}[1]{\unskip}

\newcommand\str{\bgroup\markoverwith{\textcolor{red}{\rule[0.5ex]{2pt}{1.5pt}}}\ULon} 

\newtheorem{theorem}{Theorem}[section]

\newtheorem{assumption}{Assumption}[section]
\newtheorem{lemma}[theorem]{Lemma}

\DeclareMathOperator*{\argmin}{\arg\!\min}

\usepackage{breqn}
\usepackage{todonotes}
\usepackage{etoolbox}
\makeatletter
\patchcmd{\@addmarginpar}{\ifodd\c@page}{\ifodd\c@page\@tempcnta\m@ne}{}{}
\makeatother
\reversemarginpar

\usepackage{multibib}
\newcites{latex}{\small SUPPLEMENTARY MATERIAL: REFERENCES}
\def\grad{\nabla}

\def\be{\mathbf{e}}

\def\bp{\mathbf{p}}

\def\bs{\mathbf{s}}

\def\bu{\mathbf{u}}
\def\bv{\mathbf{v}}
\def\bw{\mathbf{w}}
\def\bx{\mathbf{x}}  
\def\by{\mathbf{y}}

\def\bC{\mathbf{C}}
\def\bD{\mathbf{D}}

\def\bI{\mathbf{I}}
\def\bJ{\mathbf{J}}

\def\bL{\mathbf{L}}

\def\bQ{\mathbf{Q}}

\def\cA{\mathcal{A}}
\def\cB{\mathcal{B}}
\def\cC{\mathcal{C}}

\def\cE{\mathcal{E}}
\def\cF{\mathcal{F}}
\def\cG{\mathcal{G}}
\def\cH{\mathcal{H}}
\def\cI{\mathcal{I}}

\def\cK{\mathcal{K}}
\def\cL{\mathcal{L}}

\def\cN{\mathcal{N}}
\def\cO{\mathcal{O}}
\def\cP{\mathcal{P}}

\def\cR{\mathcal{R}}
\def\cS{\mathcal{S}}
\def\cT{\mathcal{T}}

\def\cW{\mathcal{W}}
\def\cX{\mathcal{X}}
\def\cY{\mathcal{Y}}

\def\smskip{\smallskip}

\def\texitem#1{\par\smskip\noindent\hangindent 25pt
               \hbox to 25pt {\hss #1 ~}\ignorespaces}

\def\norm#1{\left\|#1\right\|}

\newcommand{\BEAS}{\begin{eqnarray*}}
\newcommand{\EEAS}{\end{eqnarray*}}
\newcommand{\BEA}{\begin{eqnarray}}
\newcommand{\EEA}{\end{eqnarray}}
\newcommand{\BEQ}{\begin{eqnarray}}
\newcommand{\EEQ}{\end{eqnarray}}
\newcommand{\BIT}{\begin{itemize}}
\newcommand{\EIT}{\end{itemize}}
\newcommand{\BNUM}{\begin{enumerate}}
\newcommand{\ENUM}{\end{enumerate}}

\newcommand{\BA}{\begin{array}}
\newcommand{\EA}{\end{array}}

\newcommand{\ones}{\mathbf 1}

\newcommand{\reals}{\mathbb{R}}
\newcommand{\integers}{\mathbb{Z}}

\newcommand{\diag}{\mathop{\bf diag}}

\newcommand{\dom}{\mathop{\bf dom}}

\newif\ifpagenumbering
\pagenumberingtrue

\pagenumberingfalse

\newsavebox{\theorembox}
\newsavebox{\lemmabox}
\newsavebox{\defnbox}
\newsavebox{\corollarybox}
\newsavebox{\remarkbox}
\newsavebox{\assbox}
\savebox{\theorembox}{\noindent\bf Theorem}
\savebox{\lemmabox}{\noindent\bf Lemma}
\savebox{\defnbox}{\noindent\bf Definition}
\savebox{\corollarybox}{\noindent\bf Corollary}
\savebox{\remarkbox}{\noindent\bf Remark}
\newtheorem{remark}{\usebox{\remarkbox}}[section]
\newtheorem{defn}{\usebox{\defnbox}}

\def\fprod#1{\left\langle#1\right\rangle}
\def\prox#1{\mathbf{prox}_{#1}}
\def\ind#1{\mathbb{I}_{#1}}
\def\T{\mathsf{T}}
\def\id{\mathbf{I}}
\def\zero{\mathbf{0}}
\def\one{\mathbf{1}}

\def\btheta{\boldsymbol{\theta}}
\def\blambda{\boldsymbol{\lambda}}

\def\bmu{\boldsymbol{\nu}}

\def\bom{\boldsymbol{\omega}}
\def\st{{\rm s.t.}}

\def\tcR{\widetilde{\cR}}

\def\ttau{\tilde{\tau}}
\newcommand{\ubar}[1]{\underaccent{\bar}{#1}}

\def\sy{\mathbb{S}}
\def\ns{\mathbf{Null}}
\def\spn{\mathbf{Span}}

\def\sa#1{\textcolor{black}{#1}}
\def\nsa#1{\textcolor{black}{#1}}

\def\rev#1{\textcolor{black}{#1}}
\def\eyh#1{\textcolor{black}{#1}}
\def\str#1{\textcolor{red}{#1}}

\def\eh#1{\textcolor{black}{#1}}
\def\saa#1{\textcolor{black}{#1}}
\def\eyy#1{\textcolor{black}{#1}}

\def\ehh#1{\textcolor{black}{#1}}
\def\cC{\Xi}
\def\Ct{\widetilde{\cC}}
\usepackage{hyperref}
\hypersetup{
    colorlinks=true,
    linkcolor=blue,
    filecolor=magenta,
    urlcolor=cyan,
}
\usepackage{calc}
\usepackage{accents}
\usepackage{nicefrac}
\newcommand{\dbtilde}[1]{\accentset{\approx}{#1}}

\title{A Decentralized Primal-Dual Method for Constrained Minimization of a Strongly Convex Function}

\author{Erfan Yazdandoost Hamedani\\
Systems and Industrial Engineering Department\\
	The University of Arizona \\
    Tucson, AZ 85721,\\
	\texttt{erfany@arizona.edu}
	\And
	Necdet Serhat Aybat\\
	Industrial and Manufacturing Engineering Department\\
	Penn State University\\
	University Park, PA 16802,\\
	\texttt{nsa10@psu.edu}
}

\date{}

\renewcommand{\shorttitle}{Decentralized Constrained Minimization of a Strongly Convex Function}

\hypersetup{
pdftitle={Decentralized Constrained Minimization of a Strongly Convex Function},
pdfauthor={E. Yazdandoost Hamedani, N.S. Aybat},
pdfkeywords={Strongly Convex Function, Decentralized Optimization},
}

\begin{document}

\maketitle
\begin{abstract}
We propose decentralized primal-dual methods for cooperative multi-agent consensus optimization problems over both static and {time-varying communication networks}, where only local communications are allowed. The objective is to minimize {the sum of agent-specific
convex functions} over 
conic constraint sets defined by agent-specific nonlinear functions; hence, the optimal consensus decision should lie in the intersection of these private sets. 
\saa{Under the strong convexity assumption,} we provide convergence rates for sub-optimality, infeasibility and consensus violation in terms of the number of communications required; examine the effect of underlying network topology on the convergence rates. 
\end{abstract}
\allowdisplaybreaks
\vspace*{-3mm}
\section{Introduction}\label{sec:intro}
{There has been much recent interest in decentralized optimization over communication networks that arise in various applications such as i) distributed parameter estimation in wireless sensor networks~\cite{schizas2008consensus}; ii) multi-agent cooperative control and coordination in multirobot networks~\cite{zhou2011multirobot};
iii) processing distributed big-data in (online) machine learning~\cite{tsianos2012consensus,duchi2012dual}; iv) power control problem in cellular networks~\cite{ram2009distributed}.}

\rev{In this paper, from a broader perspective, we aim to study constrained distributed optimization of a strongly convex function over time-varying communication networks $\cG^t=(\cN,\cE^t)$ for $t\geq 0$; in particular, from an application perspective, we are motivated to design an efficient \emph{decentralized} solution method for finding a point in the intersection of ellipsoids that arises in many control problems which will be discussed shortly.}

\noindent {\bf Applications.} \rev{1) In \emph{multi-agent control design},
consider computing an optimal consensus decision satisfying the constraint of each agent $i\in\cN$, involving some {\it uncertain parameter} $q_i\in\reals^n$, with at least $1-\epsilon$ probability, i.e., 
$\min_x\{\sum_{i\in\cN} f_i(x)\mid\mathbb{P}(\{q_i: q_i^\top x\leq b_i\})\geq 1-\epsilon, i\in\cN\}$. 
When the 
distribution of the uncertain parameters are elliptical log-concave, the problem can be formulated as a minimization 
over an intersection of ellipsoids \cite{lagoa2005probabilistically}. 2) \emph{Multi-agent localization in sensor networks} aims to collaboratively locate a target $\bar{x}\in\reals^n$. Suppose each agent $i\in\cN$ has a directional sensor which can detect a target when it belongs to $\cE_i=\{A_i y_i:~\norm{y_i}_2\leq d_i\}$. Let $\cI=\{i\in\cN:~\bar{x}\in\cE_i\}$. Here, one might estimate the target location $\bar{x}$ by solving $\min_x\{\norm{x}_2:~x\in\cap_{i\in\cI}\cE_i\}$.} 
3) \rev{\emph{Rendez-vous point problem} arises in multi-agent control systems where the objective is to compute a common point in the reachability sets of all agents~\cite{stewart2010cooperative}. Consider a  linear system $x_i(t+1)=A_i x_i(t)+B_iu_i(t)$ for each agent with \saa{a} finite horizon $T$. Suppose the control variable $u_i(t)$ 
must satisfy $\norm{u_i(t)}\leq r_i$ for $i\in\cN$ and $t=0,\hdots,T-1$. Since $x_i(T)$ is an affine map of $\{u_i(t)\}_{t=0}^{T-1}$, $x_i(T)$ lies in an ellipsoid $\cE_i(T)$; hence, 
agents can meet at time $T$ if they find a point in $\cap_{i\in\cN}\cE_i(T)$. 4) In \emph{distributed robust optimization}, one aims to 
solve $\min_{x\in X}\{\sum_{i\in\cN} f_i(x)\mid g_i(x,q)\leq 0, ~\forall q\in Q,~ \forall i\in\cN\}$ where $Q$ represents an uncertainty set. This problem can be reformulated as \eqref{eq:central_problem} using a scenario based approach~\cite{you2018distributed}, and includes
robust model predictive control, robust regression 
as special cases\cite{bertsimas2011theory,ben2009robust}.}

\rev{
For decentralized 
constrained consensus optimization 
over communication networks, we propose
\ehh{a \emph{distributed} primal-dual algorithm (DPDA) for static networks and its time-varying variant (DPDA-TV) for time-varying 
networks.}
DPDA and {DPDA-TV are} based on the accelerated primal-dual~(APD) algorithm we recently proposed \cite{hamedani2018primal} for convex-concave saddle-point (SP) problems.}

\noindent \textbf{Problem Description.} {Let $\{\cG^t\}_{t\in\reals_+}$ denote a time-varying graph of $N$ \emph{computing} nodes. More precisely, for all $t\geq 0$, the graph has the form $\cG^t=(\cN,\cE^t)$, where $\cN\triangleq\{1,\ldots,N\}$ is the set of nodes and $\mathcal{E}^t\subseteq \mathcal{N}\times \mathcal{N}$ 
is the set of (possibly directed) edges at time $t$.} Suppose that each node $i\in\cN$ has a \emph{private} (local) cost function ${\varphi_i}:\reals^n\rightarrow\reals\cup\{+\infty\}$ such that
{
\begin{equation}
\label{eq:F_i}
{\varphi_i}(x)\triangleq \rho_i(x) + f_i(x),
\end{equation}
}%
where $\rho_i: \mathbb{R}^n \rightarrow \mathbb{R}\cup\{+\infty\}$ is a possibly \emph{non-smooth} convex function, and $f_i: \mathbb{R}^n \rightarrow \mathbb{R}$ is a \emph{smooth} convex function.

Consider the following minimization problem:
{
\begin{align}\label{eq:central_problem}
x^*\in \argmin_{x\in\reals^n}\{\bar{\varphi}(x) \triangleq \sum_{i\in \mathcal{N}}\varphi_i(x):~g_i(x) \in -\mathcal{K}_i,\  i\in\mathcal{N}\}, \vspace*{-2mm}
\end{align}}%
\noindent where for each $i\in\cN$, $\mathcal{K}_i\subseteq{\reals}^{m_i}$ is a closed, convex cone and $g_i:\reals^n \rightarrow \reals^{m_i}$ is a 
$\cK_i$-convex function~\cite[Chapter 3.6.2]{boyd2004convex}; \rev{moreover, both $\cK_i$ and $g_i$ are only known to node $i\in\cN$.}

\rev{We assume the following: \textbf{i)} projections onto $\cK_i$ can be computed efficiently, while the projection onto the \emph{preimage} ${\chi_i\triangleq} g_i^{-1}(-\cK_i)$ is 
\emph{not} practical, e.g., when $\cK_i$ is the positive semidefinite cone, computing a projection onto 
$\chi_i$ requires solving a nonlinear SDP; \textbf{ii)} \saa{each $\varphi_i$ is strongly convex for $i\in\cN$}; \textbf{iii)} nodes are willing to collaborate, without sharing their \emph{private} data defining $\chi_i$ and $\varphi_i$, to compute an optimal consensus decision;
moreover, \textbf{iv)} nodes are only allowed to communicate with neighbors 
over the links; and finally, 
\textbf{v)} all nodes are synchronized with respect to a global clock known to all nodes and that every time the global clock ticks, one communication round with instant messaging among neighboring nodes takes place and each local operation can be completed between two tics of the clock.} 

\noindent \textbf{Previous Work.} Consider $\min_{x\in\reals^n}\{\bar{\varphi}(x):\ x\in\cap_{i\in\cN}\chi_i\}$ over a communication network of computing agents $\cN$, where $\bar{\varphi}(x)=\sum_{i\in\cN}\varphi_i(x)$.  \rev{Although the \emph{unconstrained} consensus optimization, i.e., $\chi_i=\reals^n$, is well studied for static networks (see \cite{makhdoumi17,chang2015multi,shi2015extra,zeng2015extrapush,Dextra16,xi2016add}) or time-varying networks (see \cite{Nedic08_1J,nedic2014distributed,nedich2016achieving}),
the \emph{constrained} case 
is currently an area of active research, e.g.,
\cite{Nedic08_1J,nedic2014distributed,nedic2010constrained,srivastava2010distributed,zhu2012distributed,yuan2011distributed,chang2014distributed,mateos2015distributed,chang2014proximal,yuan2015regularized,aybat2016primal}.
}
For constrained consensus optimization, other than few exceptions, e.g.,~\cite{zhu2012distributed,yuan2011distributed,chang2014distributed,mateos2015distributed,chang2014proximal,yuan2015regularized,aybat2016primal}, the existing methods require that each node compute a projection on the local set $\chi_i$ in addition to consensus and (sub)gradient steps, e.g., \cite{nedic2010constrained,srivastava2010distributed}. 
Among those few exceptions, 
only \cite{chang2014distributed,mateos2015distributed,chang2014proximal,yuan2015regularized,aybat2016primal} can handle agent-specific constraints without assuming global knowledge of the constraints by all agents. However, \emph{no} rate results in terms of suboptimality, local infeasibility, and consensus violation exist for the primal-dual distributed methods in~\cite{chang2014distributed,mateos2015distributed,chang2014proximal,yuan2015regularized} when implemented for the agent-specific \emph{conic} constraint sets {$\chi_i=\{x\in\reals^n:~g_i(x)\in -\cK_i\}$} studied in this paper -- our preliminary results have appeared in~\cite{hamedani2017multi}.
{Below, we briefly discuss some closely related works.}

\rev{In~\cite{mateos2015distributed},
the authors proposed a distributed algorithm on time-varying  networks for solving 
SP problems subject to consensus constraints. The algorithm can also solve consensus optimization problems with \emph{inequality} constraints that can be written as summation of local convex functions of local and global variables. It is shown that using a 
decreasing step-size sequence, the ergodic average of primal-dual sequence converges with $\mathcal{O}(1/\sqrt{k})$ rate in terms of saddle-point evaluation error; however, when applied to constrained optimization problems, \emph{no} rate in terms of either suboptimality or infeasibility is provided. In~\cite{chang2014proximal}, a proximal dual consensus ADMM method, PDC-ADMM, is proposed by Chang to minimize $\bar{\varphi}$ subject to a coupling equality and agent-specific linear constraints over a static network -- time-varying undirected networks are also considered 
with agents on/off and communication links failing randomly with certain probabilities.
It is shown that both for static and time-varying cases, PCD-ADMM have $\cO(1/k)$ ergodic convergence rate in the mean for suboptimality and infeasibility when each function in the objective is strongly convex for $i\in\cN$. More recently, in~\cite{aybat2016primal},
we proposed a distributed primal-dual method to solve \eqref{eq:central_problem} when $\varphi_i=\rho_i+f_i$ is composite convex \saa{and $g_i$ is affine}. Assuming $f_i$ is smooth, $\cO(1/k)$ ergodic rate is shown for suboptimality and infeasibility. Here, we aim to improve on this rate and achieve $\cO(1/k^2)$ ergodic rate by further assuming 
\saa{each $\varphi_i$} is strongly convex.}

\noindent\textbf{Contribution.} To the best of our knowledge, only a handful of methods, e.g.,~\cite{chang2014distributed,mateos2015distributed,chang2014proximal,yuan2015regularized,aybat2016primal} can handle consensus problems, similar to \eqref{eq:central_problem}, with agent-specific \emph{local} constraint sets $\{\chi_i\}_{i\in\cN}$ without requiring each agent $i\in\cN$ to project onto $\chi_i$. However, no rate results in terms of suboptimality, local infeasibility, and consensus violation exist for the 
distributed methods in \cite{chang2014distributed,mateos2015distributed,chang2014proximal,yuan2015regularized} when implemented for \emph{conic}
sets $\{\chi_i\}_{i\in\cN}$ studied in this paper. \rev{We assume that projection onto $\chi_i=g_i^{-1}(-\cK_i)$ is \emph{not} practical 
while computing a projection onto $\cK_i$ is
simple. Moreover, none of these four methods can exploit the strong convexity of \saa{$\{\varphi_i\}_{i\in\cN}$} 
and to the best of our knowledge,
DPDA and DPDA-TV proposed in this paper are one of the first decentralized algorithms to solve \eqref{eq:central_problem} with $\cO(1/k^2)$ ergodic rate guarantee on both sub-optimality and infeasibility which matches with the optimal rate for the centralized setting -- see \cite{ouyang2018lower} for the lower complexity bound of $\Omega(1/k^2)$ associated with first-order primal-dual methods for bilinear SP problems.
More precisely, when 
\saa{each $\varphi_i=\rho_i+f_i$ is 
strongly convex with smooth convex $f_i$ and closed convex $\rho_i$} for $i\in\cN$, our proposed method reduces the suboptimality and infeasibility with $\mathcal{O}(1/{k^2})$ rate as $k$, the number of primal-dual iterations, increases, and it requires $\cO(k)$ and $\cO(k\log(k))$ local communications 
in total when the network topologies are static and time-varying, respectively.
It is worth noting that, our results imply that DPDA-TV can compute a point in the intersection of closed convex sets with $\cO(1/k^2)$ rate for the solution error $\sum_{i\in\cN}\norm{x_i^k-x^*}^2$ -- in a decentralized way over time-varying directed communication networks, which is faster than $\cO(1/k)$ rate of Dykstra's algorithm, e.g., see~\cite{pang2019dykstra}. 
}

\saa{Finally, in case $\bar{\varphi}$ is \emph{merely} convex, i.e., strong convexity assumption is relaxed, we get $\cO(1/k)$ rate as a corollary of our results in this paper. Indeed, this would extend our earlier results in~\cite{aybat2016primal} from affine $g_i$ to nonlinear constrained setting.}
\textbf{Notation.} Throughout 
$\norm{\cdot}$ denotes either the Euclidean norm {or the spectral norm}. Given a convex set $\cS$, ${\sigma_{\cS}(\cdot)}$ {denotes} its support function, i.e., $\sigma_{\cS}(\theta)\triangleq\sup_{w\in \cS}\fprod{\theta,~w}$, $\ind{S}(\cdot)$ {denotes} the indicator function of $\cS$, i.e., $\ind{\cS}(w)=0$ for $w\in\cS$ and is $+\infty$ otherwise, and $\cP_{\cS}(w)\triangleq\argmin\{\norm{v-w}:\ v\in\cS\}$ denote the projection onto $\cS$. For a closed convex set $\cS$, we define the distance function as $d_{\cS}(w)\triangleq\norm{\cP_{\cS}(w)-w}$. Given a convex cone $\cK\subseteq\reals^m$, let $\mathcal{K}^*$ denote its dual cone, i.e., $\mathcal{K}^*\triangleq\{\theta\in\reals^{m}: \ \langle \theta,w\rangle \geq 0\ \ \forall w\in\mathcal{K}\}$, and $\cK^\circ\triangleq -\cK^*$ denote the polar cone of $\cK$. For a given cone $\cK\subset\reals^m$, $\sigma_{\cK}(\theta)=0$ for $\theta\in\cK^\circ$ and is $+\infty$ otherwise.
Given a convex function $g:\reals^n\rightarrow\reals\cup\{+\infty\}$, its convex conjugate is defined as $g^*(w)\triangleq\sup_{\theta\in\reals^n}\fprod{w,\theta}-g(\theta)$. $\otimes$ denotes the Kronecker product, $\one_n\in \mathbb{R}^{n}$ is the vector all ones, $\id_n$ is the $n\times n$ identity matrix. $\sy^n_{++}$ $(\sy^n_+)$ denotes the cone of symmetric positive (semi)definite matrices. For $Q\succ 0$, i.e., 
$Q\in\sy^n_{++}$, $Q$-norm is defined as $\norm{z}_Q\triangleq \sqrt{z^\top Q z}$.
Given $Q\in\sy^n_{+}$, $\lambda_{\min}^+(W)$ denotes the smallest positive eigenvalue of $Q$. $\Pi$ denotes the Cartesian product. Finally, for $\theta\in\reals^n$, we adopt $(\theta)_+\in\reals^n_{+}$ to denote $\max\{\theta, \textbf{0}\}$ where max is computed componentwise. $g(n)=\Theta(f(n))$ means $\exists n_0\in\integers_+$ and $c_1,c_2>0$ such that $c_1 f(n)\leq g(n)\leq c_2 f(n)$ for $n\geq n_0.$

\rev{Now, we state some of the assumptions and definitions.} 
\begin{defn}\label{def:lip-grad-convex}
Suppose $G:\reals^n\rightarrow \reals^m$ is 
differentiable on $\reals^n$ and $m\geq 1$. The Jacobian, 
\sa{$\bJ G:\reals^n\rightarrow\reals^{m\times n}$}, is \sa{$L_G$-Lipschitz 
if there exists $L_G\in\reals^{p\times n}$ for some $p\in\integers_{++}$ such that}
\begin{equation*}
\norm{\bJ G(x)-\bJ G(\bar{x})} \leq \sa{\norm{L_G(x-\bar{x})}} \quad \forall x,\bar{x}\in\reals^n;
\end{equation*}
\sa{moreover, when $L_G=L\,\id_{n}$ for some $L>0$, abusing the notation we say $\bJ G$ is $L$-Lipschitz.} \ehh{Furthermore, $g:\reals^n\to\reals$ satisfies the descent property with $L\in\reals^{n\times n}$ if
\begin{equation*}
g(x)-g(\bar x)-\fprod{\grad g(\bar x),x-\bar x}\leq \frac{L}{2}\norm{x-\bar x}_L^2, \quad \forall x,\bar x\in\reals^n.
\end{equation*}}%

Given a pointed convex cone $\cK\subset\reals^m$, $G$ is $\cK$-convex if $\lambda G(x)+(1-\lambda)G(\bar{x})\geq_{\cK} G(\lambda x+(1-\lambda)\bar{x})$ for all $x,\bar{x}\in\reals^n$ and $\lambda\in[0,1]$, where $a\geq_{\cK}b$ if and only if $a-b\in\cK$.
\end{defn}
\begin{remark}
\label{rem:K-convex}
If $G:\reals^n\rightarrow \reals^m$ is differentiable on $\reals^n$ and $\cK$-convex, then $G(x)\geq_{\cK} G(\bar{x})+\bJ G(\bar{x})(x-\bar{x}) \quad \forall x,\bar{x}\in\reals^n$.
\end{remark}

\begin{assumption}
\label{assump:lipschitz_f}
For $i\in\cN$, $f_i$ is differentiable on an open set containing $\dom \rho_i$ with an $L_{f_i}$-Lipschitz 
gradient $\grad f_i$ for some $L_{f_i}>0$. \saa{Define $\sa{L_{\max}(f)}\triangleq \max_{i\in\cN}L_{f_i}$ and { $\bar{L}\triangleq\sqrt{\sum_{i\in\cN}L_{f_i}^2/N}$.}} 
We also assume the prox map \saa{$\prox{\tau\rho_i}(\cdot)$ is \emph{efficiently} computable for any $\tau>0$}, where\vspace*{-1mm} 
{
\begin{equation}
\label{eq:prox}
\prox{\saa{\tau}\rho_i}(x)\triangleq\argmin_{y \in \reals^n} \left\{\saa{\tau}\rho_i(y)+\tfrac{1}{2}\norm{y-x}^2 \right\}.\vspace*{-2mm}
\end{equation}
}%
\end{assumption}
\begin{assumption}
\label{assump:saddle-point}
The duality gap for \eqref{eq:central_problem} is zero, and a primal-dual solution to~\eqref{eq:central_problem} exists.
\end{assumption}
A sufficient condition 
is the existence of a Slater point, i.e., there exists $\bar{x}\in \mathbf{relint}(\dom \bar{\varphi})$ such that $g_i(\bar{x})\in \mathbf{int}(-\cK_i)$ for $i\in\cN$, where $\dom\bar{\varphi}=\cap_{i\in\cN}\dom\varphi_i$.
\begin{assumption}
\label{assump:g_i}
For 
$i\in\cN$, $g_i:\reals^n\to\reals^{m_i}$ is $\cK_i$-convex and \sa{differentiable such that $\bJ g_i$ is $L_{g_i}$-Lipschitz for some $L_{g_i}>0$, and $\norm{\bJ g_i(x)}\leq C_{g_i}$ for all $x\in\dom \rho_i$ for some $C_{g_i}>0$.}
\saa{Let $\cX\triangleq\Pi_{i\in\cN}\reals^n$ and $\cX\ni\bx=[x_i]_{i\in\cN}$.
Define $G:
\cX\rightarrow \reals^m$ such that $G(\bx)\triangleq [g_i(x_i)]_{i\in\mathcal{N}}$ where $m\triangleq\sum_{i\in\cN}{m_i}$. Moreover, let $\bC=\diag([C_{g_i}\id_{m_i}]_{i\in\cN})$, $C_{\min}\triangleq \min_{i\in\cN} C_{g_i}$ and $L_{\max}(G)\triangleq \max_{i\in\cN}L_{g_i}$.}
\end{assumption}
\begin{defn}
\label{def:sconvex}
A differentiable function ${f}:\reals^n \rightarrow \reals$ is {\it strongly convex} 
with modulus $\mu>0$
if 
{
\begin{align*}
{f(x)\geq f(\bar{x})+\fprod{\nabla f(\bar{x}),x-\bar{x}}}+\tfrac{\mu}{2}\norm{x-\bar{x}}^2\quad \forall x,\bar{x}\in\reals^n.
\end{align*}}
\end{defn}
\begin{assumption}
\label{assump:sconvex}
${\bar{f}(x)}\triangleq\sum_{i\in\cN} f_i(x)$ is strongly convex with modulus $\bar{\mu}>0$; and each $f_i$ is strongly convex with modulus \saa{$\mu_i >0$} 
for $i\in\cN$, and define {$\ubar{\mu}\triangleq\min_{i\in\cN}\{\mu_i\}\saa{>0}$}. 
\end{assumption}
\begin{remark}
{Clearly $\bar{\mu}\geq \sum_{i\in\cN}\mu_i$ is always true
; moreover, $\bar{\mu}>0$ implies that $x^*$ is the unique optimal solution to \eqref{eq:central_problem}.}
\end{remark}
\vspace*{-5mm}
\subsection{Preliminary} \label{sec:pda}
In this section, we briefly present our recent work~\cite{hamedani2018primal}, where we proposed an accelerated primal-dual (APD) algorithm for solving convex-concave saddle-point (SP) problems. As it is shown in~\cite{hamedani2018primal}, 
APD can be viewed as an extension of the primal-dual algorithm proposed in~\cite{chambolle2015ergodic}, for solving bilinear SP problems, to a more general setting with a non-bilinear coupling term. Let {$\cX\subseteq \reals^{n_x}$ and $\cY\subseteq \reals^{n_y}$} be finite-dimensional vector spaces. Here we present {a slightly} extended
version of APD~\cite{hamedani2018primal} to solve the following problem:
{
\begin{equation}
\label{eq:saddle-problem}
\min_{\bx\in\cX}\max_{\by\in\cY}\cL(\bx,\by)\triangleq (\psi_x+\phi_x)(\bx)+\cH(\bx,\by)-(\psi_y+\phi_y)(\by),
\end{equation}}%
where $\psi_x$ and $\psi_y$ are, possibly non-smooth and convex, 
\sa{$\phi_x$ and $\phi_y$ are convex and 
differentiable on open sets containing \ehh{$\dom \psi_x$ and $\dom \psi_y$ satisfying the \emph{descent property} with $L_{\phi_x}\in\reals^{n_x\times n_x}$ and $L_{\phi_y}\in\reals^{n_y\times n_y}$, respectively;}} 
and $\cH:\cX\times\cY\to \reals$ is a continuously differentiable function that is convex in $\bx$ for any $\by\in\cY$, and concave in $\by$ for any $\bx\in\cX$. {We assume {$\psi_x$ and $\psi_y$} have convexity \sa{moduli} $\mu_x\geq 0$ and $\mu_y\geq 0$, respectively}.
\ehh{Moreover, for any $\bx\in\cX$, $\grad_\by \cH(\bx,\cdot)$ and $\grad_\bx \cH(\bx,\cdot)$ are $\reals_+\ni L_{yy}$-Lipschitz and \sa{$\reals^{p_{xy}\times n_y}_+\ni L_{xy}$-Lipschitz}, respectively, for some $p_{xy}\geq 1$; for any $\by\in\cY$, $\grad_\bx \cH(\cdot,\by)$ and $\grad_\by\cH(\cdot,\by)$ are $\reals_+\ni L_{xx}$- and $\reals^{p_{yx}\times n_x}_+\ni L_{yx}$-Lipschitz, respectively, for some $p_{yx}\geq 1$.}
Given 
\sa{$\{\bQ_x^k,\bQ_y^k,\eta^k\}_{k\geq 0}\subset\mathbb{S}^{n_x}_{++}\times\mathbb{S}^{n_y}_{++}\times\reals_{++}$} and initial iterates $\bx^0,\by^0$, the slightly 
\saa{modified} version\footnote{\sa{APD in~\cite{hamedani2018primal} solves \eqref{eq:saddle-problem} without smooth $\phi_x$ and $\phi_y$, and assuming $\mu_y=0$.}} of APD iterations consists of
{
\begingroup
\begin{subequations}
\label{eq:PDA}
\begin{align}
&\bp^k \gets  ~(1+\eta^k)\grad_\by \cH(\bx^k,\by^k)-\eta^k \grad_\by \cH(\bx^{k-1},\by^{k-1}), \nonumber \\
& \ehh{\by^{k+1}\gets\eyy{\argmin_{\by}}~\psi_y(\by)+\fprod{\grad\phi_y(\by^k)-\bp^k,\by}+\frac{1}{2}\norm{\by-\by^k}^2_{\bQ_y^k},}\label{eq:PDA-y}\\
&\bx^{k+1}\gets  \argmin_{\bx}  \psi_x(\bx)+\fprod{\grad \phi_x(\bx^k)+\grad_\bx \cH(\bx^k,\by^{k+1}),~\bx} +\tfrac{1}{2}\|{\bx-{\bx}^k}\|_{\bQ_x^k}^2.
\label{eq:PDA-x}
\end{align}
\end{subequations}
\endgroup
}%
\sa{Based on the discussion in~\cite{hamedani2018primal},} if $\bQ_x^k$, $\bQ_y^k$, $\eta^k$ are chosen such that \sa{there exist some \sa{$c\geq 1$}, positive \emph{nondecreasing} $\{t^{k}\}_{k\geq 0}$ and $\{\alpha^k,\beta^k\}_{k\geq 0}\subset\reals_+$ satisfying\footnote{
The step-size condition presented here is a \saa{generalized} version of \cite{hamedani2018primal}.}}
{
\begin{subequations}\label{eq:general_step_cond}
\begin{align}
&\bQ_x^k+\mu_x\id_{n_x}\succeq \bQ_x^{k+1}/\eta^{k+1},\ \bQ_y^k+\mu_y\id_{n_y}\succeq \bQ_y^{k+1}/\eta^{k+1},\\ &\bQ_x^k-L_{\phi_x}\succeq \ehh{L_{xx}\bI_{n_x}}+\frac{1}{\alpha^{k+1}}L_{yx}^\top L_{yx},\\
&\bQ_y^k-L_{\phi_y}\succeq \sa{c}~\eta^k(\alpha^k+\beta^k)\id_{n_y}+\frac{1}{\beta^{k+1}}\ehh{L_{yy}^2\bI_{n_y}}, \label{eq:4th-cond}
\end{align}
\end{subequations}}%
and $\eta^{k+1}=\frac{t^k}{t^{k+1}}\in(0,1]$ for all $k\geq 0$ \saa{-- see \cite[Assumption 3]{hamedani2018primal} with $\delta=0$}, then \saa{\cite[Eq.(3.3)]{hamedani2018primal}} implies that 
{
\begin{align}\label{eq:DPA-rate}
&{\cW_K\Big(\cL(\bar{\bx}^K,\by)-\cL(\bx,\bar{\by}^K)\Big)+\tfrac{t^{K}}{2}\norm{\bx-{\bx}^K}_{\bQ_x^K}^2 +\tfrac{t^K}{2}\norm{\by-{\by}^K}_{\bQ_y^K-\eta^K(\alpha^K+\beta^K)\id_{n_y}}^2} \nonumber\\
&\quad \leq \tfrac{1}{2}\norm{\bx-\bx^0}_{\bQ_x^0}^2+\tfrac{1}{2}\norm{\by-\by^0}_{\bQ_y^0}^2
\end{align}}%
holds for all $\bx,\by\in\cX\times\cY$ and $K\geq 1$, where $\cW_K\triangleq\sum_{k=1}^{K}t^{k-1}$, $\bar{\bx}^{K}\triangleq \cW_K^{-1}\sum_{k=1}^{K}t^{k-1}\bx^k$ and $\bar{\by}^{K}\triangleq \cW_K^{-1}\sum_{k=1}^{K}t^{k-1}\by^k$ for $K\geq 1$.

In \saa{\cite[Theorem 2.2,~Part II]{hamedani2018primal}}, it is shown that when $\mu_x>0$, $\mu_y=0$, \ehh{$L_{yy}= 0$, and $L_{\phi_y}=\mathbf{0}$}, 
\sa{\eqref{eq:4th-cond} reduces to {$\bQ_y^k\succeq \eta^k\alpha^k\id_{n_y}$} for \sa{$\beta^k=0$ and $c=1$; hence,} choosing $\bQ_x^k=\frac{1}{\delta_x^k}\id_{n_x}$ and $\bQ_y^k=\frac{1}{\delta_y^k}\id_{n_y}$} such that $t^k=\frac{\delta_y^k}{\delta_y^0}$, {$\alpha^k=c_\delta/\delta_y^{k-1}$ for all $k\geq 0$ satisfy the conditions in \eqref{eq:general_step_cond} for $\{\delta_x^k,\delta_y^k,\eta^k\}_{k\geq 0}$ generated using the update rule: $\eta^{k+1}=1/\sqrt{1+\mu_x\delta_x^k}$, $\delta_x^{k+1}=\eta^{k+1}\delta_x^k$, and $\delta_y^{k+1}=\delta_y^k/\eta^{k+1}$} for all $k\geq 0$, for any $\delta_x^0,\delta_y^0>0$ such that 
\ehh{$\bI_{n_x}\sa{\succeq}\delta_x^0(L_{\phi_x}+L_{xx}\bI_{n_x}+ \delta_y^0 L^\top_{yx}L_{yx}/c_\delta)$} for some $c_\delta\in(0,1]$; moreover, the rule implies that ${\cW_k}=\Theta(k^2)$, $\delta_x^k=\Theta(1/k)$, $\delta_y^k=\Theta(k)$, \sa{and $\eta^k\in(0,1)$} for $k\geq 0$.

For convex optimization problem \eqref{eq:central_problem}, 
 $\cH(x,\by)=\sum_{i\in\cN}\fprod{g_i(x),y_i}=\sa{\fprod{g(x),\by}}$ where \sa{$g:x\mapsto [g_i(x)]_{i\in\cN}$}, $y_i$ is the dual variable corresponding to 
 $g_i(x)\in-\cK_i$ and $\by=[y_i]_{i\in\cN}$. \sa{Note $\norm{\grad_x\cH(x,\by)-\grad_x\cH(\bar{x},\by)}\leq L_g\norm{\by}\norm{x-\bar{x}}$ when $\bJ g$ is $L_g$-Lipschitz for some $L_g>0$. 
Hence, if 
$\{\by^k\}$ is a bounded sequence, then $L_{xx}$ exists along the trajectory of APD iterates and one can set $L_{xx}=L_g \sup_k\|\by^k\|$. 
Note that $L_{xx}$ also plays a role in primal step-size selection, and for different $\tau^0$ values, $\{\by^k\}$ will be different; hence, the dual sequence might imply a different $L_{xx}$ along the iterate trajectory. Indeed, due to close loop nature of these relations, proving the boundedness of dual sequence is a delicate issue.}
\subsection{Outline}
\rev{First, in Section~\ref{sec:static}, we develop a decentralized variant of APD\ehh{, DPDA,} for solving \eqref{eq:central_problem} over static communication networks. Each iteration of DPDA only requires one communication round among the nodes. We provide a convergence result in Theorem \ref{thm:static-error-bounds} for static networks.
Next, in Section~\ref{sec:dynamic}, we propose a decentralized algorithm DPDA-TV to solve \eqref{eq:central_problem} when the network topology is time-varying, and we extend our convergence results to time-varying case in Theorem~\ref{thm:dynamic-error-bounds}. {The results in Theorem~\ref{thm:dynamic-error-bounds} cannot be obtained from the results in \cite{hamedani2018primal} since the consensus constraint cannot be simply encoded as a single constraint; 
DPDA-TV is designed to 
\saa{deal with time-varying topology and/or directed edges,} and DPDA-TV can be analyzed as APD with inexact computations.} Finally, in Section~\ref{sec:numerics}, we test the performance of the proposed methods for finding the nearest point to the intersection of ellipsoids.}\vspace*{-2mm}
\section{A method for static networks}
\label{sec:static}
We here discuss how 
APD, stated in \eqref{eq:PDA}, can be implemented to compute an 
\sa{$\epsilon$-optimal} solution to \eqref{eq:central_problem} in a distributed way using only $\cO(1/\sqrt{\epsilon})$ local communications over a \emph{static} communication network $\cG$. 
{Let $\cG=(\cN,\cE)$ denote a \emph{connected} undirected graph of $N$ computing nodes, where $\cN\triangleq\{1,\ldots,N\}$ and $\mathcal{E}\subseteq \mathcal{N}\times \mathcal{N}$ denotes the set of edges -- without loss of generality assume that $(i,j)\in \mathcal{E}$ implies $i< j$. Suppose nodes $i$ and $j$ can exchange information only if $(i,j) \in \cE$. For $i \in \mathcal{N}$, let $\mathcal{N}_i\triangleq\{j\in\mathcal{N} : (i, j) \in \mathcal{E} \text{ or } (j, i) \in \mathcal{E}\}$ denote the set of neighboring nodes and $d_i\triangleq |\mathcal{N}_i|$ is its degree, 
{and $d_{\max}\triangleq \max_{i\in\cN}d_i$}. \sa{In the rest, we use $N=|\cN|$}.}

Let $x_i\in\reals^n$ denote the \emph{local} decision vector of node $i\in\cN$. By taking advantage of the fact that $\cG$ is {\it connected}, we can reformulate \eqref{eq:central_problem} as a 
{\it consensus} optimization problem:
{
\begin{equation}
\label{eq:dist_problem}
\min_{\substack{x_i\in\reals^n,\ i\in\cN}}\Big\{\sum_{i\in \mathcal{N}}{\varphi_i(x_i)}~|
\begin{array}{c}
  x_i=x_j:~\lambda_{ij},\ \forall (i,j)\in \mathcal{E}, \\
  g_i(x_i) \in -\mathcal{K}_i:\ \theta_i,\ \forall{i}\in\mathcal{N}
\end{array}
\Big\},
\end{equation}}%
where $\lambda_{ij}\in\reals^n$ and $\theta_i\in\reals^{m_i}$ are the corresponding dual variables. Let $\bx=[x_i]_{i\in\cN}\in\reals^{n|\cN|}$. The consensus constraints $x_i=x_j$ for $(i,j)\in \mathcal{E}$ can be formulated as $M{\bf x}=0$, where $M\in \mathbb{R}^{n|\mathcal{E}|\times n|\mathcal{N}|}$ is a block matrix such that $M=H\otimes \id_n$ where $H$ is the oriented edge-node incidence matrix, i.e., the entry $H_{(i,j),l}$, corresponding to edge $(i,j)\in \mathcal{E}$ and node $l\in \mathcal{N}$, is equal to $1$ if $l=i$, $-1$ if $l=j$, and $0$ otherwise.
Note that $M^\T M=H^\T H\otimes \id_n=\Omega\otimes \id_n$, where $\Omega\in \mathbb{R}^{|\mathcal{N}|\times |\mathcal{N}|}$ denotes the graph Laplacian of $\cG$, i.e., $\Omega_{ii}=d_i$, $\Omega_{ij}=-1$ if $(i,j)\in\cE$ or $(j,i)\in\cE$, and equal to $0$ otherwise. 
\begin{defn}
A weighted Laplacian matrix $W\in\mathbb{S}_+^{|\cN|}$ is such that $W_{ij}=W_{ji}<0$ for $(i,j)\in\cE$, $W_{ij}=W_{ji}=0$ for $(i,j)\notin \cE$, and $W_{ii}=-\sum_{j\in\cN}W_{ij}$ for $i\in\cN$.
\end{defn}
\begin{remark}
\label{rem:sc}
\saa{While $\bar{f}(x)=\sum_{i\in\cN} f_i(x)$ is strongly convex with modulus $\bar{\mu}>0$, according to Assumption~\ref{assump:sconvex},
$f(\bx)\triangleq\sum_{i\in\cN}f_i(x_i)$ is strongly convex with modulus $\ubar{\mu}>0$ such that
$\ubar{\mu}=\min_{i\in\cN}\mu_i<\bar{\mu}$.}
\end{remark}
We define some notations to facilitate the analysis. 
\begin{defn}
\label{def:problem-components-static}
Let $\cX\triangleq\Pi_{i\in\cN}\reals^n$ and $\cX\ni\bx=[x_i]_{i\in\cN}$; define ${\varphi}:\cX\rightarrow\reals\cup\{\infty\}$ such that $\varphi(\bx)\triangleq\rho(\bx)+f(\bx)$ where $\rho(\bx)\triangleq\sum_{i\in\cN}\rho_i(x_i)$ and $f(\bx)\triangleq\sum_{i\in\cN}f_i(x_i)$.

Let $\cY\triangleq\Pi_{i\in\cN}\reals^{m_i}\times\reals^{n|\cE|}$,  $\cY\ni\by=[\btheta^\top; \blambda^\top]^\top$ such that $\btheta=[\theta_i]_{i\in\cN}\in\reals^m$ and $\blambda=[\lambda_{ij}]_{(i,j)\in\cE}\in\reals^{m_0}$, where $m\triangleq\sum_{i\in\cN}{m_i}$, and {$m_0\triangleq n|\cE|$}. 
Define $h:\cY\rightarrow\reals\cup\{\infty\}$ such that $h(\by)\triangleq \sum_{i\in\cN}\sigma_{-\cK_i}(\theta_i)$ and \saa{let $G:
\cX\rightarrow \reals^m$ be 
as in Assumption~\ref{assump:g_i}}.\vspace*{-2mm}
\end{defn}
\saa{One can reformulate \eqref{eq:dist_problem} as an SP problem:}
{
\begin{align}\label{eq:static-saddle}
\min_{{\bf x}}\max_{\by}\mathcal{L}&({\bf x}, \by)\triangleq 
\langle \blambda, M{\bf x}\rangle +\sum_{i\in\mathcal{N}}\big(\varphi_i(x_i)+\langle \theta_i,g_i(x_i) \rangle -\sigma_{-\cK_i}(\theta_i) \big).
\end{align}}%
Indeed,
one can compute a primal-dual 
solution to \eqref{eq:central_problem} \saa{through computing a saddle point of $\cL:\cX\times\cY\to\reals\cup\{\pm \infty\}$.}
\saa{As discussed 
in Remark~\ref{rem:sc}, $f$ is $\ubar{\mu}$-strongly convex for some $\ubar{\mu}>0$.}
\sa{Therefore, \saa{for any $\mu\in(0, \ubar{\mu}]$}, 
\eqref{eq:static-saddle} is a special case of \eqref{eq:saddle-problem} where \saa{$\psi_x=\rho+\frac{1}{2}{\mu}\|\cdot\|^2$ with $\mu_x={\mu}$, $\phi_x=f-\frac{1}{2}{\mu}\|\cdot\|^2$}, $\psi_y=h$ with $\mu_y=0$, $\phi_y=0$, and $\cH(\bx,\by)=\fprod{\blambda,M\bx}+\fprod{\btheta,G(\bx)}$.}

\rev{According to the discussion in Section \ref{sec:pda}, in its current form \eqref{eq:PDA} prevents us from obtaining a bound on dual iterates, which is important to select a primal step-size sequence.
Indeed, 
in \eqref{eq:PDA}, $\bx^{k+1}$-update requires selecting a suitable step size $\tau^k$ based on $L_{xx}$. We should emphasize that the global Lipschitz constant $L_{xx}$ may lead to a highly conservative step-size choice and/or may not be available; that said, the theoretical guarantees would continue to hold as long as we have a uniform bound on the Lipschitz constants of 
$\grad_x\cH(\cdot,\by^{k+1})$ for all $k\geq 0$. However, arguing the existence of a uniform bound through induction is not clear for the current algorithmic form using the two tools at hand: 
\textbf{(i)} one can get a bound on the Lipschitz constant of $\grad_x\cH(\cdot,\by)$ using $\norm{\by}$; 
\textbf{(ii)} conversely, if one knows a bound on the Lipschitz constant of $\grad_x\cH(\cdot,\by^{k+1})$, one can also obtain a bound on $\norm{\by^{k+1}}$ as a byproduct of the analysis in \cite{hamedani2018primal}. The reason why simple induction technique for constructing a uniform bound on $\{\norm{\by^k}\}_{k\geq 0}$ fails is that 
one needs the Lipschitz constant of $\grad_x\cH(\cdot,\by^{k+1})$ to bound $\norm{\by^{k+1}}$; but, it depends 
on $\by^{k+1}$ at iteration $k$. One way to break this circular argument 
is to perform an $\bx$-update first followed by a $\by$-update. Thus, at iteration $k$, 
one will need a bound on the Lipschitz constant of 
$\grad_x \cH(\cdot,\by^k)$ which can be obtained using $\norm{\by^k}$. Therefore, the induction can be used to get a bound on $\norm{\by^{k+1}}$ in terms of $\by^k$; consequently, one can get a uniform bound on the dual iterate sequence. In this paper, we will use this idea to prove the desired convergence properties of the proposed algorithm for distributed constrained optimization.}

First, multiplying \eqref{eq:saddle-problem} by $-1$ \sa{leads to $\min_{\by}\max_{\bx}\ehh{-\cL(\bx,\by)}$}, one can implement APD in~\eqref{eq:PDA} on this equivalent formulation, which in effect can be obtained by interchanging the roles of $\bx$-variable with $\by$-variable in~\eqref{eq:PDA}. 
The 
resulting algorithm, which can solve \eqref{eq:saddle-problem}, is given below:%
{
\begin{subequations}
\label{eq:PDA-2}
\begin{flalign}
&\bp^k \gets  ~(1+\eta^k)\grad_\bx \cH(\bx^k,\by^k)-\eta^k \grad_\bx \cH(\bx^{k-1},\by^{k-1}), \nonumber \\
&\bx^{k+1}\gets  \argmin_{\bx}  \psi_x(\bx)+\langle\grad \phi_x(\bx^k)+\bp^k,~\bx\rangle +\tfrac{1}{2}\|\bx-{\bx}^k\|_{\bQ_x^k}^2, \label{eq:PDA-x-2} \\
& \ehh{\by^{k+1}\gets\eyy{\argmin_{\by}}~\psi_y(\by)+\fprod{\grad\phi_y(\by^k)-\grad_{\by}\cH(\bx^{k+1},\by^k),\by}+\tfrac{1}{2}\norm{\by-\by^k}^2_{\bQ_y^k},}\label{eq:PDA-y-2}
\end{flalign}
\end{subequations}
}%
where $\{\bQ_x^k,\bQ_y^k,\eta^k\}_k$ are chosen such that \sa{there exist some \sa{$c\geq 1$}, positive sequences $\{t^k\}$ and $\{\alpha^k,\beta^k\}$ satisfying}
{
\begin{subequations}\label{eq:reverse_general_step_cond}
\begin{align}
&\bQ_y^k+\mu_y\id_{n_y}\succeq \bQ_y^{k+1}\sa{/}\eta^{k+1},~ \bQ_x^k+\mu_x\id_{n_x}\succeq \bQ_x^{k+1}\sa{/}\eta^{k+1},\\ &\bQ_y^k-L_{\phi_y}\succeq \ehh{L_{yy}\bI_{n_y}}+\frac{L_{xy}^\top L_{xy}}{\alpha^{k+1}},\\
&\bQ_x^k-L_{\phi_x}\succeq \sa{c}~\eta^{k}(\alpha^{k}+\beta^{k})\id_{n_x}+\ehh{L_{xx}^2\bI_{n_x}}/{\beta^{k+1}}, \label{eq:cond3}
\end{align}
\end{subequations}}%
and $\eta^{k+1}=\frac{t^k}{t^{k+1}}\sa{\in(0,1]}$. \sa{If \eqref{eq:reverse_general_step_cond} holds, then APD in \eqref{eq:PDA-2} converges with a rate result analogous to \eqref{eq:DPA-rate}.}

\saa{Clearly, $L_{\phi_x}=\diag([(L_{f_i}-{\mu})\id_n]_{i\in\cN})
\preceq L_{\max}(f)\id_{nN}$.} 
\sa{Note $\grad_{\bx}\cH(\cdot,\by)$ is $L_{xx}(\btheta)$-Lipschitz where $L_{xx}(\btheta)=\norm{\btheta}\diag([L_{g_i}\id_n]_{i\in\cN})$. Provided that $\|\btheta^k\|\leq B$ for $k\geq 0$ for some $B>0$, we can bound 
\ehh{$L_{xx}(\btheta)\preceq L_{xx}\bI_{nN}$ with $L_{xx}=BL_{\max}(G)$.}
On the other hand, if $G$ is a linear function, then \ehh{$L_{xx}=0$}.} Moreover, $L_{xy}=[\bC~M^\top]$, and since $\|{M M^\top}\|=\|{M^\top M}\|\leq 2d_{\max}$, we get
$L_{xy}^\top L_{xy}\preceq 2\diag([\bC^2~2d_{\max}\id_{m_0}]^\top)$. \sa{Clearly, \ehh{$L_{\phi_y}=\mathbf{0}$ and $L_{yy}=0$}.}

\sa{For $k\geq 0$, given some $\tilde\tau^k,\gamma^k>0$} and $\kappa_i^k>0$ for $i\in\cN$, let
{
\begin{align}
\label{eq:QxQy}
\bQ_x^k=\frac{1}{\tilde{\tau}^k}\id_{nN}, \quad \bQ_y^k=\bD^k_{\kappa,\gamma},
\end{align}}%
where $\bD^k_{\gamma}\triangleq\frac{1}{\gamma^k}\id_{m_0}$, $\bD^k_{\kappa}\triangleq\diag([\frac{1}{\kappa_i^k}\id_{m_i}]_{i\in\cN})$,
and { $\bD^k_{\kappa,\gamma}\triangleq \begin{bmatrix} \bD^k_{\kappa}& \zero \\ \zero & \bD^k_{\gamma} \end{bmatrix}$}.
\sa{Hence, \saa{for any $\mu\in(0, \ubar{\mu}]$}, given the initial iterates $\bx^0$ and $\by^0=[{\btheta^0}^\top; {\blambda^0}^\top]^\top$,
implementing APD in \eqref{eq:PDA-2} on \eqref{eq:static-saddle} leads to the following iterations:}
{
\setlength{\arraycolsep}{0.0em}
\begin{subequations}
\label{eq:pock-pd}
\begin{align}
&{\bf p}^k\gets M^\top((1+\eta^k)\blambda^{k}-\eta^k\blambda^{k-1})+ (1+\eta^k)\bJ G(\bx^k)^\top\btheta^{k} -\eta^k\bJ G(\bx^{k-1})^\top\btheta^{k-1}, \\
&{\bf x}^{k+1}\gets\argmin_{{\bf x}} \saa{\rho(\bx)+\tfrac{1}{2}{\mu}\norm{\bx}^2+\langle \nabla {f(\bx^k)}-{\mu}\bx^k+\bp^k,\bx \rangle} +\tfrac{1}{ 2\tilde\tau^k}\|\bx-\bx^k\|^2, \label{eq:pock-pd-x}\\
&\theta_i^{k+1}\gets\argmin_{\theta_i}\sigma_{-\cK_i}(\theta_i) -\langle g_i(x_i^{k+1}),\theta_i\rangle  +\tfrac{1}{ 2\kappa_i^k}\|\theta_i-\theta_i^k\|^2,   \quad  i\in\cN, \label{eq:pock-pd-theta}\\
&\blambda^{k+1}\gets \argmin_{\blambda} -\langle M\bx^{k+1},\blambda \rangle + \tfrac{1}{ 2\gamma^k}\|\blambda-\blambda^k\|^2. \label{eq:pock-pd-lambda} 
\end{align}
\end{subequations}
\setlength{\arraycolsep}{5pt}
}%
\sa{For $k\geq 0$, defining {$\tau^k\triangleq (\tfrac{1}{\tilde\tau^k}+{\mu})^{-1}$},
\eqref{eq:pock-pd-x} can be equivalently written as follows:}
{
\begin{align}\label{eq:new-x-update}
{\bf x}^{k+1}\gets\argmin_{{\bf x}} \rho(\bx)+\langle \saa{\nabla f(\bx^k)}+\bp^k,\bx \rangle +\tfrac{1}{ 2{\tau}^k}\|\bx-\bx^k\|^2.
\end{align}}%
The following lemma specifies a particular choice of parameters satisfying the step-size conditions \sa{in \eqref{eq:reverse_general_step_cond}.}
\begin{lemma}\label{lem:step-static}
\sa{Let \saa{$\mu\in(0, \ubar{\mu}]$.}
Given $\delta,\gamma^0>0$, let $\{\tilde\tau^k,\gamma^k\}_{k\geq 0}$ and $\{\kappa_i^k\}_{k\geq 0}$ for $i\in\mathcal{N}$ be the step-size sequences for~\eqref{eq:pock-pd}
such that $\eta^0=0$, $\kappa_i^0=\gamma^0\tfrac{\delta}{C_{g_i}^2}$ for $i\in\cN$,}\vspace*{-2mm} 
\sa{
\begin{align*}
\tilde{\tau}^0=\big(L_{\max}\nsa{(f)}+2[
2\gamma^0(2d_{\max}+\delta)+B L_{\max}(G)]\big)^{-1}
\end{align*}}%
for some $B>0$ and generated by using the update rule, 
\sa{
\begin{subequations}\label{eq:stepsize-rule-static}
\begin{align}
&\tau^k\gets(\frac{1}{\tilde{\tau}^k}+\mu)^{-1},\ \gamma^{k+1}\gets\gamma^k\sqrt{1+\mu\tilde{\tau}^k},\  \eta^{k+1}\gets\gamma^k/\gamma^{k+1},\\
&\tilde\tau^{k+1}\gets \tilde\tau^k\eta^{k+1},\ \kappa_i^{k+1}\gets\gamma^{k+1}\tfrac{\delta }{C_{g_i}^2}, i\in\cN.
\end{align}
\end{subequations}}%
for all $k\geq 0$. \sa{If $G$ is affine, then $B L_{\max}(G)$ term disappears and the rule in~\eqref{eq:stepsize-rule-static} satisfies \eqref{eq:reverse_general_step_cond} with $Q_x^k$ and $Q_y^k$ as in \eqref{eq:QxQy} for $k\geq 0$, and 
$\gamma^k=\Theta(k)$ and {$\tilde{\tau}^k=\Theta(1/k)$}. \nsa{For general $G$,} these results hold provided that $\|\btheta^k\|\leq B$ holds for $k\geq 0$.}
\end{lemma}
\begin{proof}
For $k\geq 0$, replacing the Lipschitz parameters in \eqref{eq:reverse_general_step_cond} with their bounds discussed in the paragraph after \eqref{eq:reverse_general_step_cond}, setting \sa{$c=2$, $t^k=\gamma^k/\gamma^0$} and $\beta^k=B L_{\max}(G)$ within \eqref{eq:reverse_general_step_cond} and bounding $\eta^{k}\beta^k$ with $\beta^k$ in~\eqref{eq:cond3} (since $\eta^k\in(0,1]$), we obtain
\sa{sufficient conditions on the step-size sequences of \eqref{eq:pock-pd}}, i.e., 
{
\begin{align*}
&\frac{1}{\tilde\tau^k}+\mu\geq \frac{1}{\tilde\tau^{k+1}\eta^{k+1}},\ \frac{1}{\tilde\tau^k}\geq L_{\max}\nsa{(f)}+\sa{2}\big(
\alpha^{k}\eta^{k}+B L_{\max}(G)\big)\\
&\eta^{k+1}=\frac{\gamma^k}{\gamma^{k+1}},\  \bD_{\kappa,\gamma}^k\succeq \frac{2}{\alpha^{k+1}}\begin{bmatrix}\bC^2 & \mathbf{0}\\ \mathbf{0} & 2d_{\max}\id_{m_0}\end{bmatrix},\ \bD_{\kappa,\gamma}^k\succeq \frac{1}{\eta^{k+1}}\bD_{\kappa,\gamma}^{k+1}.
\end{align*}}%
$\bD_{\kappa,\gamma}^k\succeq \frac{2}{\alpha^{k+1}}\diag([\bC^2 ~2d_{\max}\id_{m_0}]^\top)$ holds if $\bD_{\kappa}^k\succeq \frac{2}{\alpha^{k+1}}\bC^2$ and $\frac{1}{\gamma^k}\geq \frac{4}{\alpha^{k+1}}d_{\max}$ which are satisfied by choosing \sa{$\alpha^{k+1}=2\gamma^{k}(\delta+2d_\max)$}. Therefore, the last three conditions clearly hold for the update rule in~\eqref{eq:stepsize-rule-static}. Next, 
\eqref{eq:stepsize-rule-static} implies $\eta^{k+1}=1/\sqrt{1+\mu\tilde\tau^k}$ and $\tilde\tau^{k+1}=\eta^{k+1}\tilde\tau^k$; hence, ${\tilde\tau^{k+1}\eta^{k+1}}=\frac{\tilde\tau^k}{1+\mu\tilde\tau^k}=(\frac{1}{\tilde\tau^k}+\mu)^{-1}$. Using the fact that $\alpha^k\eta^k\leq 2\gamma^k(\delta+2d_{\max})$, the second condition holds if $\frac{1}{\tilde\tau^k}\geq L_{\max}\nsa{(f)}+2\big( 
2\gamma^k(\delta+2d_{\max})+BL_{\max}(G)\big)$. Now we will prove this inequality by induction on $k$. Clearly, for $k=0$, from the initialization of $\tilde\tau^0$ the condition holds. Suppose that the upper bound on $\frac{1}{\tilde\tau^k}$ holds, then using $\tilde\tau^{k+1}=\tilde\tau^k\eta^{k+1}$ and $\eta^{k}\in(0,1)$ the induction can be proved.
Finally, the proof of convergence rate for $\gamma^k,\tilde\tau^k$ is similar to the proof of \cite[Corollary~1]{chambolle2011first}; hence, the results follow by observing that $\tau^k=\Theta(\tilde\tau^k)$.\vspace*{-3mm}
\end{proof}
{\sa{Suppose $G$ is not affine,} we now show there exists $B>0$ such that 
if $\{\tilde\tau^k,\gamma^k\}_{k\geq 0}$ and $\{\kappa_i^k\}_{k\geq 0}$ for $i\in\mathcal{N}$ are selected as in Lemma~\ref{lem:step-static}, then $\|\btheta^k\|\leq B$ holds for $k\geq 0$. As discussed in Section \ref{sec:pda}, the idea is to use induction. Due to the limited space we briefly give the proof sketch and refer the reader to the proof of Theorem \ref{thm:dynamic-error-bounds} 
which shows the result for the dynamic network case. 
Suppose $\|\btheta^k\|\leq B$, for $k\in\{0,\hdots,K-1\}$ for some $K\geq 1$. Note if $(\bx^*,\btheta^*,\blambda^*)$ is a saddle-point of $\cL$ such that $\blambda^*\neq\mathbf{0}$, then it trivially follows that $(\bx^*,\btheta^*,\mathbf{0})$ is another saddle-point of $\mathcal{L}$. \sa{Starting from some $\bx^0$, $\btheta^0=\mathbf{0}$ and $\blambda^0=\mathbf{0}$,} iterations in \eqref{eq:pock-pd} only use $\{\|\btheta^k\|\}_{k=0}^{K-1}$ to compute $(\bx^K,\btheta^K,\blambda^K)$; and this information is sufficient to show a result analogous to \eqref{eq:DPA-rate}, corresponding to the problem in \eqref{eq:static-saddle}, and is obtained by switching $\bx$ and $\by$ in \eqref{eq:DPA-rate}. Evaluating the resulting inequality at $(\bx,\by)=(\bx^*,\by^*)$, where $\by^*=[{\btheta^*}^\top; \mathbf{0}^\top]^\top$, implies \vspace*{-1mm}
{
\begin{eqnarray*}
{\tfrac{\gamma^K}{2\sa{\gamma^0}}\|{\btheta^*-\btheta^K}\|_{\bD_\kappa^K}^2}\leq \tfrac{t^K}{2}\|\by^*-\by^K\|_{\bQ_y^K}^2\leq \sum_{i\in\mathcal{N}}\Big[\tfrac{1}{ 2\sa{\tilde{\tau}^0}}\|x_i^0-x^*\|^2+\tfrac{1}{\sa{2\kappa_i^0}}\|\theta^*_i\|^2\Big]\triangleq\bar{\Lambda}.
\end{eqnarray*}}%
Therefore, 
$\|\btheta^K\|\leq B$ holds for $B>0$ such that \vspace*{-1mm}
\saa{
\begin{equation}\label{eq:B-bound}
B\geq \norm{\btheta^*}+ \Big(
2A_0 B+A_1\Big)^{\tfrac{1}{2}},
\end{equation}}%
where the constants \saa{ $A_0\triangleq \frac{\gamma^0\delta L_{\max}(G)}{C_{\min}^2}\norm{\bx^*-\bx^0}^2$ and  $A_1\triangleq [{(4\gamma^0(\delta+2d_{\max})+L_{\max}{(f)})} 
\gamma^0\delta \norm{\bx^*-\bx^0}^2+ 
\norm{\btheta^*}_{\bC^2}^2]/C_{\min}^2$}. Since \eqref{eq:B-bound} is a quadratic inequality, there exists $\bar{B}>0$ such that \eqref{eq:B-bound} holds, which implies that $\|\btheta^k\|\leq B$ for $k\geq 0$ hold for any $B\geq \bar{B}$. Therefore, the conclusion of Lemma~\ref{lem:step-static} holds for all $B\geq \bar{B}$.\qed
}
\eh{
\begin{remark}
\label{rem:bound1}
We can explicitly characterize the dual bound $B$ satisfying \eqref{eq:B-bound}. In particular, \eqref{eq:B-bound} holds if $B\geq\bar{B}$ where
\begin{equation}
    \label{eq:bar_B}
    \bar{B}\triangleq A_0+\norm{\btheta^*}+\sqrt{A_0^2+2A_0\norm{\btheta^*}+A_1}.
\end{equation}
This bound depends on some parameters such as $\norm{\btheta^*}$ and $d_\max$. If a Slater point is available, using Lemma 6.1 in \cite{aybat2019distributed} one can compute a bound on $\norm{\btheta^*}$ in a distributed manner. Moreover, one can run a $\max$-consensus algorithm first among the nodes to compute $d_{\max}$ prior to running the algorithm.
\end{remark}}

Since $\cK_i$ is a cone, $\prox{\kappa_i^k\sigma_{-\cK_i}}(\cdot)=\cP_{\cK_i^*}(\cdot)$; hence, 
{
$
\theta_i^{k+1}=\cP_{\cK_i^*}\Big(\theta_i^k+\kappa_i^k g_i(x_i^{k+1})\Big) \nonumber
$
}%
for $i\in\cN$. From \eqref{eq:pock-pd-lambda} one can observe that $\blambda^{k+1}=\blambda^k+\gamma^k M{\bf x}^{k+1}$ which {we can write $\blambda^{k}$} as a partial sum of primal iterates $\{\bx^\ell\}_{\ell=0}^{k-1}$,
i.e., $\blambda^k=\blambda^0+\sum_{\ell=0}^{k-1} \gamma^\ell M{\bf x}^{\ell+1}$. Let $\blambda^0\gets \mathbf{0}$, and define $\{\bs^k\}_{k\geq 0}$ such that $\bs^0=\mathbf{0}$ and $\bs^{k+1}=\bs^k+ \gamma^k\bx^{k+1}$ for $k\geq 0$; hence, $\blambda^k=M\bs^{k}$ for $k\geq 0$. Using 
$M^\top M=\Omega\otimes \id_n$, we obtain
{
$\langle {\bf x},M^\top\blambda^{k} \rangle=~\langle {\bf x},~(\Omega \otimes \id_n)\bs^{k}\rangle
=\sum_{i\in\mathcal{N}}\langle { x_i},$ $~\sum_{j\in\mathcal{N}_i}(s^{k}_i-s^{k}_j)\rangle.$
}%
Thus, the iterations given in~\eqref{eq:pock-pd} for the static graph $\cG$ can be computed in a \emph{decentralized} way, via the node-specific computations as in distributed primal dual algorithm (DPDA) displayed in Fig.~\ref{alg:PDS}.
\begin{figure}
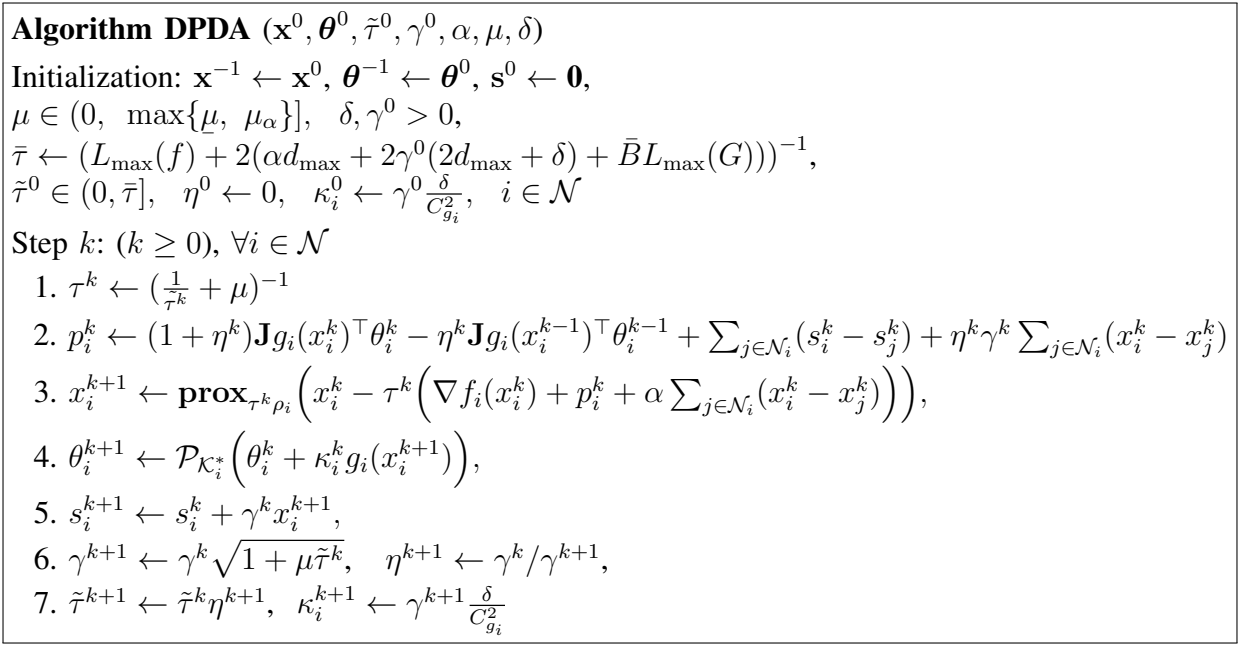

\centering
\framebox{\parbox{0.98\columnwidth}
{
\textbf{Algorithm DPDA} ($\bx^{0},\btheta^0,\gamma^0, 
\mu,\delta, \saa{B}$) \\[1.5mm]
Initialization: $\bx^{-1}\gets \bx^0$, \sa{$\btheta^{-1}\gets \btheta^0$}, $\bs^0\gets \textbf{0}$,\\
$\saa{\mu\in(0,~\ubar{\mu}]},
\ \ \delta ,\gamma^0>0$,\\
\saa{$\tilde{\tau}^0\gets \big(L_{\max}{(f)}+2( 
2\gamma^0(2d_{\max}+\delta)+{B}L_{\max}(G))\big)^{-1}$},\\
$
{\eta^0\gets 0},
\ \ \kappa_i^0\gets\gamma^0\frac{\delta }{C_{g_i}^2},\ \ i\in\cN$\\[1mm]
Step $k$: ($k \geq 0$), $\forall i\in\cN$\\[1mm]
\text{ } 1. $\tau^k\gets(\frac{1}{\tilde{\tau}^k}+\mu)^{-1}$\\[1mm]
\text{ } 2. $p_i^k\gets (1+\eta^k)\bJ \sa{g_i}(x_i^k)^\top \nsa{\theta_i^k}-\eta^k \bJ \sa{g_i}(x_i^{k-1})^\top \nsa{\theta_i^{k-1}}$\\
\text{ } \qquad \quad ~$+\sum_{j\in\cN_i}(s_i^k-s_j^k)+\eta^k\gamma^k\sum_{j\in\cN_i}(x_i^k-x_j^{k})$\\[1mm]
\text{ } 3. $x_i^{k+1}\gets\prox{\tau^k\rho_i}\Big(x_i^k-\tau^k\Big(\nabla f_i(x_i^k)+p_i^k\Big)\Big),$\\[1mm]
\text{ } 4. $\theta_i^{k+1}\gets\cP_{\cK_i^*}\Big(\theta_i^k+\kappa_i^k g_i(x_i^{k+1})\Big),$\\[1mm]
\text{ } 5. $s_i^{k+1}\gets  s_i^k+\gamma^k x_i^{k+1},$ \\[1mm]
\text{ } 6. $\gamma^{k+1}\gets\gamma^k\sqrt{1+\mu\tilde{\tau}^k}$, \ \ $\eta^{k+1}\gets\gamma^k/\gamma^{k+1}$,\\[1mm]
\text{ } 7. $\tilde\tau^{k+1}\gets \tilde\tau^k\eta^{k+1}$,\ \ $\kappa_i^{k+1}\gets\gamma^{k+1}\frac{\delta }{C_{g_i}^2}$}}
\caption{ Distributed Primal Dual Algorithm (DPDA)}
\label{alg:PDS}
\vspace*{-3mm}
\end{figure}

Next, 
we quantify the suboptimality and infeasibility of the DPDA iterate sequence. {The results follows from \eqref{eq:DPA-rate}.} 
\begin{theorem}\label{thm:static-error-bounds}
Suppose Assumption~\ref{assump:saddle-point} holds, 
\saa{let $\mu\in(0,~\ubar{\mu}]$}
and
\saa{$B\geq\min\{0,\bar{B}\}$ where $\bar{B}$ is defined in \eqref{eq:bar_B}.}
Let $\{\bx^k,\btheta^k\}_{k\geq 0}$ be the 
sequence generated by Algorithm DPDA, displayed in Fig.~\ref{alg:PDS}, initialized from an arbitrary $\bx^0$ and $\btheta^0=\mathbf{0}$. 
Then {$\{\bx^k\}_{k\geq 0}$ converges to $\bx^*=\ones\otimes x^*$ such that $x^*$ is the optimal solution to \eqref{eq:central_problem}};
moreover, the following bounds hold for all $K\geq 1$: \sa{$\norm{\bx^K-\bx^*}^2\leq \frac{\tilde{\tau}^K}{\gamma^K}4\gamma^0  \bar{\Lambda}=\cO(1/K^2)$}, \vspace*{-2mm}
{
\begin{align*}
&\max\Big\{|
\saa{\varphi(\bar{\bx}^K)}-\varphi({\bf x}^*)|,\|M\bar{\bf x}^K\|+\sum_{i\in\cN}\norm{\theta_i^*} d_{-\cK_i}(g_i(\bar{x}_i^K))\Big\} \leq \sa{\Lambda_0/ \cW_K=\cO(1/K^2),}
\end{align*}
}%
where $\cW_K=\sum_{k=1}^{K}{\gamma^{k-1}/\gamma^0}$, $\bar{\bx}^{K}=\cW_K^{-1}\sum_{k=1}^{K}{\frac{\gamma^{k-1}}{\gamma^0}}\bx^k$, and \sa{$\Lambda_0\triangleq \tfrac{1}{ 2\gamma^0}+\sum_{i\in\mathcal{N}}\tfrac{1}{ 2\saa{\tilde{\tau}^0}}\|x_i^0-x^*\|^2+\tfrac{2}{ \kappa_i^0}\|\theta^*_i\|^2$}.
\end{theorem}
\begin{remark}
The result in Theorem \ref{thm:static-error-bounds} can be extended to weighted graphs by replacing
the consensus constraint $M\bx=0$ in \eqref{eq:dist_problem} with \sa{$(W^{\tfrac{1}{2}}\otimes\id_n)\bx=0$}, e.g., setting $w_{ij}$ to negative of the effective resistance for the edge $(i,j)$ leads to a better performance on barbel graphs~\cite{Aybat_GSIP,can2019randomized}.
\end{remark}
\begin{remark}\label{rem:linear}
It is important to note that when the constraint functions \ehh{$\{g_i\}_{i\in\cN}$} are \sa{affine}, then $L_{\max}(G)=0$; therefore, the primal step-size \sa{$\tilde{\tau}^0$} is independent of \saa{the bound $\bar{B}$ on $\{\btheta^k\}$, i.e., any positive $\tilde{\tau}^0$ such that $\tilde{\tau}^0\leq (L_{\max}(f)+4\gamma^0d_{\max}+\delta)^{-1}$ is admissible}.\vspace*{-1mm}
\end{remark}
\subsection{Extensions}
{Inspired by Proposition 3.6. in \cite{shi2015extra},
we show in the following lemma
that by suitably regularizing $f$, one can obtain a {\saa{restricted} strongly convex function} when $\ubar{\mu}=0$.}
\begin{lemma}\label{lem:restricted-convex-static}
Under Assumption~\ref{assump:sconvex}, consider $f(\bx)=\sum_{i\in\cN}f_i(x_i)$ 
with $\ubar{\mu}=0$. Given $\alpha>0$, let $f_\alpha(\bx)\triangleq f(\bx)+\alpha~r(\bx)$, where 
$r(\bx)\triangleq \tfrac{1}{2}\norm{\bx}_{W\otimes\id_n}^2$.
Then $\fprod{\nabla f_{\alpha}(\bx)-\nabla f_{\alpha}(\bx^*),~\bx-\bx^*}\geq \mu_\alpha \norm{\bx-\bx^*}$ for any $\bx$, i.e.,
$f_\alpha$ is \saa{restricted} {strongly convex} with modulus ${\mu}_\alpha \triangleq \frac{\bar{\mu}/N~+\alpha {\lambda}_2}{2}-\sqrt{\big(\frac{\bar{\mu}/N~-\alpha {\lambda}_2}{2}\big)^2+4\bar{L}^2}>0$ for $\alpha>\frac{4 N}{ {\lambda}_2 }\bar{L}^2/\bar{\mu}$, where $\bar{L}=\sqrt{{\sum_{i\in\cN}L_{f_i}^2}/{N}}$ and {${\lambda}_2=\lambda^+_{\min}(W)$}.
\end{lemma}
\begin{remark}
\label{rem:modulus}
The condition {$\alpha >\frac{4}{ \bar{\mu}\lambda^+_{\min}(W)}\sum_{i\in\cN}L_{f_i}^2$} is similar to the one in~\cite{shi2015extra}, where $\alpha$ should be greater than {$\frac{N \sa{L_{\max}^2(f)}}{2\bar{\mu}\lambda^+_{\min}({W})}$ for some ${W}\in\sy^{N}_+$ which is a parameter for their algorithm satisfying certain conditions.
}
\end{remark}

\section{A 
method for time-varying networks}
\label{sec:dynamic}
In this section we develop a distributed primal-dual algorithm for solving \eqref{eq:central_problem} when the communication network topology is \emph{time-varying}. We will adopt the following definition and assumption for the \emph{time-varying} network model.

\begin{defn}
\label{def:neighbors}
Given $t\geq 0$, for an undirected graph $\cG^t=(\cN,\cE^t)$, let $\cN_i^t\triangleq\{j\in\cN:\ (i,j)\in\cE^t~\hbox{ or }~(j,i)\in\cE^t\}$ denote the set of neighboring nodes of $i \in \mathcal{N}$, and $d_i^t\triangleq |\mathcal{N}_i^t|$ represent the degree of node $i\in \mathcal{N}$ at time $t$; for a directed graph $\cG^t=(\cN,\cE^t)$, let $\cN^{\,t,{\rm in}}_i\triangleq\{j\in\cN:\ (j,i)\in\cE^t\}\cup\{i\}$ and $\cN^{\,t,{\rm out}}_i\triangleq\{j\in\cN:\ (i,j)\in\cE^t\}\cup\{i\}$ denote the in-neighbors and out-neighbors of node $i$ at time $t$, respectively; and {$d_i^t\triangleq |\cN^{\,t,{\rm out}}_i|-1$} be the out-degree of node $i$.
\end{defn}
\begin{assumption}
\label{assump:communication_general}
Let $\{\cG^t\}_{t\in\reals_+}$ be a collection of either all directed or all undirected graphs. When $\cG^t$ is an undirected graph, 
$i\in\cN$ can send and receive data to and from $j\in\cN$ at time $t$ only if $j\in\cN_i^t$, i.e., $(i,j) \in \cE^t$ or $(j,i) \in \cE^t$; on the other hand, when $\cG^t$ is a directed graph, node $i\in\cN$ can receive data from $j\in\cN$ only if $j\in\cN_i^{\,t,{\rm in}}$, i.e., $(j,i)\in\cE^t$, and can send data to $j\in\cN$ only if $j\in\cN_i^{\,t,{\rm out}}$, i.e., $(i,j)\in\cE^t$.
\end{assumption}

We assume a \emph{compact} domain, i.e., let \nsa{${\Delta_i}\triangleq\max_{x_i\in \dom {\varphi_i}}\|x_i\|$} and ${\Delta}\triangleq\max_{i\in\mathcal{N}}\Delta_i<\infty$.  Let $\cB_0\triangleq\{x\in\reals^n:\ \norm{x}\leq {2\Delta}\}$ and $\mathcal{B}\triangleq\Pi_{i\in\cN}\cB_0\subset \cX\triangleq\Pi_{i\in\cN}\reals^n$; 
and let $\Ct\triangleq \cC\cap\cB$ be a set of bounded consensus decisions, {where $\cC$ is the consensus cone defined below:}
{
\begin{align}
\label{eq:consensus_set}
\cC \triangleq\{{\bf x}\in\cX
:\ \exists\bar{x}\in\mathbb{R}^n\ \st\ x_i=\bar{x}\quad \forall\ i\in\mathcal{N} \}.
\end{align}}%
\begin{defn}
\label{def:bregman}
Let $\cX\triangleq\Pi_{i\in\cN}\reals^n$ and $\cX\ni\bx=[x_i]_{i\in\cN}$; $\cY\triangleq\Pi_{i\in\cN}\reals^{m_i}\times\reals^{n|\cN|}$,  $\cY\ni\by=\eyh{[\btheta^\top; \blambda^\top]}^\top$ such that $\btheta=[\theta_i]_{i\in\cN}\in\reals^m$ and $\blambda\in\reals^{m_0}$, where $m\triangleq\sum_{i\in\cN}{m_i}$, and {$m_0\triangleq n|\cN|$}. Given 
$\gamma^k>0$ and $\kappa_i^k>0$ for $i\in\cN$,
\eyh{let $\bD^k_{\gamma}\triangleq\frac{1}{\gamma^k}\id_{m_0}$, $\bD^k_{\kappa}\triangleq\diag([\frac{1}{\kappa_i^k}\id_{m_i}]_{i\in\cN})$,
and { $\bD^k_{\kappa,\gamma}\triangleq \begin{bmatrix} \bD^k_{\kappa}& \zero \\ \zero & \bD^k_{\gamma} \end{bmatrix}$}.}
\end{defn}
\begin{defn}
\label{def:problem-components}
\eyh{Let ${\varphi}:\cX\rightarrow\reals\cup\{\infty\}$ be such that $\varphi(\bx)\triangleq\rho(\bx)+f(\bx)$ where $\rho(\bx)\triangleq\sum_{i\in\cN}\rho_i(x_i)$ and $f(\bx)\triangleq\sum_{i\in\cN}f_i(x_i)$.}
Let $h:\cY\rightarrow\reals\cup\{\infty\}$ such that $h(\by)\triangleq\sigma_{{\Ct}}(\blambda)+\sum_{i\in\cN}\sigma_{-\cK_i}(\theta_i)$, and \saa{let $G:
\cX\rightarrow \reals^m$ be 
as in Assumption~\ref{assump:g_i}}.
\end{defn}
The problem in \eqref{eq:central_problem} can be equivalently written as
{
\begin{align}
\label{eq:equiv-formulation}
\min_{\bx\in \Ct}\{
\varphi(\bx):\
g_i(x_i)\in-\cK_i,~i\in\cN\},
\end{align}}%
and \saa{\eqref{eq:equiv-formulation} can be reformulated} as an SP problem:
{
\begin{align}\label{eq:dynamic-saddle}
\min_{\bx} \max_{\blambda,\btheta}~ &\cL(\bx,\by)\triangleq 
\fprod{\blambda,~\bx}-\sigma_{\Ct}(\blambda) +\sum_{i\in\mathcal{N}}(\saa{\varphi_i(x_i)}+\langle \theta_i,g_i(x_i)\rangle-\sigma_{-\cK_i}(\theta_i)),
\end{align}}%
where 
$\btheta=[\theta_i]_{i\in\cN}$ and $\blambda\in\reals^{n|\cN|}$.
Therefore, 
one can compute a primal-dual optimal solution to \eqref{eq:central_problem} through {computing a saddle-point to \eqref{eq:dynamic-saddle}.}

\saa{As discussed 
in Remark~\ref{rem:sc}, $f$ is $\ubar{\mu}$-strongly convex for some $\ubar{\mu}>0$.}
\sa{Therefore, \saa{for any $\mu\in(0, \ubar{\mu}]$},}
\eqref{eq:dynamic-saddle} is a special case of \eqref{eq:saddle-problem} where $\psi_x=\rho+\frac{\mu}{2}\|\cdot\|^2$ \saa{with $\mu_x=\mu$}, $\phi_x=f-\frac{\mu}{2}\|\cdot\|^2$, and $\psi_y=h$ \saa{with $\mu_y=0$}, $\phi_y=0$, and $\cH(\bx,\by)=\fprod{\blambda,\bx}+\fprod{\btheta,G(\bx)}$.

In the rest, 
we consider an implementation of APD in~\eqref{eq:PDA-2} for solving \eqref{eq:dynamic-saddle}. This does not immediately result in a decentralized method; thus, we subsequently discuss how to 
modify it so that it works in a distributed fashion over time-varying communication networks.

Given the initial iterates $\bx^0$, $\by^0=[{\btheta^0}^\top; {\bmu^0}^\top]^\top$, and $\blambda^0=\bmu^0$, and step-size sequence $\{\tilde\tau^k,[\kappa_i^k]_{i\in\cN},\gamma^k\}_k$, \nsa{implementing APD in \eqref{eq:PDA-2} on \eqref{eq:dynamic-saddle}}
using $\bQ_x^k=\frac{1}{\tilde{\tau}^k}\id_{nN}$ and $\bQ_y^k=\bD^k_{\kappa,\gamma}$ leads to 
{
\begin{subequations}\label{eq:pock-pd-2}
\begin{align}
&{\bf p}^k\gets (1+\eta^k)\bmu^{k}-\eta^k\bmu^{k-1} +(1+\eta^k)\bJ G(\bx^k)^\top\btheta^{k}-\eta^k\bJ G(\bx^{k-1})^\top\btheta^{k-1}, \\
&\eyh{{\bf x}^{k+1}\gets}\argmin_{{\bf x}} \rho(\bx)+\tfrac{\mu}{2}\norm{\bx}^2+\langle \nabla \saa{f(\bx^k)}-\mu\bx^k+\bp^k,\bx \rangle +\tfrac{1}{ 2\tilde\tau^k}\|\bx-\bx^k\|^2, \label{eq:pock-pd-3-x}\\
&\theta_i^{k+1}\gets\argmin_{\theta_i} \sigma_{-\cK_i}(\theta_i) -\langle g_i({x_i^{k+1}}),~\theta_i\rangle  +\tfrac{1}{ 2\kappa_i^k}\|\theta_i-\theta_i^k\|^2,\quad i\in\cN, \label{eq:pock-pd-2-theta} \\
&\blambda^{k+1}\gets\argmin_{\blambda} \sigma_{\Ct}(\blambda)-\langle {{\bx}^{k+1}},\blambda \rangle +\tfrac{1}{ 2\gamma^k}\|\blambda-\nsa{\bmu^k}\|^2, \label{eq:pock-pd-2-lambda}\\
&\bmu^{k+1}\gets \blambda^{k+1},  \label{eq:pock-pd-2-nu}
\end{align}
\end{subequations}
}%
where $\bx^{-1}=\bx^0$, \sa{$\btheta^{-1}=\btheta^0$}, $\tau^k=(\tfrac{1}{\tilde\tau^k}+\mu)^{-1}$, \eyh{and
\eqref{eq:pock-pd-3-x} can be equivalently written as follows:
{
\begin{equation}
    \bx^{k+1}\gets\argmin_{{\bx}} \rho(\bx)+\langle \saa{\nabla f({\bx^k})}+\bp^k,~\bx \rangle   +\tfrac{1}{ 2\tau^k}\|\bx-\nsa{\bx^k}\|^2. \label{eq:pock-pd-2-x}
\end{equation}}}%
A possible choice for 
the positive parameter sequences $\{\gamma^k\}_k$, $\{\tau^k\}_k$, $\{\kappa_i^k\}_k$ for $i\in\cN$ is given in Figure~\ref{alg:PDD}.

For $k\geq 0$, using extended \rev{Moreau decomposition for proximal operators \cite{rockafellar2015convex}}, $\blambda^{k+1}$ in \eqref{eq:pock-pd-2-lambda} can be computed as
\saa{$\blambda^{k+1} =\prox{\gamma^k\sigma_{\Ct}}{(\bmu^k+\gamma^k {\bx}^{k+1})}=\gamma^k\left(\bom^k-\mathcal{P}_{\Ct}(\bom^k)\right)$,
where $\bom^k\triangleq\frac{1}{\gamma^k}{\bmu^k+{\bx}^{k+1}}$ for {$k\geq 0$}}.
For any $\bx
\in\cX$, $\mathcal{P}_{\Ct}(\bx)$ 
can be computed as $\mathcal{P}_{\Ct}(\bx)
=\mathcal{P}_{\mathcal{B}}(\cP_\cC(\bx))$, and $\cP_\cC(\bx)=\one\otimes \nsa{p(\bx)}$,
where $\nsa{p(\bx)}\triangleq\frac{1}{ |\mathcal{N}|}\sum_{i\in\mathcal{N}}x_i$, $\cP_{\cB}(\bx)=[\cP_{\cB_0}(x_i)]_{i\in\cN}$ and $\cP_{\cB_0}(x_i)=x_i\min\{1,\frac{2\Delta}{\norm{x_i}}\}$ for $i\in\cN$.

\saa{Although $\bx$-step in \eqref{eq:pock-pd-2-x}} and $\btheta$-step in \eqref{eq:pock-pd-2} can be computed locally at each node, computing
$\blambda$-step requires communication among the nodes to evaluate $\cP_{\Ct}(\bom^k)$. 
Indeed, evaluating the average operator {$p(\cdot)$} is \emph{not} a simple operation in a decentralized computational setting which only allows for communication among the neighbors. 
To overcome this issue, we will approximate the average operator {$p(\cdot)$} using multi-communication rounds, and analyze the resulting iterations as an \emph{inexact} primal-dual algorithm.

We define a \emph{communication round} at time $t$ as an operation over $\cG^t$ such that every node simultaneously sends and receives data to and from its neighboring nodes according to Assumption~\ref{assump:communication_general} -- the details of this operation will be discussed shortly. 
We assume that communication among neighbors occurs \emph{instantaneously}, and nodes operate \emph{synchronously}; and we further assume that for each iteration $k\geq 0$, there exists an approximate averaging operator $\cR^k(\cdot)$ 
which can be computed in a decentralized fashion and it approximates $\mathcal{P}_{\cC}(\cdot)$ 
with decreasing \emph{approximation error} as $k$, the number of iterations, increases. This \emph{inexact} version of APD using approximate averaging operator $\cR^k(\cdot)$ and running on time-varying communication network $\{\cG^t\}_{t\in\reals_+}$ will be called DPDA-TV. 
\begin{assumption}
\label{assump:approximate-average}
Given a time-varying network $\{\cG^t\}_{t\in\reals_+}$ such that $\cG^t=(\cN,\cE^t)$ for $t\geq 0$. Suppose that there is a global clock known to all $i\in\cN$. Assume that the local operations required to compute \rev{$\cP_{\cK_i^*}$} as in \eqref{eq:pock-pd-2-theta}, and $\prox{\rho_i}$ and $\grad f_i$ as in \eqref{eq:pock-pd-2-x} can be completed between two ticks of the clock for all $i\in\cN$ and $k\geq 0$; and  every time the clock ticks a communication round with instantaneous messaging between neighboring nodes takes place subject to Assumption~\ref{assump:communication_general}. \rev{Let $q_k\in\integers_+$ denotes the number of communication rounds at iteration $k\geq 0$.} Suppose that for each $k\geq 0$ there exists $\cR^k(\cdot)=[\cR_i^k(\cdot)]_{i\in\cN}$ such that $\cR_i^k(\cdot)$ can be computed with local information available to node $i\in\cN$, and decentralized computation of $\cR^k$ requires $q_k$ communication rounds. Furthermore, we assume that there exist $\Gamma>0$ and ${\beta}\in (0,1)$ such that for all $k\geq 0$,
{
\begin{align}
\label{eq:approx_error-for-full-vector-x}
{ \|\cR^k(\nsa{\bom})-\mathcal{P}_{\cC}(\nsa{\bom})\| \leq N~\Gamma \beta^{q_k}\norm{\nsa{\bom}}},\quad\forall~\nsa{\bom\in\cX}. 
\end{align}}%
\end{assumption}
\vspace*{-3mm}
\begin{figure}[h!]
\centering
\framebox{\parbox{0.98\columnwidth}{
{
\textbf{Algorithm DPDA-TV} \rev{( $\bx^{0},\btheta^0,\gamma^0,
\mu,\delta, {B},\{q_k\}$)} \\[0.5mm]
Initialization: $\bx^{-1}\gets\bx^0,\ \ \nsa{\btheta^{-1}\gets \btheta^0},\ \ \nsa{\bmu^{-1}\gets\mathbf{0}},\ \ \bmu^0\gets\mathbf{0}$,\\
\rev{$
\saa{\mu\in(0,~\ubar{\mu}]},\ \ \delta ,\gamma^0>0$,}\\
\saa{$\tilde{\tau}^0\gets  \big(L_{\max}{(f)} 
+2\gamma^0(1+\delta)+2B L_{\max}(G)\big)^{-1}$},\\[1mm]
$
\eta^0\gets 0,\ \ \kappa_i^0\gets\gamma^0\frac{\delta}{C_{g_i}^2}\ \ i\in\cN$\\[1mm]
Step $k$: ($k \geq 0$), $\forall i\in\cN$\\
\text{ } 1. $\tau^k\gets (\frac{1}{\ttau^{k}}+\mu)^{-1}$, \\[1mm]
\text{ } 2. $p_i^k\gets (1+\eta^k)\bJ g_i(x_i^k)^\top\theta^{k}_i-\eta^k\bJ g_i(x_i^{k-1})^\top\theta_i^{k-1}$ \\[1mm]
\text{ } \quad\qquad~ $+(1+\eta^k)\nu_i^{k}-\eta^k\nu_i^{k-1}$,\\[1mm]
\text{ } 3. 
$\saa{x_i^{k+1}\gets\prox{\tau^k\rho_i}\big[x_i^k-\tau^k\big(\nabla f_i(x_i^k)+p_i^k 
\big)\big]},$\\[1mm]
\text{ } 4. $\theta_i^{k+1}\gets\cP_{\cK_i^*}\Big(\theta_i^k+\kappa_i^kg_i(x_i^{k+1})\Big),$\\[1mm]
\text{ } 5. $\omega^k_i\gets\tfrac{1}{\gamma^k}{\nu}_i^k+x_i^{k+1},$\\[1mm]
\text{ } 6. \rev{Perform $q_k$ communication rounds to compute $\cR^k_i(\bom^k)$}, \\[1mm]
\text{ } \quad ~${\nu}_i^{k+1}\gets\gamma^k \Big(\omega_i^k - \cP_{\cB_0}\Big(\cR_i^k\big(\bom^k\big)\Big)\Big),$ \\[1mm]
\text{ } 7. $\gamma^{k+1}\gets\gamma^k\sqrt{1+\mu\tilde{\tau}^k}$, \ \ $\eta^{k+1}\gets\gamma^k/\gamma^{k+1}$,\\[1mm]
\text{ } 8. $\kappa_i^{k+1}\gets\gamma^{k+1}\frac{\delta}{C_{g_i}^2}$, \ \ $\tilde\tau^{k+1}\gets\tilde\tau^k\eta^{k+1}$
}}}
\caption{ Distributed Primal-Dual Alg. for Time-Varying $\{\cG^t\}_{t\geq 0}$ (DPDA-TV)}
\label{alg:PDD}
\vspace*{-1mm}
\end{figure}
{Now we briefly talk about such operators. Let $V^t\in\reals^{|\cN|\times|\cN|}$ be a matrix encoding the topology of $\cG^t=(\cN,\cE^t)$ in some way for $t\in\integers_+$. We define $W^{t,s}\triangleq V^tV^{t-1}...V^{s+1}$ for any $t,s\in\integers_+$ such that $t\geq s+1$.  For directed time-varying graph ${\cG^t}$, set $V^t\in\reals^{|\cN|\times |\cN|}$ {as follows: for each $i\in\cN$,}}
\begin{align}
\label{eq:directed-weights}
V^t_{ij}=\tfrac{1}{ d_j^t\nsa{+1}}\ \hbox{ if }\ j\in\cN^{\,t,{\rm in}}_i;\ \ V^t_{ij}=0\ \hbox{ if }\ j\not\in\cN^{\,t,{\rm in}}_i.
\end{align}
Let $\mathcal{C}_k\in\integers_{+}$ be the total number of \emph{communication rounds} 
before the $k$-th iteration of {DPDA-TV}, and let $q_k\in\integers_{+}$ be the number of communication rounds to be performed within the $k$-th iteration while evaluating $\cR^k$. For $\bom\in\cX$, 
define
{
\begin{equation}
\label{eq:approx-average-dual-directed}
\cR^k(\bom)\triangleq{\diag(W^{\mathcal{C}_k+q_k,\mathcal{C}_k}\ones_{|\cN|})^{-1}}(W^{\mathcal{C}_k+q_k,\mathcal{C}_k}\otimes\id_n)~\bom
\end{equation}
}%
to approximate $\mathcal{P}_{\cC}(\cdot)$. Note that $\mathcal{R}^k(\cdot)$ can be computed in a \emph{distributed fashion} requiring $q_k$ communication rounds -- $\cR^k$ is nothing but the push-sum protocol~\cite{kempe2003gossip}. Assuming that the digraph sequence $\{\cG^t\}_{t\in\integers_+}$ is uniformly strongly connected ($M$-strongly connected), it follows from \cite{nedic2015distributed} that $\cR^k$ satisfies Assumption~\ref{assump:approximate-average}. \nsa{When $\{\cG^t\}_{t\in\integers_+}$ is undirected, suppose the edge union of every $M$ consecutive graph is connected.
Under this assumption, if one chooses $\{V^t\}_{t\in\integers_+}$ 
such that $\zeta\leq V_{ij}^t$ for all $(i,j)\in\cE$ for some $\zeta>0$, 
then 
\begin{align}
\label{eq:approx-average-dual-undirected}
\cR^k(\bom)\triangleq (W^{\mathcal{C}_k+q_k,\mathcal{C}_k}\otimes\id_m)\bom
\end{align}
satisfies Assumption~\ref{assump:approximate-average} with
$\Gamma=\Theta(\tfrac{1}{N})$ and $\beta=(1-\tfrac{\zeta}{2N^2})^{\frac{1}{2M}}$ -- for details see~\cite{nedic2009distributedquant}.}

Note that for $\tcR^k(\cdot)\triangleq\cP_{\cB}(\cR^k(\cdot))$, we have $\tcR^k(\bw)\in\cB$, and $\|\tcR^k(\bom)-\mathcal{P}_{\Ct}(\bom)\|\leq N~\Gamma \beta^{q_k}\norm{\bom}$ for $\bom\in\cX$ 
due to non-expansivity of $\cP_\cB$.
Consider the $k$-th iteration of the algorithm in \eqref{eq:pock-pd-2}. {Instead of computing $\blambda^{k+1}$ 
as shown in \eqref{eq:pock-pd-2-lambda}, 
which require computing $\cP_\cC$, we propose
using the inexact averaging operator $\cR^k$ to approximate $\cP_\cC$.} Hence,
we obtain an \emph{inexact} variant of \eqref{eq:pock-pd-2} by replacing \eqref{eq:pock-pd-2-nu} 
with
{
\begin{equation}
\bmu^{k+1} \gets \gamma^k\left(\bom^k-\cP_\cB\big(\cR^k(\bom^k)\big)\right), \label{eq:inexact-rule-mu}
\end{equation}}%
\saa{and \eqref{eq:pock-pd-2-lambda} is never computed in the implementation.}

Thus, the updates 
in~\eqref{eq:pock-pd-2} can be computed inexactly, and in a \emph{decentralized} way for any {time-varying} connectivity network $\{\cG^t\}_{t\in\integers_+}$, via the node-specific computations as in the 
distributed primal-dual algorithm displayed in Fig.~\ref{alg:PDD}. Indeed, the iterate sequence $\{\bx^k,\bmu^k,\btheta^k\}_{k\geq 0}$ generated by~DPDA-TV displayed in Fig.~\ref{alg:PDD} is the same sequence generated by the recursion in \eqref{eq:pock-pd-2-theta}, \eqref{eq:pock-pd-2-x} and \eqref{eq:inexact-rule-mu}. 
\sa{However, due to inexact computations, the result in \eqref{eq:DPA-rate} is not applicable anymore.} 

Next, 
we quantify the suboptimality and infeasibility of the DPDA-TV iterate sequence.
\begin{theorem}\label{thm:dynamic-error-bounds}
{Suppose Assumptions 
\nsa{\ref{assump:lipschitz_f}-\ref{assump:sconvex}}, \ref{assump:communication_general} and \ref{assump:approximate-average} hold.}
Starting from $(\btheta^0,\bmu^0)=(\mathbf{0},\mathbf{0})$ and 
arbitrary $\bx^0$, let $\{\bx^k,\btheta^k,\bmu^k\}_{k\geq 0}$ be the iterate sequence generated by Algorithm DPDA-TV 
in Fig.~\ref{alg:PDD} 
\nsa{for $\{q_k\}_{k\geq 0}$ such that $
\sum_{k=1}^{\infty} \beta^{q_{k-1}}k^4<\infty$ and $B>0$ chosen sufficiently large -- if $\{g_i\}_{i\in\cN}$ are affine, one can set $B=0$.} Then $\{\bx^k\}_{k\geq 0}$ converges to $\bx^*=\ones\otimes x^*$ such that $x^*$ is the optimal solution to \eqref{eq:central_problem}.
Moreover, the following bounds hold for all $K\geq 1$:
{
\begin{subequations}\label{eq:rate_result-d}
\begin{align}
&\max\big\{|\varphi(\bar{\bx}^K)-\varphi({\bf x}^*)|,~d_{\cC}(\bar{\bx}^K)+\sum_{i\in\cN}\norm{\theta_i^*} d_{-\cK_i}(g_i(\bar{x}_i^K))\big\} \leq \frac{\nsa{\Lambda(K)}}{\cW_K}=\cO\big(1/K^2\big),\label{eq:rate_result-d-a}\\
&\|\bx^K-\bx^*\|^2\leq \tfrac{\tilde\tau^K}{\gamma^K}4\gamma^0\nsa{\Lambda(K)}=\cO\big(1/K^2\big), \label{eq:rate_result-d-b}
\end{align}
\end{subequations}
}%
{and the parameters satisfy $\max\{\cW_K,~\gamma^K/\tilde{\tau}^K\}=\cO(K^2)$, where $\cW_K=\sum_{k=1}^{K}{\gamma^{k-1}/\gamma^0}$, $\bar{\bx}^{K}=\cW_K^{-1}\sum_{k=1}^{K}\frac{\gamma^{k-1}}{\gamma^0}\bx^k$, $\nsa{\Lambda(K)}=\cO\big(\sum_{k=1}^K\beta^{q_{k-1}} k^4\big)$; hence, $\sup_{K\in\integers_+}\nsa{\Lambda(K)}<\infty$.}
\end{theorem}
The following lemma is {a slight} extension of Proposition 3 in~\cite{chen2012fast}, where it is stated for $p=1$; its proof is omitted.
\begin{lemma}\label{lemsum}
Let $\beta\in(0,1)$, $p\geq 1$ is a rational number, and $d\in\integers_+$. Define $P(k,d)=\{\sum_{i=0}^d c_i k^i:\ c_i\in\reals\ i=1,\ldots,d \}$
denote the set of polynomials of $k$ with degree at most $d$. Let {$r^{(k)}\in P(k,d)$} for $k\geq 1$, then
$\sum_{k=0}^{+\infty}{r^{(k)}}\beta^{\sqrt[\leftroot{-3}\uproot{3}p]{k}}$ is finite.
\end{lemma}
\begin{remark}
When the constraint functions \ehh{$\{g_i\}_{i\in\cN}$} are affine, the primal step-size \saa{$\tilde\tau^0$} is independent of $B$ and one can take $B=0$. In general, the results hold for any $B>0$ sufficiently large \nsa{-- for details see the proof of Theorem~\ref{thm:dynamic-error-bounds}.}
\end{remark}
{\begin{remark}
It is worth mentioning that the summability condition in Theorem \ref{thm:dynamic-error-bounds} can be reduced to $\sum_{k=1}^{\infty} \beta^{q_{k-1}}k^3$ (at the cost of larger $\cO(1)$ constant in rate result) with a tighter analysis exploiting the boundedness of $\{\bmu^k\}_k$ similar to the Theorem 3.1 in \cite{aybat2019distributed}. Note that the condition in that paper is $\sum_{k=1}^{\infty} \beta^{q_{k-1}}k<+\infty$ and the stronger condition here is due to the fact that the accelerated methods are more fragile to inexact computation errors -- see \cite{schmidt2011convergence}.
\end{remark}}
\begin{remark}
Note that, {at the $K$-th iteration}, the suboptimality, infeasibility and consensus violation are $\cO\left(\tfrac{1}{\cW_K}~\nsa{\Lambda(K)}\right)$ in the ergodic sense, and the distance of iterates to $\bx^*$ is $\cO\left(\tfrac{\tilde{\tau}^K}{\gamma^K}~\saa{\Lambda(K)}\right)$ where \nsa{$\Lambda(K)$} denotes the error accumulations due to average approximation.
Moreover, \nsa{$\Lambda(K)$} can be bounded above for all $K\geq 1$
\saa{since $\Lambda(K)\leq \cO(\sum_{k=1}^{\infty} \beta^{q_{k-1}}k^4)<\infty$. Indeed,}
for any $c>0$, {choosing $q_k=(5+c)\log_{{1}/{\varsigma}}(k+1)$ for some $(0,1)\ni\varsigma\geq \beta$, ensures that 
\saa{$\sum_{k=1}^{\infty} \beta^{q_{k-1}}k^4\leq 1+\tfrac{1}{c}$,}
and the total number of communication rounds right before the $K$-th iteration is equal to $\mathcal{C}_K\triangleq\sum_{k=0}^{K-1}q_k\leq(5+c)K\log_{{1}/{\varsigma}}(K)$.}
\end{remark}
\begin{remark}
{Choosing the number of communications as $q_k=\cO(\log_{1/\varsigma}(k))$, for some $(0,1)\ni\varsigma\geq \beta$, requires the knowledge of global information $\beta$. In the case where such global information is not available, one can let $q_k=(k+1)^{1/p}$, for any $p\geq 1$. Using Lemma~\ref{lemsum}, we have $\nsa{\Lambda(K)}<+\infty$. This choice of $q_k$ leads to $\mathcal{C}_K=\sum_{k=0}^{K-1}q_k= \cO(K^{1+1/p})$.
On other hand, a practical way to estimate $\beta\in (0,1)$ is to run an average consensus iterations with a random initialization until iterates stagnate around the average; this leads to a rate coefficient $\beta_i$ for $i\in\cN$. Next, nodes can do a max consensus to compute $\bar{\beta}=\max_{i\in\cN}\beta_i$ and use it to set $q_k=(5+c)\log_{{1}/{\bar{\beta}}}(k+1)$.}
\end{remark}
\begin{remark}
\label{rem:N-complexity}
\saa{At the end of the proof of Theorem \ref{thm:dynamic-error-bounds}, we show that $\Lambda(K)=\cO(\tfrac{N}{\gamma^0}+NB+(\gamma^0+1)N^2\Gamma)$; moreover, from Lemma \ref{lem:step-size}, we also get $\cW_K=\Theta(\mu\tilde \tau^0 K^2)$. Note that for an undirected time-varying graph $\{\cG_t\}$, $\Gamma={\Theta(1/N)}$ and $\log(1/\varsigma)=\Omega(1/N^3)$ --see~\cite{nedic2009distributedquant}. Hence, selecting
$\gamma^0=1/\sqrt{N}$ and $q_k=(5+c)\log_{{1}/{\varsigma}}(k+1)$ for some $(0,1)\ni\varsigma\geq \beta$, implies that the iteration complexity of DPDA-TV is $\cO(N^2/K^2)$ and drops to \ehh{$\cO(N/K^2)$} when \ehh{$\{g_i\}_{i\in\cN}$} are affine functions. Furthermore, the total number of communications to achieve $\epsilon$-suboptimality/infeasibility is $\tilde\cO(N^4/\sqrt{\epsilon})$, and drops to $\tilde\cO(N^{3.5}/\sqrt{\epsilon})$ when \ehh{$\{g_i\}_{i\in\cN}$} are affine functions.}
One can conclude that the result in \cite{aybat2016primal} implies $\tilde{\cO}(N^4/\epsilon)$ number of communications to achieve $\epsilon$-suboptimality/infeasibility when \ehh{$\{g_i\}_{i\in\cN}$} are affine functions while in this paper we derived the improved rate of $\tilde{\cO}({N^{3.5}}/\sqrt{\epsilon})$ for strongly convex objective.
\end{remark}
\subsection{Extensions}
Next lemma 
is an extension of Proposition 3.6. in \cite{shi2015extra}, where we show
by suitably regularizing $f$, one can obtain a \ehh{restricted strongly convex function} when $\ubar{\mu}=0$.
\begin{lemma}\label{lem:restricted-convex-general}
Under Assumption~\ref{assump:sconvex}, consider $f(\bx)=\sum_{i\in\cN}f_i(x_i)$ 
with $\ubar{\mu}=0$. Given $\alpha>0$, let $f_\alpha(\bx)\triangleq f(\bx)+\tfrac{\alpha}{2}d_{\cC}^2(\bx)$. 
Then $f_\alpha$ is \ehh{restricted strongly convex}
with modulus ${\mu}_\alpha \triangleq \frac{\bar{\mu}/N~+\alpha}{2}-\sqrt{\big(\frac{\bar{\mu}/N~-\alpha}{2}\big)^2+4\bar{L}^2}>0$ for any $\alpha>\frac{4}{ \bar{\mu}}N\bar{L}^2$ \nsa{where $\bar{L}$ is defined in Assumption~\ref{assump:lipschitz_f}.} 
\end{lemma}
\begin{remark}
Lemma~\ref{lem:restricted-convex-general} generalizes Proposition~3.6. of \cite{shi2015extra} to the time-varying network setting.
{If $\ubar{\mu}>0$, we set $\alpha=0$ and choose $\mu=\ubar{\mu}$; otherwise, when $\ubar{\mu}=0$,
since $f_\alpha$ is \ehh{restricted strongly convex} with modulus $\mu_\alpha>0$ for any $\alpha>\frac{4}{\bar{\mu}}N\bar{L}^2$, we set $\mu=\mu_\alpha$ for some $\alpha>\frac{4}{\bar{\mu}}N\bar{L}^2$.}
\end{remark}
\vspace*{-3mm}
\section{Numerical Experiments}
\label{sec:numerics}
\rev{In this section, we test the performance of the proposed algorithms and compare it with other decentralized and centralized algorithms to compute a projection onto the intersection of ellipsoids. Given $x_0$, let $x^*$ be the optimal solution to 
{
\begin{align}\label{eq:qcqp}
\min_{x\in\mathbb{R}^n:\norm{x}\leq D}\{\tfrac{1}{2}\norm{x-x_0}^2:\tfrac{1}{2}x^\top A_i x+b_i^\top x\leq c_i,~\forall i\in\cN\}
\end{align}}%
over network $\cG^t=(\cN,\cE^t)$ for $t\geq 0$ --additional numerical experiments are presented in  Supplementary Material. We compare our methods with DRPDS \cite{yuan2015regularized} (decentralized method) and Mirror-prox \cite{he2015mirror} (centralized method), in terms of the relative error and infeasibility of the iterate sequence, i.e., $ \max_{i\in\cN}\norm{{x}_i^k-x^*}/ \norm{x^*}$ and {$\max_{i\in\cN}\norm{(g_i(\bar{x}_i^k))_+}$}.
}

\rev{In the experiments, we set $n=20$, $D=5$ and generate $x_0$ such that its entries are 
i.i.d 
with uniform distribution
on $[-1,1]$. For each $i\in\cN$, we generate $c_i\in\reals$ uniformly at random within $[0.5,1.5]$, $b_i\in\reals^n$ such that
its entries are 
i.i.d. 
with standard Gaussian distribution, and  $A_i=R_i^\top R_i/\norm{R_i}\in\reals^{n\times n}$ where 
entries of $R_i\in \reals^{n\times n}$ are 
i.i.d. 
with standard Gaussian distribution. Let {$g_i(x)\triangleq \tfrac{1}{2}x^\top A_i x+b_i^\top x-c_i$} for $i\in\cN$. 
\sa{Below, we fix 
$|\cN|=12$ and $|\cE|=24$.}} For both DPDA and DPDA-TV, we chose $\gamma^0=\frac{1}{4}$ and $\delta= C_\min$. \saa{Moreover, since $\bar{x}=\mathbf{0}$ is a Slater point of \eqref{eq:qcqp}, a dual bound $B=\frac{1}{2}\norm{x_0}^2/\min_{i\in\cN}\{c_i\}$ can be obtained according to \cite[Lemma 6.1]{aybat2019distributed}.}

\begin{figure}[h]
\vspace*{-3mm}
\centering
\includegraphics[scale=0.4]{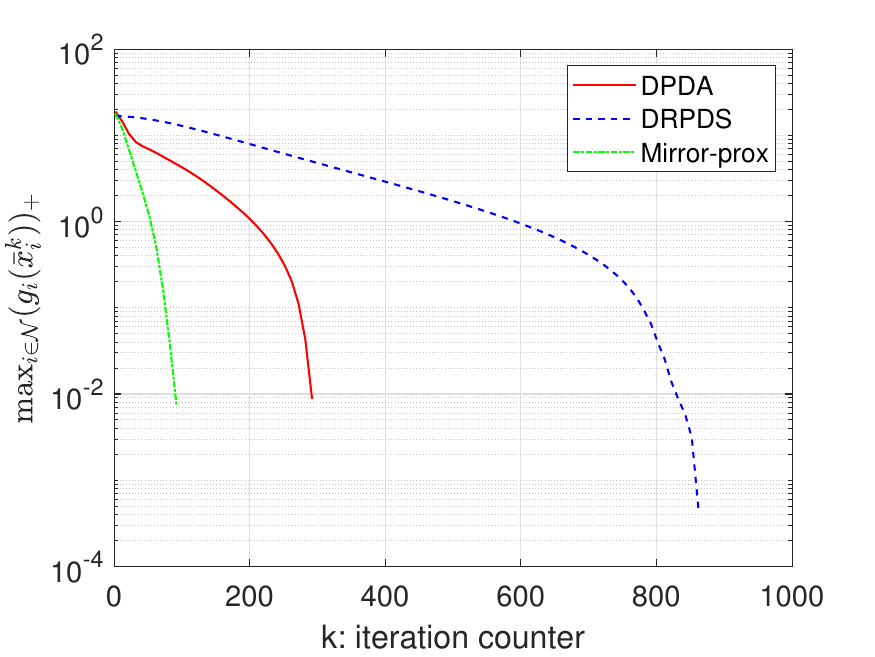}\hspace*{-3mm}
\includegraphics[scale=0.4]{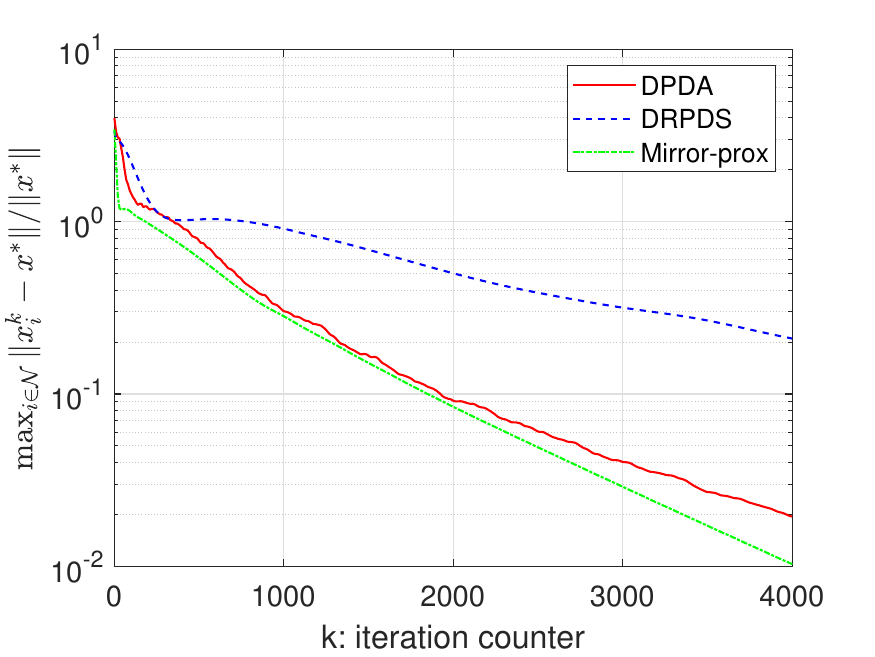}
\caption{\sa{DPDA against DRPDS and Mirror-prox on a static undirected network.}}\vspace*{-2mm}
\label{fig:qcqp_s}
\end{figure}
\rev{\indent{\bf Test on static undirected networks:} $\cG=(\cN,\cE)$ is generated as a random small-world network, i.e., 
we choose $|\cN|$ edges creating a random cycle over nodes, 
then the remaining $|\cE|-|\cN|$ edges are selected uniformly at random.
The results in Figure \ref{fig:qcqp_s} shows that our method outperform DRPDS and has a comparable performance against the centralized Mirror-prox.}

\begin{figure}[h]
\vspace*{-3mm}
\centering
\includegraphics[scale=0.4]{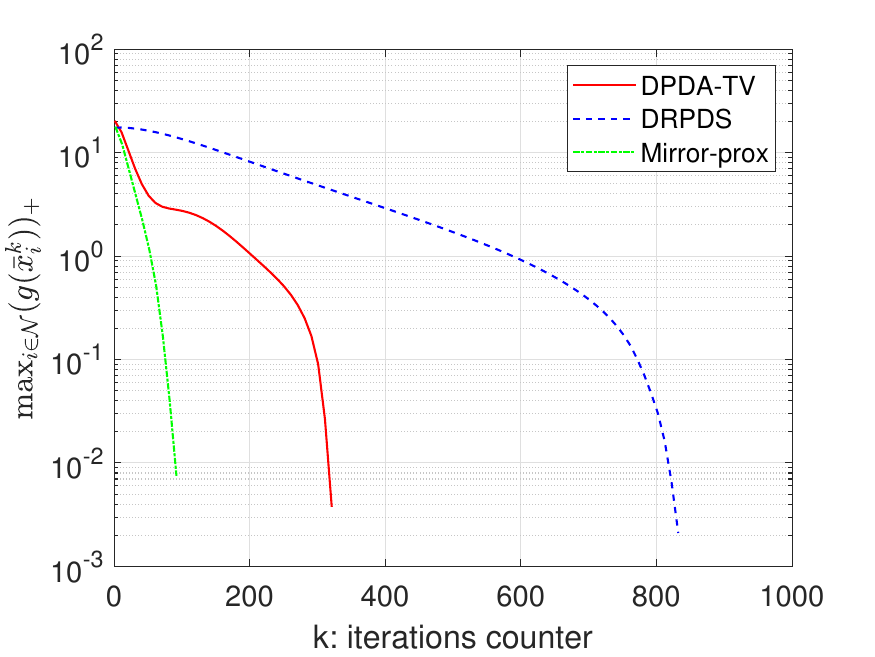}\hspace*{-3mm}
\includegraphics[scale=0.4]{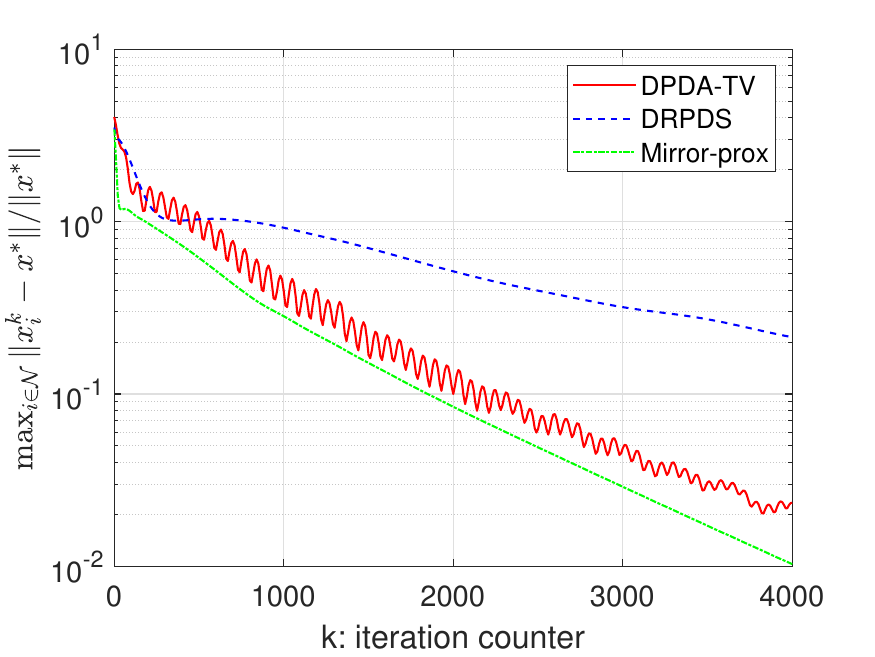}
\caption{\sa{DPDA-TV against DRPDS and Mirror-prox on a time-varying undirected network.}}\vspace*{-3mm}
\label{fig:qcqp_d_u_iter}
\end{figure}
\begin{figure}[h]
\vspace*{-1mm}
\centering
\includegraphics[scale=0.4]{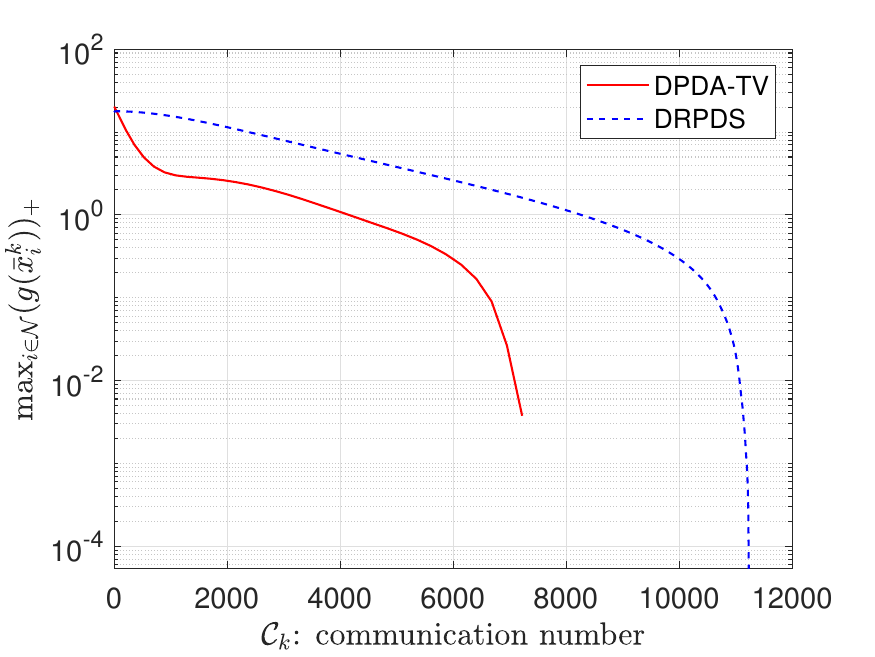}\hspace*{-3mm}
\includegraphics[scale=0.4]{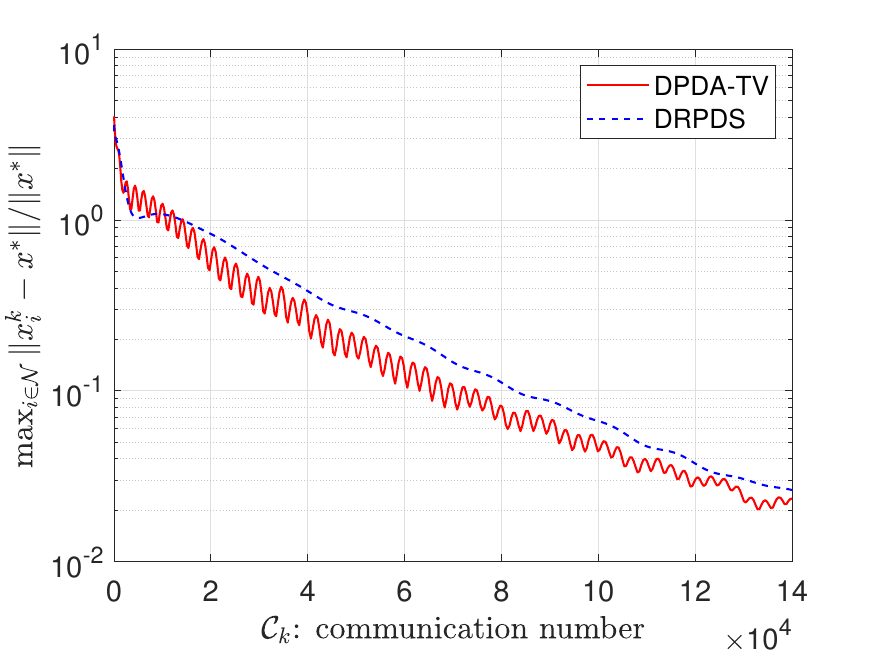}
\caption{Communication complexity of DPDA-TV against DRPDS 
on a time-varying undirected network.}\vspace*{-3mm}
\label{fig:qcqp_d_u_com}
\end{figure}
\rev{
\indent{\bf Test on time-varying undirected networks:} 
Consider the random small-world \sa{$\cG=(\cN,\cE)$} as described above. Given $M\in\integers_+$, and $p\in(0,1)$, for each $k\in\integers_+$, we generate $\cG^t=(\mathcal{N},\mathcal{E}^t)$, the communication network at time $t\in\{(k-1)M,\ldots,kM-2\}$ by sampling $\lceil p |\cE|\rceil$ edges of $\cG$ uniformly at random and we set $\mathcal{E}^{kM-1}=\cE\setminus \bigcup_{t=(k-1)M}^{kM-2}\cE^t$. In all experiments, we set $M=5$, $p=0.8$ and the number of communications per iteration is set to $q_k=5\ln(k+1)$.
For each consensus round $t\geq 1$, $V^t$ is formed according to Metropolis weights, i.e., for each $i\in\cN$, $V^t_{ij}=1/(\max\{d_i,d_j\}+1)$ if $j\in\cN_i^t$,  $V^t_{ii}=1-\sum_{i\in\cN_i}V^t_{ij}$, and $V^t_{ij}=0$ otherwise -- see~\eqref{eq:approx-average-dual-undirected} for our choice of $\cR^k$. 
Figure \ref{fig:qcqp_d_u_iter} show the comparison of the methods when the $x$-axis is in terms of iteration counter $(k)$. We observe that DPDA-TV outperform DRPDS and slightly slower rate of DPDA-TV \sa{in comparison to Mirror-prox} is the price we pay for the decentralized setting. Note that DRPDS has one communication round per iteration; hence, we also plot Figure \ref{fig:qcqp_d_u_com} where the $x$-axis denotes the total number of communication rounds, i.e., $\mathcal{C}_k=\sum_{i=0}^{k-1}q_i$, and our method has still a better performance.}
\begin{figure}[h]
\vspace*{-3mm}
\centering
\includegraphics[scale=0.4]{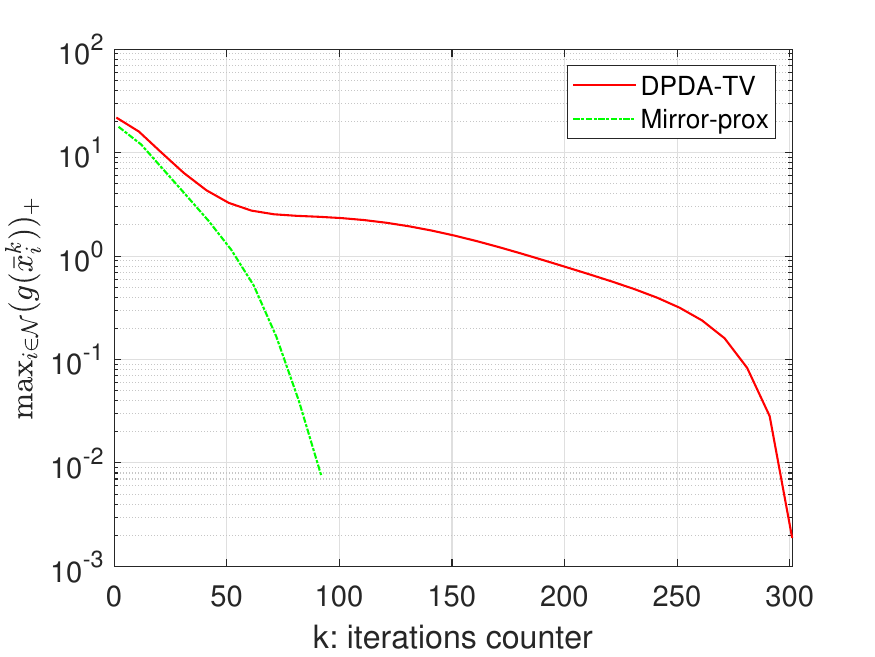}\hspace*{-3mm}
\includegraphics[scale=0.4]{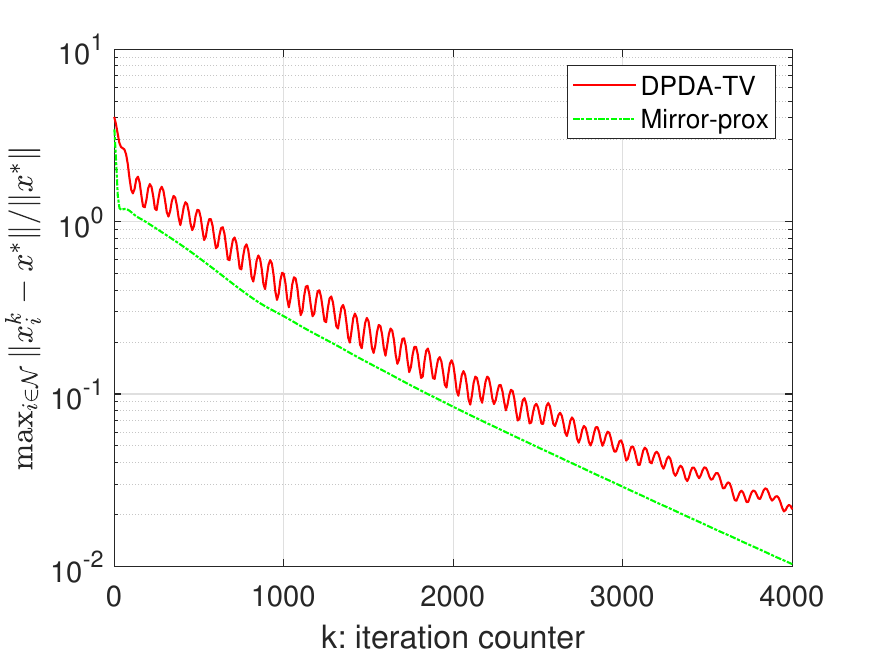}
\caption{DPDA-TV against Mirror-prox over time-varying directed network.}\vspace*{-1mm}
\label{fig:qcqp_d_d}
\end{figure}
\begin{figure}[htbp]
\centering
\begin{minipage}[c]{0.5\textwidth}
    \begin{center}
      \begin{tikzpicture}[scale=0.9]
        \coordinate (x1) at (-0.5,0.2);
        \coordinate (x2) at (4.4,1.3);
        \coordinate (x3) at (4.4,-0.3);
        \coordinate (x4) at (1.9,1.8);
        \coordinate (x5) at (3,-0.8);
        \coordinate (x6) at (1.5,-0.8);
        \coordinate (x7) at (3,1.5);
        \coordinate (x8) at (0,-0.6);
        \coordinate (x9) at (1.25,1.25);
        \coordinate (x10) at (0.6,0.7);
        \coordinate (x11) at (0.4,1.5);
        \coordinate (x12) at (3.6,0.45);


        \draw [arrows={- triangle 45}] (x1) -- (x10);
        \node[align=left, above] at (x10) { $10$};

        \draw [arrows={- triangle 45}] (x1) -- (x6);
        \node[align=left, below] at (x6) { $6$};

        \draw [arrows={- triangle 45}] (x8) -- (x1);
        \node[align=left, left] at (x1) { $1$};

        \draw [arrows={- triangle 45}] (x8) -- (x10);

        \draw [arrows={triangle 45 - triangle 45}] (x8) -- (x6);
        \node[align=left, left] at (x8) { $8$};

        \draw [arrows={- triangle 45}] (x6) -- (x3);
        \node[align=left, right] at (x3) { $3$};

        \draw [arrows={- triangle 45}] (x11) -- (x1);

        \draw [arrows={- triangle 45}] (x9) -- (x11);
        \node[align=left, left] at (x11) { $11$};

        \draw [arrows={- triangle 45}] (x9) -- (x3);

         \draw [arrows={- triangle 45}] (x9) -- (x5);
        \node[align=left, left] at (x5) { $5$};

        \draw [arrows={- triangle 45}] (x4) -- (x9);
        \node[align=left, right] at (x9) { $9$};

        \draw [arrows={- triangle 45}] (x4) -- (x11);

        \draw [arrows={- triangle 45}] (x7) -- (x4);
        \node[align=left, above] at (x4) { $4$};

        \draw [arrows={- triangle 45}] (x7) -- (x12);
        \node[align=left, right] at (x12) { $12$};

        \draw [arrows={- triangle 45}] (x7) -- (x6);

        \draw [arrows={- triangle 45}] (x2) -- (x10);

        \draw [arrows={- triangle 45}] (x12) -- (x6);

        \draw [arrows={- triangle 45}] (x12) -- (x2);
        \node[align=left, above] at (x2) { $2$};

        \draw [arrows={- triangle 45}] (x12) -- (x5);

        \draw [arrows={- triangle 45}] (x3) -- (x12);

        \draw [arrows={- triangle 45}] (x5) -- (x3);

        \draw [arrows={- triangle 45}] (x10) -- (x7);
        \node[align=left, above] at (x7) { $7$};

        \draw [arrows={- triangle 45}] (x10) -- (x5);

        \filldraw[fill=red!50!](x1) circle [radius=0.08];
        \filldraw[fill=red!50!] (x2) circle [radius=0.08];
        \filldraw[fill=red!50!] (x3) circle [radius=0.08];
        \filldraw[fill=red!50!] (x4) circle [radius=0.08];
        \filldraw[fill=red!50!] (x5) circle [radius=0.08];
        \filldraw[fill=red!50!](x6) circle [radius=0.08];
        \filldraw[fill=red!50!] (x7) circle [radius=0.08];
        \filldraw[fill=red!50!] (x8) circle [radius=0.08];
        \filldraw[fill=red!50!] (x9) circle [radius=0.08];
        \filldraw[fill=red!50!] (x10) circle [radius=0.08];
        \filldraw[fill=red!50!] (x11) circle [radius=0.08];
        \filldraw[fill=red!50!] (x12) circle [radius=0.08];
      \end{tikzpicture}
    \end{center}
\end{minipage}
 \caption{$\cG_d=(\cN,\cE_d)$ directed strongly connected graph} \label{fig:Gd}
\end{figure}
\rev{
\indent{\bf Time-varying directed network:} We generate directed time-varying communication networks similar to \cite{nedich2016achieving}. Let $\cG=(\cN,\cE)$ be the directed graph shown in Fig.~\ref{fig:Gd} where it has $|\cN|=12$ nodes and $|\cE|=24$ directed edges. 
$\{\cG^t\}_{t\geq 0}$ is generated as in the undirected case with parameters $M=5$ and $p=0.8$; hence, $\{\cG^t\}_{t\geq 0}$ is $M$-strongly-connected. Moreover, communication weight matrices $V^t$ are formed according to rule \eqref{eq:directed-weights}. We chose the initial step-sizes for DPDA-TV as in the time-varying undirected case. Note that DRPDS cannot handle directed networks; hence, in Fig.~\ref{fig:qcqp_d_d} we only compare DPDA-TV against Mirror-prox. We observe that over time-varying directed networks DPDA-TV 
\sa{is still competitive against Mirror-prox}.}
\singlespacing
\bibliographystyle{IEEEtran}
\bibliography{papers}{}
\section{Appendix}
\subsection{Proof of Lemma~\ref{lem:restricted-convex-static}}
Let $\bx^*=\ones_{|\cN|}\otimes x^*$, where $x^*$ is the 
solution to \eqref{eq:central_problem} 
for $\bar{f}$ strongly convex with modulus $\bar{\mu}>0$ -- see Assumption~\ref{assump:sconvex}. Fix 
$\alpha>\frac{4}{{\lambda}_2}\sum_{i\in\cN}L_{f_i}^2/\bar{\mu}$ and $\bx\in\reals^{n|\cN|}$. \sa{WLOG suppose $n=1$.} Using 
$\ns(W)=\spn\{\ones\}$, any $\bx\in\reals^{|\cN|}$ can be decomposed into $\bu\in \spn\{\ones\}$ and $\bv\in \spn\{\ones\}^\perp$ where $\bx=\bu+\bv$ and $\norm{\bx}^2=\norm{\bu}^2+\norm{\bv}^2$. From definition of $f_\alpha$,
\vspace*{-2mm}
{
\begin{eqnarray}\label{grad-g}
\lefteqn{\fprod{\nabla f_{\alpha}(\bx)-\nabla f_{\alpha}(\bx^*),~\bx-\bx^*} =} \nonumber \\
& & \fprod{\nabla f(\bx)-\nabla f(\bx^*),~\bx-\bx^*}+\alpha\norm{\bx-\bx^*}_{W\otimes \id_n}^2. \vspace*{-6mm}
\end{eqnarray}
}%
Let $\bar{L}\triangleq \sqrt{\sum_{i\in\cN}L_{f_i}^2/N}$, where $N\triangleq |\cN|$. The inner product on the rhs of~\eqref{grad-g} can be bounded by
\ehh{using Lipschitz continuity of $\grad f$ and strong convexity of $f$}
as follows:
{
\begin{flalign*}
&\fprod{\grad f(\bx)-\grad f(\bx^*),~\bx-\bx^*} =\fprod{\grad f(\bu)-\grad f(\bx^*) ,\bu-\bx^*}\\
&+\fprod{\grad f(\bx)-\grad f(\bu),\bx-\bu}+\fprod{\grad f(\bu)-\grad f(\bx^*),\bx-\bu}\\
&+\fprod{\grad f(\bx)-\grad f(\bu),\bu-\bx^*} \geq \frac{\bar{\mu}}{N}\norm{\bx^*-\bu}^2-2\bar{L}\norm{\bx^*-\bu}\norm{\bv}. \vspace*{-3mm} 
\end{flalign*}
}%
Thus, from \eqref{grad-g} 
and the fact that $\norm{\bx-\bx^*}_{W\otimes \id_n}^2=\norm{\bv}_{W\otimes \id_n}^2\geq {\lambda}_2 \norm{\bv}^2$, 
we get
{
\begin{eqnarray}\label{g-strong-convex-s}
\lefteqn{\fprod{\nabla f_{\alpha}(\bx)-\nabla f_{\alpha}(\bx^*),~\bx-\bx^*}  \geq} \nonumber \\
& &\frac{\bar{\mu}}{N}\norm{\bx^*-\nsa{\bu}}^2-2\bar{L}\norm{\bx^*-\bu}\norm{\bv}+\alpha{\lambda}_2 \norm{\bv}^2.
\end{eqnarray}
}%
Next, fix $\omega>0$. Then either \textbf{(i)} $\norm{\bv}\leq \omega \norm{\bu-\bx^*}$, or \textbf{(ii)} $\norm{\bv}\geq \omega \norm{\bu-\bx^*}$ holds. Using the same arguments to obtain \eqref{g-strong-convex-result}, we can conclude that
{
\begin{eqnarray}\label{g-strong-convex-result-s}
\lefteqn{\fprod{\nabla f_{\alpha}(\bx)-\nabla f_{\alpha}(\bx^*),~\bx-\bx^*} } \nonumber \\
& & \geq \min\Big\{\frac{\bar{\mu}}{N}-2\bar{L}\omega, ~\alpha {\lambda}_2-\frac{2\bar{L}}{\omega}\Big\} \norm{\bx-\bx^*}^2.
\end{eqnarray}
}%
Since $\omega\geq 0$ is arbitrary, {$f_{\alpha}$ is \ehh{restricted strongly convex}} with modulus $\mu_\alpha=\max_{\omega\geq 0}\min\Big\{\frac{\bar{\mu}}{N}-2\bar{L}\omega, ~\alpha {\lambda}_2-\frac{2\bar{L}}{\omega}\Big\}$.
Note $\mu_\alpha$ is attained for $\omega_\alpha\geq 0$ such that $\frac{\bar{\mu}}{N}-2\bar{L}\omega_\alpha=\alpha {\lambda}_2-\frac{2\bar{L}}{\omega_\alpha}$, which implies that $\omega_\alpha=\frac{1}{2}\Big(\frac{\bar{\mu}/N~-\alpha {\lambda}_2}{2\bar{L}} + \sqrt{\frac{\bar{\mu}/N~-\alpha {\lambda}_2}{2\bar{L}} +4} \Big)$. 
Moreover, $\mu_\alpha=\frac{\bar{\mu}}{N}-2\bar{L}\omega_\alpha$ is the value given in the statement of the lemma, and we have $\frac{\bar{\mu}}{N}>\mu_\alpha>0$ for any $\alpha>\frac{4N}{{\lambda}_2 \bar{\mu}}\bar{L}^2$. It is worth mentioning that $\mu_\alpha$ is a concave increasing function of $\alpha$ over $\reals_{++}$, and $\sup_{\alpha>0} \mu_\alpha = \lim_{\alpha\nearrow\infty}\mu_\alpha=\frac{\bar{\mu}}{N}$.\qed

\subsection{Proof of Lemma~\ref{lem:restricted-convex-general}}
\rev{Let $\bx^*=\ones_{|\cN|}\otimes x^*$, where $x^*$ is the unique optimal solution to \eqref{eq:central_problem}, and according to Assumption~\ref{assump:sconvex}, \ehh{$\bar f$} is strongly convex
with modulus $\bar{\mu}>0$. Fix some arbitrary $\alpha>\frac{4}{\bar{\mu}}\sum_{i\in\cN}L_{f_i}^2$ and $\bx\in\reals^{n|\cN|}$. Since $\cC$ is a closed convex cone, $\bx$ can be decomposed into $\bu=\cP_\cC(\bx)$ and $\bv=\cP_{\cC^\circ}(\bx)$, i.e., $\bx=\bu+\bv$ and $\norm{\bx}^2=\norm{\bu}^2+\norm{\bv}^2$. From the definition of $f_\alpha$,
{
\begin{eqnarray}\label{grad-g-d}
\lefteqn{\fprod{\nabla f_\alpha(\bx)-\nabla f_\alpha(\bx^*),~\bx-\bx^*}=} \nonumber \\
& &\fprod{\nabla f(\bx)-\nabla f(\bx^*),~\bx-\bx^*}+\alpha\fprod{\bx-\bx^*,~\bv},
\end{eqnarray}
}%
which follows from the fact that $\grad r(\bx)=\bx-\cP_\cC(\bx)$; hence $\grad r(\bx^*)=\zero$.
Let $N\triangleq |\cN|$ and $\bar{L}\triangleq \sqrt{\frac{\sum_{i\in\cN}L_{f_i}^2}{N}}$. Since $\bx^*,\bu\in\cC$ and $f$ is convex, Lipschitz differentiable, and strongly convex, we get
{
\begin{align*}
\fprod{\nabla f(\bx)-f(\bx^*),\bx-\bx^*} \geq \frac{\bar{\mu}}{N}\norm{\bu-\bx^*}^2-2\bar{L}\norm{\bu-\bx^*}\norm{\bv}
\end{align*}}%
\sa{--for more details see the proof of a related result in Lemma~\ref{lem:restricted-convex-static}.}
Note that $\bu-\bx^*\in\cC$; hence, $\fprod{\bu-\bx^*,\bv}=0$ since $\bv\in\cC^\circ$. Thus, $\fprod{\bx-\bx^*, \bv}=\norm{\bv}^2$. 
Therefore, using \eqref{grad-g-d} and 
and the previous inequality
implies that
{
\begin{eqnarray}\label{g-strong-convex}
\lefteqn{\fprod{\nabla f_\alpha(\bx)-\nabla f_\alpha(\bx^*),\bx-\bx^*}\geq} \nonumber \\
& &  \frac{\bar{\mu}}{N}\norm{\bu-\bx^*}^2-2\bar{L}\norm{\bu-\bx^*}\norm{\bv}+\alpha \norm{\bv}^2.
\end{eqnarray}
}%
Next, fix some arbitrary $\omega\geq 0$.
We consider two cases:
(i) if $\norm{\bv}\leq \omega \norm{\bu-\bx^*}$, then from \eqref{g-strong-convex}, we get
{
\begin{eqnarray}\label{g-strong-convex-i}
\lefteqn{\fprod{\nabla f_{\alpha}(\bx)-\nabla f_{\alpha}(\bx^*),~\bx-\bx^*} } \nonumber \\
& & \geq \big(\frac{\bar{\mu}}{N}-2\omega\bar{L}\big)\norm{\bu- \bx^*}^2+\alpha  \norm{\bv}^2\nonumber  \\
& & \geq \min\left\{\frac{\bar{\mu}}{N}-2\omega\bar{L}, ~\alpha \right\} \norm{\bx-\bx^*}^2;
\end{eqnarray}
}%
and (ii) if $\norm{\bv}\geq \omega \norm{\bu-\bx^*}$, then we get
{
\begin{eqnarray}\label{g-strong-convex-ii}
\lefteqn{ \fprod{\nabla f_{\alpha}(\bx)-\nabla f_{\alpha}(\bx^*),~\bx-\bx^*}  \frac{\bar{\mu}}{N}\norm{\bu- \bx^*}^2+\big(\alpha-\frac{2\bar{L}}{\omega} \big)\norm{\bv}^2} \nonumber \\
& & \geq \min\Big\{\frac{\bar{\mu}}{N}, ~\alpha -\frac{2\bar{L}}{\omega}\Big\} \norm{\bx-\bx^*}^2.
\end{eqnarray}
}%
Combining \eqref{g-strong-convex-i} and \eqref{g-strong-convex-ii} we conclude that
{
\begin{eqnarray}\label{g-strong-convex-result}
\fprod{\nabla f_\alpha(\bx)-\nabla f_\alpha(\bx^*),~\bx-\bx^*}\geq \min\left\{\frac{\bar{\mu}}{N}-2\bar{L}\omega, ~\alpha-\frac{2\bar{L}}{\omega}\right\} \norm{\bx-\bx^*}^2.
\end{eqnarray}
}%
Since $\omega\geq 0$ is arbitrary, $f_\alpha$ is restricted strongly convex with respect to $\bx^*$ with modulus $\mu_\alpha=\max_{\omega\geq 0} \min\left\{\frac{\bar{\mu}}{N}-2\bar{L}\omega, ~\alpha-\frac{2}{\omega}\bar{L}\right\}$. Note $\mu_\alpha$ is attained for $\omega_\alpha\geq 0$ such that $\frac{\bar{\mu}}{N}-2\bar{L}\omega_\alpha= \alpha-\frac{2}{\omega_\alpha}\bar{L}$, which implies that $\omega_\alpha=\frac{\bar{\mu}/N-\alpha}{4\bar{L}} + \sqrt{\left(\frac{\bar{\mu}/N-\alpha}{4\bar{L}}\right)^2 +1}$. Moreover, $\mu_\alpha=\frac{\bar{\mu}}{N}-2\bar{L}\omega_\alpha$ is the value given in the statement of the lemma, and we have $\frac{\bar{\mu}}{N}>\mu_\alpha>0$ for any $\alpha>\frac{4N}{ \bar{\mu}}\bar{L}^2$. It is worth mentioning that $\mu_\alpha$ is a concave increasing function of $\alpha$ over $\reals_{++}$, and $\sup_{\alpha>0} \mu_\alpha = \lim_{\alpha\nearrow\infty}\mu_\alpha=\frac{\bar{\mu}}{N}$.} \qed
\vspace*{-3mm}

\subsection{Key lemmas for the proof of Theorem~\ref{thm:dynamic-error-bounds}}
We first define \saa{the proximal error sequence $\{\be^k\}_{k\geq 1}$,} 
which will be used for analyzing the 
Algorithm~DPDA-TV displayed in Fig.~\ref{alg:PDD}. For $k\geq 0$, let
{%
\begin{align}
\label{eq:prox-error-seq-1}
&\be^{k+1} \triangleq \cP_{\Ct}\left(\bom^k\right)-\tcR^k\left(\bom^k\right),
\end{align}
}%
where $\bom^k=\tfrac{1}{\gamma^k}{\bmu}^k+{{\bx}^{k+1}}$ and $\tcR^k(\bx)=\cP_{\cB}(\cR^k(\bx))$, i.e., $\tcR^k(\bx)=[\tcR_i^k(\bx)]_{i\in\cN}$ and $\tcR_i^k(\bx)=\cP_{\cB_0}(\cR_i^k(\bx))$, for $\bx\in\cX$. Thus, for $k\geq 0$, $\bmu^{k+1}=\blambda^{k+1}+\gamma^k \be^{k+1}$ \saa{since 
\eqref{eq:pock-pd-2-nu} is replaced with \eqref{eq:inexact-rule-mu}.}
In the rest, we set 
$\bmu^0$ to $\mathbf{0}$.

The following observations will also be useful for proving error bounds for DPDA-TV iterate sequence:

\noindent \textbf{(i)} Note that \eqref{eq:inexact-rule-mu} {and boundedness of domain of $\varphi_i$} implies for each $i\in\cN$,
$\|\nu_i^{k+1}\| \leq \gamma^k\|\omega_i^{k}\|+\break \gamma^k\|{\widetilde{\cR}}_i^k\big(\bom^k\big)\|\leq \|\nu_i^k\|+\gamma^k\big(\|x_i^{k+1}\|+2\Delta\big){\leq \|\nu_i^k\|+3\gamma^k\Delta}$.
Thus, we trivially get the following bound on $\norm{\bmu^{k+1}}$:
\vspace*{-3mm}
{
\begin{equation}\label{eq:mu-bound}
\|\bmu^{k+1}\| \leq~3\sqrt{N}\Delta\sum_{t=0}^k \gamma^t.   
\end{equation}
}%
\textbf{(ii)} Moreover, 
for any $\bmu$ and $\blambda$ we have that
{
\begin{equation} \label{eq:support-error}
\sigma_{\Ct}({\bmu})=\sup_{{\bf x}\in \Ct}~\langle \blambda,{\bf x} \rangle+\langle \bmu-\blambda,{\bf x} \rangle \leq \sigma_{\Ct}(\blambda)+2\sqrt{N}~\Delta~\|\bmu-\blambda\|.
\end{equation}
}%
\begin{defn}\label{definition}
Given $\mu,\delta,\nsa{B}>0$, and 
positive sequences
$\{\tau^k,[\kappa_i^k]_{i\in\cN},\gamma^k\}_k$, define {$\bD^k_{\tau}\triangleq\frac{1}{\tau^k}\id_{n|\cN|}$,} \saa{$\mathbf{\widetilde{D}}^k_{\tau}\triangleq\bD^k_{\tau}-{\mu}\id_{n|\cN|}$}, 
$\mathbf{\bar{D}}^k_{\tau}\triangleq\diag([(\frac{1}{ \tau^k} -L_{f_i})\id_{n}]_{i\in \mathcal{N}})$, and $\bL_G\triangleq \diag([L_{g_i}\id_n]_{i\in\cN})$.
Let {$\cA^k\triangleq(\eta^k)^2\gamma^{k}(1+{\delta })~\id_{n|\cN|}+\eta^kB\bL_G\succ 0$}, and $\bD^k_{\kappa,\gamma}$ is defined in Definition~\ref{def:bregman}. \nsa{Define nonlinear map $T:\cX\to\cY$ such that $T(\bx)\triangleq \eyh{[G(\bx)^\top;\bx^\top]}^\top$.}
\end{defn}
To prove Theorem~\ref{thm:dynamic-error-bounds}, we first
need 
Lemmas~\ref{lem:step-size} and \ref{lem:bound-J-s} describing a proper choice {for the 
step-size sequences.}
\begin{lemma}\label{lem:step-size}
Let $\bD^k_{\kappa,\gamma}$ be as 
in Definition \ref{def:bregman}, \nsa{$\bD^k_{\tau}$,} $\mathbf{\widetilde{D}}^k_{\tau}$ and $\nsa{\bar{\bD}^k_{\tau}}$ be as in Definition \ref{definition} for 
\saa{any $\mu\in(0,\ubar{\mu}]$}. Suppose
positive step-size sequences $\{\tau^k,[\kappa_i^k]_{i\in\cN},\gamma^k,\nsa{\eta^k}\}_k$ are chosen as in DPDA-TV displayed in Fig.~\ref{alg:PDD} \nsa{for some $B>0$}, then the following relations hold for all $k\geq 0$:\vspace*{-3mm}
\begin{subequations}\label{eq:cond-d}%
\begin{minipage}{0.48\hsize}
\begin{align}
&{\gamma^k}\bD^k_{\tau}\succeq \gamma^{k+1}\mathbf{\widetilde{D}}^{k+1}_{\tau}, \label{cond-1-d}\\
& \gamma^k\bD^k_{\kappa}\succeq \gamma^{k+1}\bD^{k+1}_{\kappa},\label{cond-2-d}\\
& {\bar\bD^k_{\tau}\succeq \cA^{k}+B\bL_G} \label{cond-5-d},
\end{align}
\end{minipage}%
\begin{minipage}{0.48\hsize}
\begin{align}
& \gamma^k\bD^k_{\gamma}\succeq \gamma^{k+1}\bD^{k+1}_{\gamma}, \label{cond-3-d}\\
& \gamma^k= \gamma^{k+1}\eta^{k+1},\label{cond-4-d}\\
& {\widetilde{\bD}^K_{\tau}\succeq 2\cA^k}. \label{cond-6-d}
\end{align}
\end{minipage}
\end{subequations}

\noindent Moreover, $\eta^k\in(0,1)$, \saa{$\frac{1}{\tau^k}=\frac{1}{\tilde\tau^k}+\mu=\Theta(\mu k)$, and $\gamma^k=\Theta(\mu \tilde\tau^0\gamma^0 k)$ for $k\geq 1$}.
\end{lemma}
\begin{proof}
The conditions \eqref{cond-2-d}, \eqref{cond-3-d}, and \eqref{cond-4-d} trivially hold from the step-size update rules of $\eta^{k+1}$ and \nsa{$\kappa_i^{k+1}$ for $i\in\cN$}. The condition \eqref{cond-1-d} can be equivalently written as $\frac{1}{\tilde{\tau}^k}+\mu\geq \frac{1}{\tilde{\tau}^{k+1}\eta^{k+1}}$, \nsa{which holds with equality since $1/\eta^{k+1}=\sqrt{1+\mu\tilde{\tau}^k}$}; moreover, \eqref{cond-5-d} and \eqref{cond-6-d} hold if $\frac{1}{\tilde{\tau}^k}\geq L_{\max}\nsa{(f)} 
+2(\gamma^k(1+\delta)+BL_{\max}(G))$, \nsa{which can be shown through induction and using $\{\gamma^k\}$ is increasing}. 

\saa{Finally, 
from step-size update rule in Fig.~\ref{alg:PDD}, we have $\gamma^{k+1}=\gamma^k\sqrt{1+\mu\tilde\tau^k}$ and $\gamma^{k+1}\tilde\tau^{k+1}=\gamma^k\tilde\tau^k$, for all $k\geq 0$, which implies that $\gamma^k\tilde\tau^k=\gamma^0\tilde\tau^0$. Hence, we conclude that $(\gamma^{k+1})^2=(\gamma^k)^2+\mu\gamma^0\tilde\tau^0\gamma^k$. Using induction one can show that $\frac{\mu\gamma^0\tilde{\tau}^0}{3}k\leq \gamma^k\leq \frac{\mu\gamma^0\tilde{\tau}^0}{2}k$
for any $k\geq 1$; hence, $\gamma^k=\Theta(\mu \tilde\tau^0\gamma^0 k)$ and $1/\tilde\tau^k=\gamma^k/(\tilde\tau^0\gamma^0)=\Theta(\mu k)$. }
\vspace*{-1mm}
\end{proof}
\vspace*{-5mm}
\nsa{\begin{lemma}\label{lem:bound-J-s}
For any $\bx\in\dom \rho$ and any positive step-size sequences $\{\tau^k,[\kappa_i^k]_{i\in\cN},\gamma^k,\nsa{\eta^k}\}_k$, we have
{
\begin{equation*}
\eta^k\fprod{\bJ T(\bx)^\top(\by-\bar{\by}),\bx-\bar\bx}\leq \tfrac{1}{2}\norm{\bx-\bar{\bx}}_{\bD^k}^2+\tfrac{1}{2}\norm{\by-\bar{\by}}^2_{\bD^k_{\kappa,\gamma}},
\end{equation*}}%
for any $\bar{\bx}\in\cX$ and $\by=\eyh{[{\btheta}^\top; {\bmu}^\top]}^\top,~\bar{\by}=\eyh{[\bar{\btheta}^\top ;\bar{\bmu}^\top]}^\top\in\cY$, where $\bD^k\triangleq(\eta^k)^2\diag([(C_{g_i}^2\kappa^k_i+\gamma^k)\id_{n}]_{i\in\cN})$. 
\end{lemma}
\begin{proof}
For $F\in\sy^{n}_{++}$, we have $\fprod{z,\bar{z}}\leq \tfrac{1}{2}\norm{z}^2_{F^{-1}}+\tfrac{1}{2}\norm{\bar{z}}^2_{F}$ for any $z,\bar{z}\in\reals^n$. Using this inequality twice, we get
{
\begin{eqnarray*}
\lefteqn{\eta^k\fprod{\bJ T(\bx)^\top(\by-\bar{\by}),\bx-\bar\bx}}
\nonumber\\
& &=\eta^k\fprod{\bJ G(\bx)(\bx-\bar\bx),\btheta-\bar{\btheta}} +\eta^k\fprod{\bx-\bar\bx,\bmu-\bar{\bmu}} \nonumber\\
& &\leq \tfrac{(\eta^k)^2}{2}\norm{\bJ G(\bx)(\bx-\bar\bx)}^2_{(\bD^k_\kappa)^{-1}}+\tfrac{1}{2}\norm{\btheta-\bar{\btheta}}^2_{\bD_\kappa^k}+\tfrac{(\eta^k)^2\gamma^k}{2}\norm{\bx-\bar\bx}^2+\tfrac{1}{2\gamma^k}\norm{\bmu-\bar{\bmu}}^2
\end{eqnarray*}}%
where in the first inequality we used \ehh{$F=\bD_\kappa^k/\eta^k$} for the first term and \ehh{$F=\bD_\gamma^k/\eta^k$} for the second one. 
The result follows from the fact that $g_i$ has a bounded Jacobian for any $i\in\cN$ -- see Assumption~\ref{assump:g_i}.
\end{proof}}%
\vspace*{-2mm}
\sa{
Next, we appropriately bound $\mathcal{L}(\bar{\bx}^K,\by)-\mathcal{L}(\bx^*,\bar{\by}^K)$ for all $\by\in\cY$ and \saa{$\|\bx^K-\bx^*\|$} for all $K\geq 1$, and we also \saa{account} for the approximation errors for the time-varying case, arising due to use of $\cR^k$ which allows for distributed computation.}
\begin{lemma}\label{thm:dynamic-rate}
Let $\{\bx^k,\by^k\}_{k\geq 0}$ be the iterate sequence generated by Algorithm DPDA-TV {using positive step-size sequences, $\{\tau^k,[\kappa_i^k]_{i\in\cN},\gamma^k,\nsa{\eta^k}\}_k$}, as displayed in Fig.~\ref{alg:PDD} \nsa{for some $B>0$} where $\by^k=\eyh{[{\btheta^k}^\top; {\bmu^k}^\top]}^\top$ for $k\geq 0$. Let \saa{$\{\be^k\}_{k\geq 1}$} be 
as in \eqref{eq:prox-error-seq-1} and \nsa{$B^k\triangleq\|\btheta^k\|$ for $k\geq -1$}. Then for any $k\geq 0$ and any $\by\in\cY$,
$\{\bx^k,\by^k\}_{k\geq 0}$ satisfies
{
\begin{align}\label{lemeq:inexact-lagrangian-bound}
&\saa{\mathcal{L}(\bx^{k+1},\by)-\mathcal{L}(\bx^*,\by^{k+1})\leq }\\ 
&  \mbox{} E^{k+1}(\bmu)+ \Big[\tfrac{1}{2}\|\bx^*-\bx^{k}\|_{\mathbf{\widetilde{D}}^k_{\tau}}^2+\tfrac{1}{2}\|\by-\by^k\|_{\bD^k_{\kappa,\gamma}}^2
+\tfrac{1}{2}\|\by^k-\by^{k-1}\|^2_{\bD^k_{\kappa,\gamma}} \nonumber\\
&\mbox{} +{\eta^k}\fprod{ \bJ T(\bx^k)^\top\by^k-\bJ T(\bx^{k-1})^\top\by^{k-1},~\bx^*-\bx^k} \nonumber\\
&\mbox{}+\tfrac{1}{2}\|\bx^k-\bx^{k-1}\|^2_{\eta^kB^{k-1}\bL_G}\Big] - \Big[\tfrac{1}{2}\|\bx^*-\bx^{k+1}\|_{\bD^k_{\tau}}^2 +\tfrac{1}{2}\|\by-\by^{k+1}\|_{\bD^k_{\kappa,\gamma}}^2+\tfrac{1}{2}\|\by^{k+1}-\by^{k}\|^2_{\bD^k_{\kappa,\gamma}} \nonumber \\
&  \mbox{} +\fprod{\bJ T({\bx^{k+1}})^\top\by^{k+1}-\bJ T({\bx^{k}})^\top\by^{k},~\bx^*-{\bx^{k+1}}} +\tfrac{1}{2}\|\bx^{k+1}-\bx^k\|^2_{\mathbf{\bar{D}}^k_{\tau}-(\eta^k)^2\gamma^k(\delta+1)\id_{nN}-\eta^kB^{k-1}\bL_G} \Big], \nonumber
\end{align}
}%
where 
 $E^{k+1}(\bmu)\triangleq\|\nsa{\be^{k+1}}\| \left({4}\gamma^{k}\sqrt{N}~\Delta+\|\bmu-\bmu^{k+1}\|\right)$.

\end{lemma}
\begin{proof}
{Fix $\by=\eyh{[\btheta^\top;\bmu^\top]}^\top\in\cY$.} 
{It follows from~\eqref{eq:pock-pd-2-lambda} that using} strong convexity of ${\sigma_{\Ct}}({\bmu})-\langle \bx^{k+1},~{\bmu}\rangle+\frac{1}{ 2\gamma^k}\|{\bmu}-{\bmu}^k\|^2$ in $\bmu$ and the fact that ${\blambda}^{k+1}$ is its minimizer, we conclude that
{
\begin{align*}
&{ {\sigma_{\Ct}}(\bmu)-\langle {\bx^{k+1}},~\bmu\rangle+\tfrac{1}{2\gamma^k}\|{\bmu}-{\bmu}^k\|^2 \geq } \\
& {\sigma_{\Ct}}({\blambda}^{k+1})-\langle {\bx^{k+1}},~{\blambda}^{k+1}\rangle+\tfrac{1}{2\gamma^k}\|{\blambda}^{k+1}-{\bmu}^k\|^2  +\tfrac{1}{2\gamma^k}\|\bmu-{\blambda}^{k+1}\|^2.
\end{align*}}%
According to \eqref{eq:prox-error-seq-1}, ${\bmu}^{k+1}={\blambda}^{k+1}+\gamma^k \be^{k+1}$ for all $k\geq 1$; hence, from \eqref{eq:support-error} we have 
{
\begin{eqnarray}\label{eq:support-function-bound}
\lefteqn{\sigma_\cC({\bmu})-\langle \bx^{k+1},~{\bmu}\rangle+\tfrac{1}{2\gamma^k}\|{\bmu}-{\bmu}^k\|^2 \geq}\nonumber\\
& & \sigma_\cC({\bmu}^{k+1})-\langle \bx^{k+1},~{\bmu}^{k+1}\rangle +\tfrac{1}{2\gamma^k}\|{\bmu}^{k+1}-{\bmu}^k\|^2 +\tfrac{1}{2\gamma^k}\|{\bmu}-{\bmu}^{k+1}\|^2- \saa{S^{k+1}(\bmu)},
\end{eqnarray}
}%
where  
\saa{$S^{k+1}(\bmu)\triangleq 2\gamma^k \sqrt{N}~\Delta \|\be^{k+1}\|-\gamma^k\|\be^{k+1}\|^2  \nsa{-}\fprod{\be^{k+1},~\bmu-2\bmu^{k+1}+\bmu^k+\gamma^k\bx^{k+1}}$}.
For $k\geq 0$, we have $\bmu^k+\gamma^k\bx^{k+1}=\gamma^k\bom^k$, $\bmu^{k+1}=\blambda^{k+1}+\gamma^k\be^{k+1}$, and $\blambda^{k+1}=\gamma^k(\bom^k-\cP_{\Ct}(\bom^k))$; 
hence, $\bmu^k+\gamma^k\bx^{k+1}-\bmu^{k+1}=\gamma^k(\cP_{\Ct}(\bom^k)-\be^{k+1})$.

Thus,
\setlength{\arraycolsep}{0.0em}
{
\begin{align}\label{eq:Sk-bound}
S^{k+1}(\bmu)&={2\gamma^k\sqrt{N}~\Delta}~\|\be^{k+1}\|\nsa{-}\fprod{\be^{k+1},~\bmu-\bmu^{k+1}+\gamma^k{\cP_{\Ct}(\bom^k)}}\nonumber\\
&\leq E^{k+1}(\bmu),
\end{align}}%
\setlength{\arraycolsep}{5pt}%
where the inequality follows from Cauchy-Schwarz and $\|\cP_{\Ct}(\bom^k)\|\leq 2\sqrt{N}~\Delta$ since $\cP_{\Ct}(\bom^k)\in \Ct$.
Moreover, it follows from the strong convexity of the objective in \eqref{eq:pock-pd-2-theta} that
{
\begin{eqnarray*}
\lefteqn{\sigma_{-\cK_i}(\theta_i)-\fprod{{g_i(x_i^{k+1})},\theta_i}+\tfrac{1}{2\kappa_i^k}\|\theta_i-\theta_i^k\|^2}\\
& & \geq \sigma_{-\cK_i}(\theta_i^{k+1})-\fprod{{g_i(x_i^{k+1})},\theta_i^{k+1}}+\tfrac{1}{2\kappa_i^k}\|\theta_i^{k+1}-\theta_i^k\|^2 +\tfrac{1}{2\kappa_i^k}\|\theta_i-\theta_i^{k+1}\|^2.\vspace*{-2mm}
\end{eqnarray*}}%
Summing the above inequality over $i\in\cN$, then summing the resulting inequality with \eqref{eq:support-function-bound} and using \eqref{eq:Sk-bound}, lead to \vspace*{-2mm}
{
\begin{eqnarray}\label{eq:lemma-y}
\lefteqn{h(\by)+\tfrac{1}{2}\|\by-\by^k\|_{\bD^k_{\kappa,\gamma}}^2+E^{k+1}(\bmu)\geq -\fprod{{T(\bx^{k+1})},~\by^{k+1}-\by}} \nonumber\\
& &\hspace*{-2mm}+h(\by^{k+1})+\tfrac{1}{2}\|\by-\by^{k+1}\|_{\bD^k_{\kappa,\gamma}}^2+ \tfrac{1}{2}\|\by^{k+1}-\by^k\|_{\bD^k_{\kappa,\gamma}}^2.
\end{eqnarray}}%
\vspace*{-2mm}
Next, strong convexity of the objective in \eqref{eq:pock-pd-2-x} implies that
{
\begin{eqnarray}\label{eq:rho}
\lefteqn{\rho(\bx^*)+\fprod{\grad f(\bx^k)+\bp^{k},~\bx^*}
+\tfrac{1}{2\tau^k}\|\bx^*-\bx^k\|^2}\nonumber\\
& & \geq \rho(\bx^{k+1})+\langle\grad f(\bx^k)+\bp^{k},~\bx^{k+1}\rangle
+\tfrac{1}{2\tau^k}\|\bx^{k+1}-\bx^k\|^2+\tfrac{1}{2\tau^k}\|\bx^*-\bx^{k+1}\|^2.
\end{eqnarray}}%
Next, note that since each $\grad f_i$ 
is $L_{f_i}$-Lipschitz, 
we have for any $\bx$ and $\bar{\bx}$ that
{
\begin{align}\label{eq:f-lip}
{f(\bx)\leq f(\bar{\bx})+\fprod{\nabla f(\bar{\bx}),~\bx-\bar{\bx}}+\sum_{i\in\cN}\tfrac{L_{f_i}}{2}\norm{x_i-\bar{x}_i}^2.}
\end{align}
}%
{Define $\bL=\diag([L_{f_i}\id_n]_{i\in\cN})$.} It follows from strong convexity of $\bar{f}$ that
for any $\mu\in(0,\ubar{\mu}]$ we have
{
\setlength{\arraycolsep}{0.0em}
\begin{eqnarray}\label{eq:g-strong}
f(\bx^*)&{}\geq{}& f(\bx^{k})+\fprod{\grad f(\bx^{k}),~\bx^*-\bx^{k}} +\tfrac{\mu}{2}\|\bx^*-\bx^{k}\|^2 \nonumber \\
&{}\geq{}& f(\bx^{k+1})+\fprod{\grad f(\bx^{k}),~\bx^*-\bx^{k+1}} +\tfrac{\mu}{2}\|\bx^*-\bx^{k}\|^2\nonumber \\
& & {-} \: \tfrac{1}{2}\|\bx^{k+1}-\bx^k\|^2_{\bL}.
\end{eqnarray}
\setlength{\arraycolsep}{5pt}}%
{where the last inequality follows from \eqref{eq:f-lip}.} 
Next, summing inequalities 
\eqref{eq:rho} and \eqref{eq:g-strong}, 
we get
{
\begin{align}\label{eq:lemma-x}
&\saa{\varphi(\bx^*)}+\tfrac{1}{2}\|\bx^*-\bx^{k}\|_{\mathbf{\widetilde{D}}^k_{\tau}}^2 
\geq \fprod{\bx^{k+1}-\bx^*,~\bp^{k}}\nonumber \\
&\qquad+\saa{\varphi(\bx^{k+1})}+\tfrac{1}{2}\|\bx^*-\bx^{k+1}\|_{\bD^k_{\tau}}^2+\tfrac{1}{2}\|\bx^{k+1}-\bx^k\|^{{2}}_{\mathbf{\bar{D}}_\tau^k}. 
\end{align}}%
Next, summing \eqref{eq:lemma-y} and \eqref{eq:lemma-x}, and rearranging terms, we obtain
{
\begin{eqnarray}\label{eq:inexact-lagrangian-bound}
\lefteqn{\mathcal{L}(\bx^{k+1},\by)-\mathcal{L}(\bx^*,\by^{k+1})\leq } 
\nonumber \\
& & E_1^{k+1}(\bmu)+\bigg[\tfrac{1}{2}\|\bx^*-\bx^{k}\|_{\mathbf{\widetilde{D}}^k_{\tau}}^2+\tfrac{1}{2}\|\by-\by^k\|_{\bD^k_{\kappa,\gamma}}^2\bigg]- \bigg[\tfrac{1}{2}\|\bx^*-\bx^{k+1}\|_{\bD^k_{\tau}}^2+\tfrac{1}{2}\|\by-\by^{k+1}\|_{\bD^k_{\kappa,\gamma}}^2 \nonumber \\
& &+\tfrac{1}{2}\|\by^{k+1}-\by^k\|_{\bD^k_{\kappa,\gamma}}^2+\tfrac{1}{2}\|\bx^{k+1}-\bx^k\|^2_{\mathbf{\bar{D}}^k_{\tau}}\bigg] +\fprod{\bp^k,\bx^*-\bx^{k+1}}+ \fprod{T(\bx^{k+1})-T(\bx^*),\by^{k+1}}.
\end{eqnarray}
}%
Recall \nsa{$\bp^{k}= (1+\eta^k)\bJ T(\bx^k)^\top\by^{k}-\eta^k\bJ T(\bx^{k-1})^\top\by^{k-1}$ where $\bJ T(\bx^k)^\top=[\bJ G(\bx^k)^\top~\id_{n|\cN|}]$} and note that { $\fprod{T(\bx^{k+1})-T(\bx^*),\by^{k+1}}=\fprod{G(\bx^{k+1})-G(\bx^*),\btheta^{k+1}}+\fprod{\bx^{k+1}-\bx^*,\bmu^{k+1}}$}. \nsa{Let $\cK\triangleq \Pi_{i\in\cN}\cK_i$ and note that {$G(\cdot)$} is $\cK$-convex. Therefore, using Remark~\ref{rem:K-convex}
and the fact that $\btheta^{k+1}\in\cK^*$}, we can 
bound the two inner products in \eqref{eq:inexact-lagrangian-bound}: 
{
\begin{align}\label{eq:inner-product-d}
& \fprod{T(\bx^{k+1})-T(\bx^*),~\by^{k+1}}+\fprod{\bp^k,~\bx^*-\bx^{k+1}}\nonumber\\
\leq & \fprod{\bJ T(\bx^{k+1})(\bx^{k+1}-\bx^*),\by^{k+1}}+\fprod{\bJ T(\bx^{k})(\bx^*-\bx^{k+1}),~\by^{k}} \nonumber\\
 & \mbox{ }+\eta^k\fprod{\bJ T(\bx^{k})^\top\by^k-\bJ T(\bx^{k-1})^\top\by^{k-1},~\bx^{*}-\bx^{k+1}}\nonumber\\
=&\fprod{\bJ T(\bx^{k+1})^\top\by^{k+1}-\bJ T(\bx^{k})^\top\by^{k},~\bx^{k+1}-\bx^*}\nonumber\\
 & \mbox{ }+\eta^k\fprod{\bJ T(\bx^{k})^\top\by^k-\bJ T(\bx^{k-1})^\top\by^{k-1},~\bx^{*}-\bx^{k}}\nonumber\\
 & \mbox{ }+\eta^k\fprod{\bJ T(\bx^{k})^\top\by^k-\bJ T(\bx^{k-1})^\top\by^{k-1},~\bx^{k}-\bx^{k+1}}.
\end{align}}%
{Moreover, the last inner product in \eqref{eq:inner-product-d} can be bounded as 
{
\begin{eqnarray}\label{eq:ineq-d}
\lefteqn{\eta^k\fprod{\bJ T(\bx^{k})^\top\by^k-\bJ T(\bx^{k-1})^\top\by^{k-1},~\bx^{k}-\bx^{k+1}}} \nonumber\\
&&=\eta^k\fprod{\bJ T(\bx^{k})^\top(\by^k-\by^{k-1}),~\bx^{k}-\bx^{k+1}}+\eta^k\fprod{\bJ T(\bx^{k})^\top\by^{k-1}-\bJ T(\bx^{k-1})^\top\by^{k-1},~\bx^{k}-\bx^{k+1}}\nonumber\\
&&\leq\eta^k\fprod{\bJ T(\bx^{k})^\top(\by^k-\by^{k-1}),~\bx^{k}-\bx^{k+1}}+\tfrac{\eta^k}{2}\|\bx^k-\bx^{k-1}\|_{B^{k-1}\bL_G}^2+\tfrac{\eta^k}{2}\|\bx^{k+1}-\bx^k\|_{B^{k-1}\bL_G}^2 \nonumber\\
&&\leq \tfrac{1}{2}\|\bx^{k+1}-\bx^{k}\|^2_{(\eta^k)^2\gamma^k(\delta+1)}+\tfrac{1}{2}\|\by^{k}-\by^{k-1}\|_{\bD^k_{\kappa,\gamma}}^2+\tfrac{\eta^k}{2}\|\bx^k-\bx^{k-1}\|_{B^{k-1}\bL_G}^2 \nonumber\\
&&\mbox{ }+\tfrac{\eta^k}{2}\|\bx^{k+1}-\bx^k\|_{B^{k-1}\bL_G}^2,
\end{eqnarray}}%
where in the first inequality we used Lipschitz continuity of $\bJ G$, the fact that $\bJ T(\bx^k)-\bJ T(\bx^{k-1})=\eyh{[(\bJ G(\bx^k)-\bJ G(\bx^{k-1}))^\top ; \boldsymbol{0}]}^\top$, and $B^{k-1}=\|{\btheta^{k-1}}\|$, and the last inequality
follows from Lemma \ref{lem:bound-J-s} 
and using the fact that
\nsa{$\delta\gamma^k\geq \max_{i\in\cN}{\kappa_i^k C_{g_i}^2}$.} 
Therefore, using \eqref{eq:ineq-d} within \eqref{eq:inner-product-d} and substituting the result in \eqref{eq:inexact-lagrangian-bound} lead to the desired result.}
\end{proof}
Now we are ready to prove Theorem~\ref{thm:dynamic-error-bounds}.\vspace*{-5mm}
\subsection{Proof of Theorem~\ref{thm:dynamic-error-bounds}}
Under Assumption~\ref{assump:saddle-point}, a saddle-point $(\bx^*,\by^*)$ for $\min_{\bx\in\cX}\max_{\by\in\cY}\cL(\bx,\by)$ in~\eqref{eq:dynamic-saddle} exists, where $\by^*=\eyh{[{\btheta^*}^\top;{\blambda^*}^\top]}^\top$; moreover,  \nsa{$(\bx^*,\btheta^*,\blambda^*)$ is a saddle-point of $\cL$ in~\eqref{eq:dynamic-saddle} if and only if $\bx^*=\one \otimes x^*$ such that $(x^*,\btheta^*)$ is a primal-dual solution to \eqref{eq:central_problem} and $\blambda^*\in\cC^\circ$, i.e., $\sum_{i\in\cN}\lambda_i^*=\mathbf{0}$. Therefore, through computing a saddle point to \eqref{eq:dynamic-saddle}, we indeed solve \eqref{eq:central_problem}. It is also worth emphasizing that if $(\bx^*,\btheta^*,\blambda^*)$ is a saddle-point of $\cL$ such that $\blambda^*\neq\mathbf{0}$, then it trivially follows that $(\bx^*,\btheta^*,\mathbf{0})$ is another saddle-point of $\mathcal{L}$.}

{We use induction to 
show that there exists some $B>0$ such that $\nsa{B^k\triangleq}\|\btheta^k\|\leq B$. In fact, since $\btheta^0=\boldsymbol{0}$, the bound holds trivially \nsa{for $k=0$ for any $B>0$. Given some $K\geq 1$} we assume that $
\nsa{B^k}\leq B$ holds for some $B>0$ for all $0\leq k\leq K-1$, and we will show it also holds for $k=K$.}
From the induction assumption, we have $(\eta^k)^2\gamma^k(\delta+1)\id_{nN} +\eta^kB^{k-1}\bL_G\preceq\cA^k$ in \eqref{lemeq:inexact-lagrangian-bound} for $0\leq k\leq K-1$.
Moreover, note the \nsa{step-sizes chosen as in Fig.~\ref{alg:PDD} satisfies the conditions in \eqref{eq:cond-d} for any $B\geq 0$}. Therefore, multiplying both sides of \eqref{lemeq:inexact-lagrangian-bound} by $\frac{\gamma^k}{\gamma^0}$ and using Lemma \ref{lem:step-size}, \nsa{for $0\leq k\leq K-1$}, we get
{
\begin{align}\label{eq:lagrangian-sum}
&\tfrac{\gamma^k}{\gamma^0}[\cL(\bx^{k+1},\by)-\cL(\bx^*,\by^{k+1})]\leq  
\nonumber\\
&\mbox{}\tfrac{\gamma^k}{\gamma^0}E^{k+1}(\bmu)+ \tfrac{\gamma^k}{\gamma^0}\bigg[\tfrac{1}{2}\|\bx^*-\bx^{k}\|_{\mathbf{\widetilde{D}}^k_{\tau}}^2
+\tfrac{1}{2}\|\by-\by^k\|_{\bD^k_{\kappa,\gamma}}^2 +\tfrac{1}{2}\|\by^k-\by^{k-1}\|^2_{\bD^k_{\kappa,\gamma}}\nonumber \\
&\mbox{}   +{\eta^k}\fprod{ \bJ T(\bx^k)^\top\by^k-\bJ T(\bx^{k-1})^\top\by^{k-1},~\bx^*-{\bx^k}} \nonumber \\
&\mbox{}  +\tfrac{1}{2}\|\bx^k-\bx^{k-1}\|^2_{\eta^kB\bL_G} \bigg]- \tfrac{\gamma^{k+1}}{\gamma^0} \bigg[\tfrac{1}{2}\|\bx^*-\bx^{k+1}\|_{\mathbf{\widetilde{D}}^{k+1}_{\tau}}^2  +\tfrac{1}{2}\|\by-\by^{k+1}\|_{\bD^{k+1}_{\kappa,\gamma}}^2+\tfrac{1}{2}\|\by^{k+1}-\by^{k}\|^2_{\bD^{k+1}_{\kappa,\gamma}}\nonumber\\
&\mbox{}  +\eta^{k+1}\fprod{ \bJ T(\bx^{k+1})^\top\by^{k+1}-\bJ T(\bx^{k})^\top\by^{k},~\bx^*-\bx^{k+1}} +\tfrac{1}{2}\|\bx^{k+1}-\bx^k\|^2_{\eta^{k+1}B\bL_G}\bigg].
\end{align}}%
We sum~\eqref{eq:lagrangian-sum} from $k=0$ to $K-1$; using Jensen inequality, and the fact $\bx^{-1}=\bx^0$ and \nsa{$\by^{-1}=\by^0$}, we get
{
\begin{align}\label{eq:dynamic-saddle-rate}
&{2 \cW_K\big(\mathcal{L}(\bar{\bx}^K,\by)-\mathcal{L}(\bx^*,\bar{\by}^K)\big) \leq } \nonumber \\
& \sum_{k=0}^{K-1}\frac{\gamma^k}{\gamma^0}E_1^{k+1}(\bmu) 
+ \|\bx^*-\bx^{0}\|_{{\widetilde{\bD}}^0_{\tau}}^2+\|\by-\by^0\|_{\bD^0_{\kappa,\gamma}}^2\nonumber \\
& \mbox{} -\frac{\gamma^K}{\gamma^0}\bigg[ \|\bx^*-\bx^{K}\|_{\widetilde{\bD}^K_{\tau}}^2+\|\by-\by^{K}\|_{\bD^K_{\kappa,\gamma}}^2+\|\by^K-\by^{K-1}\|_{\bD^K_{\kappa,\gamma}}^2
\nonumber \\
& \mbox{} +2\eta^{K}\fprod{ \bJ T(\bx^K)^\top\by^K-\bJ T(\bx^{K-1})^\top\by^{K-1},~\bx^*-\bx^K}  +\|\bx^{K}-\bx^{K-1}\|^2_{\eta^{K}B\bL_G}\bigg]  \nonumber \\
& \leq\sum_{k=0}^{K-1}\frac{\gamma^k}{\gamma^0}E^{k+1}(\bmu) 
+ \|\bx^*-\bx^{0}\|_{{\widetilde{\bD}}^0_{\tau}}^2+\|\by-\by^0\|_{\bD^0_{\kappa,\gamma}}^2 \nonumber \\
& \mbox{}
-\frac{\gamma^{K}}{\gamma^0}\bigg[
\tfrac{1}{2}\|\bx^*-\bx^{K}\|_{\widetilde{\bD}^K_{\tau}}^2+\|\by-\by^{K}\|_{\bD^{K}_{\kappa,\gamma}}^2\bigg],
\end{align}}%
where \nsa{$E^{k+1}(\bmu)$ 
is defined in Lemma~\ref{thm:dynamic-rate},} $\cW_K=\sum_{k=1}^{K}\frac{\gamma^{k-1}}{\gamma^0}$, $\bar{\bx}^{K}=\cW_K^{-1}\sum_{k=1}^{K}\frac{\gamma^{k-1}}{\gamma^0}\bx^k$ and {$\bar{\by}^{K}=\cW_K^{-1}\sum_{k=1}^{K}\frac{\gamma^{k-1}}{\gamma^0}\by^k$ for $\by^k=\eyh{[{\btheta^k}^\top; {\bmu^k}^\top]}^\top$ for $k\geq 0$}. The last inequality in~\eqref{eq:dynamic-saddle-rate} \nsa{follows from $B^{K-1}\leq B$ and \eqref{cond-6-d},
and we use \eqref{eq:ineq-d} for $k=K$ after $\bx^{k+1}$ is replaced with $\bx^*$ in \eqref{eq:ineq-d}.} 


Note that $E^{{k+1}}(\bmu)$ 
appearing in \eqref{eq:dynamic-saddle-rate} is the error terms due to approximating {$\cP_{\cC}$ with $\cR^k$} in the $k$-th iteration of the algorithm for $k\geq 0$. 
Furthermore, dropping the non-positive terms in \eqref{eq:dynamic-saddle-rate} and using $\tilde{\tau}^k>\tau^k$ for $k\geq 0$ lead to the following bound on the Lagrangian measure
{
\begin{align} \label{eq:saddle-rate-dynamic}
&\mathcal{L}(\bar{\bx}^K,\btheta,\bmu)-\mathcal{L}({\bf x}^*,\bar{\btheta}^K,\bar{\bmu}^K)\leq \nsa{\cT^K}(\bx^*,\btheta,\bmu)/\cW_K, \\
&\nsa{\cT^K}(\bx^*,\btheta,\bmu)\triangleq \sum_{i\in\mathcal{N}}\bigg[\frac{1}{ 2\saa{\tilde\tau^0}}\norm{x^*-x_i^0}^2+\frac{1}{2\kappa_i^0}\norm{\theta_i-\theta_i^0}^2 \bigg] \nonumber \\
&+\frac{1}{ 2\gamma^0}\norm{\bmu-\bmu^0}^2+\sum_{k=0}^{K-1} \frac{\gamma^k}{\nsa{2}\gamma^0} E^{k+1}(\bmu). \nonumber 
\end{align}
}%
\nsa{As discussed at the beginning of the proof, given a primal-dual solution $(x^*,\btheta^*)$ to \eqref{eq:central_problem}, $(\bx^*,\btheta^*,\blambda^*)$ such that $\blambda^*=\mathbf{0}$ and $\bx^*=\one \otimes x^*$ is a saddle-point for $\cL$ in \eqref{eq:dynamic-saddle}.}
Next, letting $\by=\by^*$ in \eqref{eq:dynamic-saddle-rate}, dropping the last non-positive term, and the fact that $\cL(\bar{\bx}^K,\by^*)-\cL({\bx}^*,\bar{\by}^K)\geq 0$, \nsa{we obtain
{
\begin{subequations}
\begin{align}
&\|\saa{\bx^{K}-\bx^*}\|^2\leq \frac{4\gamma^0\tilde{\tau}^K}{\gamma^K}~\nsa{\cT^K}(\bx^*,\btheta^*,\blambda^*),\label{eq:x-bound}\\
&\big\|\btheta^*-\btheta^K\big\|_{\gamma^K\bD^K_{\kappa}}^2\leq \nsa{2}\gamma^0\nsa{\cT^K}(\bx^*,\btheta^*,\blambda^*). \label{eq:theta-bound-1}
\end{align}
\end{subequations}}%
Next we bound $\nsa{\cT^K}(\bx^*,\btheta^*,\blambda^*)$ and show $\|{\btheta^K}\|\leq B$ for $B$ sufficiently large.
}

Using \eqref{eq:approx_error-for-full-vector-x} and the non-expansivity of projection, $\cP_{\cB}(\cdot)$, we conclude that $\|{\tcR}^k(\bx)-\mathcal{P}_{\Ct}(\bx)\|\leq N~\Gamma \beta^{q_k}\norm{\bx}$, for all $\bx$ \nsa{and $k\geq 0$}.
Moreover, since we assumed that each $\rho_i$ has a compact domain with diameter at most $\Delta$, we immediately conclude that $\|\bx^k\|\leq\sqrt{N}~\Delta$ 
for $k\geq 1$. Hence, from \eqref{eq:prox-error-seq-1} and using non-expansivity of prox operator \saa{together with 
\eqref{eq:approx_error-for-full-vector-x}, $\bom^k=\tfrac{1}{\gamma^k}{\bmu}^k+{{\bx}^{k+1}}$ and \eqref{eq:mu-bound} for $k=0,\ldots,K-1$,} we get
{
\begin{align*}
\|\be^{k+1}\|
&\leq N~\Gamma \beta^{q_k}\|\tfrac{1}{\gamma^k}{\bmu}^k+\bx^{k+1}\|
\leq N^{\frac{3}{ 2}}\Delta\Gamma \beta^{q_k}\Big(1+\frac{3}{\gamma^k}\sum_{t=0}^{k-1} \gamma^t \Big).\vspace*{-3mm}
\end{align*}
}%
Thus, \eqref{eq:mu-bound} implies that $\sum_{k=0}^{K-1} \frac{\gamma^{k}}{\nsa{2}\gamma^0} E^{k+1}(\blambda^*)\leq\saa{\tilde{\Lambda}(K)}$ where
{
\begin{align}\label{eq:lambda}
\saa{\tilde\Lambda(K)}\triangleq \sum_{k=0}^{K-1} \frac{\Delta^2}{\nsa{2}\gamma^0}N^{2} \Gamma\beta^{q_k}\Big(\gamma^k+3\sum_{t=0}^{k-1} \gamma^t\Big)\Big({\nsa{7}\gamma^k}+3\sum_{t=0}^{k-1} \gamma^t\Big).
\end{align}}
\saa{Therefore, since $\bmu^0=\mathbf{0}$ and $\btheta^0=\mathbf{0}$, we have
{
\begin{align}\label{eq:bound-at-optimal}
\cT^K(\bx^*,\btheta^*,\blambda^*)\leq \frac{1}{2\saa{\tilde\tau^0}}\norm{\bx^*-\bx^0}^2+\sum_{i\in\cN}\frac{1}{2\kappa_i^0}\norm{\theta_i^*}^2+\tilde\Lambda(K).
\end{align}}
Moreover, from \eqref{eq:theta-bound-1}, we get
$\|{\btheta^K}\|\leq \norm{\btheta^*}+[\frac{2\gamma^0}{\gamma^K}\max_{i\in\cN}\{\kappa_i^K\}\cT^K(\bx^*,\btheta^*,\blambda^*)]^{1/2}$.
From Lemma~\ref{lem:step-size} we have $\gamma^k/\gamma^0=\Theta(\mu k)$;
thus,
\eqref{eq:lambda} and \eqref{eq:bound-at-optimal} imply that $\cT^K(\bx^*,\btheta^*,\blambda^*)=\cO(B+\sum_{k=0}^{K-1}\beta^{q_k}k^4)$. 
Recall that we chose $\{q_k\}$ such that $\sum_{k=0}^{\infty}\beta^{q_k}k^4<\infty$; hence, we obtain $\|\btheta^K\|\leq B$ for any $B>0$ satisfying
\begin{align}
\label{eq:quadratic_II}
   \norm{\btheta^*}+ [2\gamma^0\tfrac{\delta}{C^2_{\min}}\sup_{k\in\integers_+}\cT^k(\bx^*,\btheta^*,\blambda^*)]^{1/2}\leq B.
\end{align}
Note since $\sup_{K\in\integers_+}\cT^K(\bx^*,\btheta^*,\blambda^*)=\cO(B)$, the inequality in \eqref{eq:quadratic_II} is quadratic in $B$; hence, there exists $\bar{B}>0$ such that any $B\geq \bar{B}$ satisfies \eqref{eq:quadratic_II}. This completes the induction ensuring $\|\btheta^k\|\leq B$ for all $k\geq 0$ and for any $B\geq \bar{B}$.}

\saa{Indeed, as we did in Remark~\ref{rem:bound1}, we can compute $\bar{B}$ in terms of algorithm and problem parameters. Note \eqref{eq:quadratic_II} can be equivalently written as
{
\begin{equation}\label{eq:B-bound-dynamic}
    B\geq  \norm{\btheta^*}+\sqrt{2\tilde A_0 B+\tilde A_1+\tilde A_2}
\end{equation}}%
where the constants $\tilde A_0, \tilde A_1$ and $\tilde A_2$ are defined as follows:
{
\begin{align}
    \tilde A_0 &\triangleq \frac{\gamma^0\delta L_{\max}(G)}{C_{\min}^2}\norm{\bx^*-\bx^0}^2 \label{eq:A0}\\
    \tilde A_1 &\triangleq [{(2\gamma^0(\delta+1)+L_{\max}{(f)})}
\gamma^0\delta \norm{\bx^*-\bx^0}^2+ \norm{\btheta^*}_{\bC^2}^2]/C_{\min}^2 \label{eq:A1}
\end{align}}%
and $\tilde A_2\triangleq 2\gamma^0\tfrac{\delta}{C_{\min}^2}\sup_{K\in\integers_+}\tilde\Lambda(K)$. To get a sufficient condition for \eqref{eq:B-bound-dynamic}, we will further upper bound $\tilde\Lambda(K)$. Note $\tilde\Lambda(K)\leq \sum_{k=0}^{K-1}\tfrac{\gamma^0}{2}\Delta^2N^2\Gamma\beta^{q_k}\left(7\sum_{t=0}^k\tfrac{\gamma^t}{\gamma^0}\right)^2$. From the proof of Lemma~\ref{lem:step-size}, we have $\tfrac{\gamma^k}{\gamma^0}\leq \frac{\mu\tilde{\tau}^0}{2}k$ for $k\geq 1$. Moreover, since $\tilde{\tau}^0\leq \frac{1}{L_{\max}(f)}$, we have $\tilde{A}_2\leq \dbtilde{A}_2$ where
{
\begin{align}
    \label{eq:A2}
    \dbtilde{A}_2\triangleq (\gamma^0)^2\frac{\delta}{C_{\min}^2}\Delta^2N^2\Gamma\sum_{k=0}^{\infty}\beta^{q_k}\left(7+\frac{\mu k (k+1)}{2L_{\max}(f)}\right)^2.
\end{align}}%
Trivially, \eqref{eq:B-bound-dynamic} holds for any $B\geq \bar B\triangleq \tilde A_0+\norm{\btheta^*}+\sqrt{\tilde A_0^2+2\tilde A_0\norm{\btheta^*}+\tilde A_1+\dbtilde A_2}$.}

Next, we show the rate result in \eqref{eq:rate_result-d}.
\nsa{Fix any $K\geq 1$, and} define $\tilde{\btheta}=[\tilde{\theta}_i]_{i\in\cN}$ where $\tilde{\theta}_i\triangleq 2\|\theta_i^*\|\big( \|\cP_{\mathcal{K}_i^*}(g_i(\bar{x}_i^K))\|\big)^{-1}~\cP_{\mathcal{K}_i^*}(g_i(\bar{x}_i^K))\in\cK_i^*$, which implies
{
\begin{equation}
\label{eq:tilde-theta-dynamic}
\langle g_i(\bar{x}_i^K), \tilde{\theta}_i \rangle=2\|\theta_i^*\|~d_{-\cK_i}(g_i(\bar{x}_i^K)).
\end{equation}}%
Note that ${\cC}$ is a closed convex cone, and 
$\mathcal{P}_{\cC}(\bx)=\one\otimes {p}(\bx)$ \nsa{where $\nsa{p(\bx)}\triangleq\frac{1}{ |\mathcal{N}|}\sum_{i\in\mathcal{N}}x_i$.} 
Similarly, define $\tilde{\bmu}=\frac{\mathcal{P}_{\cC^\circ}(\bar{\bx}^K)}{ \|\mathcal{P}_{\cC^\circ}(\bar{\bx}^K)\|}\in \cC^\circ$, where $\cC^\circ$ denotes polar cone of $\cC$. Hence, it can be verified that $\langle \tilde{\bmu}, \bar{\bx}^K\rangle= d_{\cC}(\bar{\bx}^K)$. 
Note that $\tilde{\bmu} \in \cC^\circ$ implies that $\sigma_{\cC}(\tilde{\bmu})=0$; moreover, we also have $\Ct \subseteq \cC$; hence, $\sigma_{\Ct} (\tilde{\bmu})\leq\sigma_{\cC}(\tilde{\bmu})=0$. Therefore, we can conclude that $\sigma_{\Ct}(\tilde{\bmu})=0$ since $\zero\in \Ct$. Together with \eqref{eq:tilde-theta-dynamic}, we get
{
\begin{equation}\label{eq:lagrange-equality}
\saa{\cL(\bar{\bx}^K,\tilde{\btheta},\tilde{\bmu})= \varphi(\bar{\bx}^K)+2\sum_{i\in\cN} d_{-\cK_i}(g_i(\bar{x}_i^K))\|\theta_i^*\|+d_{\cC}(\bar{\bx}^K).} \vspace*{-2mm}
\end{equation}}%
Since $\bx^*\in\Ct$, we also have that
{
\begin{equation}\label{eq:v-bar}
\fprod{\bx^*,~\bar{\bmu}^K}-\sigma_{\Ct}(\bar{\bmu}^K) \leq \sup_{\bmu} \fprod{\bx^*,~{\bmu}}-\sigma_{\Ct}({\bmu}) =\ind{\Ct}(\bx^*)=0.
\end{equation}}%
For any $i\in\cN$, $\bar{\theta}_i^K\in\nsa{\cK_i^*}$; hence, $\sigma_{-\cK_i}(\bar{\theta}_i^K)=0$. In addition, since $\bar{\theta}^K_i\in\cK_i^*$, and $g_i( x^*)\in\nsa{-\cK_i}$, we have
$\fprod{g_i(x^*),~\bar{\theta}^K_i} \leq 0$ which together with \eqref{eq:v-bar} implies $\cL(\bx^*,\bar{\btheta}^K,\bar{\bmu}^K) {\leq} \varphi(\bx^*)$. {Setting $\btheta=\tilde{\btheta}$ and $\bmu=\tilde{\bmu}$ \saa{within
the $\cT^K$ definition in~\eqref{eq:saddle-rate-dynamic},} 
one can show that $\cT^K(\bx^*,\tilde\btheta,\saa{\tilde\bmu})\leq \saa{\tilde\Lambda_0+\tilde\Lambda_1(K)}$, where
{
\begin{align*}
&\saa{\tilde\Lambda_0}\triangleq \frac{1}{2\gamma_0}+ \frac{1}{2\saa{\tilde\tau^0}}\norm{\bx^*-\bx^0}^2+\sum_{i\in\cN}\frac{2}{\kappa_i^0}\norm{\theta_i^*}^2,\\
&\saa{\tilde\Lambda_1(K)}\triangleq\sum_{k=0}^{K-1} \frac{\Delta^2}{2\gamma^0}N^{2} \Gamma\beta^{q_k}\Big(\gamma^k+3\sum_{t=0}^{k-1} \gamma^t\Big)\Big(1+ 7\gamma^k+ 3\sum_{t=0}^{k-1} \gamma^t\Big).
\end{align*}}}

Therefore, using \eqref{eq:lagrange-equality}, $\cL(\bx^*,\bar{\btheta}^K,\bar{\bmu}^K) {\leq} \varphi(\bx^*)$, the first inequality in \eqref{eq:saddle-rate-dynamic}, and the fact $\bmu^0=\mathbf{0}$ and $\btheta^0=\mathbf{0}$, we get
{
\begin{align}\label{eq:upper-bound-dynamic}
&\varphi(\bar{\bx}^K)-\varphi(\bx^*)+2\sum_{i\in\cN}\|\theta_i^*\|~ d_{-\cK_i}(g_i(\bar{x}_i^K))+d_{\cC}(\bar{\bx}^K) \nonumber \\
&\leq \saa{\frac{1}{\cW_K}\bigg(\tilde\Lambda_0+\tilde\Lambda_1(K)\bigg)= \frac{1}{\cW_K}\Lambda(K)},
\end{align}}%
\saa{where $\Lambda(K)\triangleq \tilde\Lambda_0+\tilde\Lambda_1(K)$.}
\nsa{\saa{The inequality \eqref{eq:upper-bound-dynamic}} follows from $\sum_{k=0}^{K-1} \frac{\gamma^{k}}{\nsa{2}\gamma^0} E_1^{k+1}(\tilde{\bmu})\leq \saa{\tilde\Lambda_1(K)}$ where we assume $\sqrt{N}\Delta\geq1$ (for simplicity of the bounds).}

Since $({\bf x}^*,\btheta^*,\blambda^*)$ is a saddle-point for $\cL$ in \eqref{eq:dynamic-saddle} \nsa{with $\blambda^*=\mathbf{0}$}, we have $\mathcal{L}(\bar{\bx}^K,\btheta^*,\blambda^*)-\mathcal{L}({\bf x}^*,\btheta^*,\blambda^*) \geq 0$; therefore,
\begin{equation}\label{eq:aux-lower-dynamic}
\varphi(\bar{\bx}^K)-\varphi({\bf x}^*)+\sum_{i\in\cN}\fprod{\theta_i^*,~g_i(\bar{x}_i^K)}\geq 0.
\end{equation}
Since $\theta_i^*\in\cK_i^*$, using conic decomposition of \nsa{$g_i(\bar{x}_i^K)$}, we obtain that
$\langle g_i(\bar{x}_i^K), \theta_i^*\rangle\leq \|\theta^*_i\|~d_{-\cK_i}(g_i(\bar{x}_i^K)).$ %
Thus, together with \eqref{eq:aux-lower-dynamic}, we conclude that
{
\begin{equation}\label{eq:lower-bound-dynamic}
\varphi(\bar{\bx}^K)-\varphi({\bf x}^*)+\sum_{i\in\cN}\|\theta^*_i\|~d_{-\cK_i}(g_i(\bar{x}_i^K)) \geq 0.
\end{equation}}%
\nsa{The desired result in \eqref{eq:rate_result-d-a} follows from combining \eqref{eq:upper-bound-dynamic} and \eqref{eq:lower-bound-dynamic}. Moreover, \eqref{eq:x-bound} and \break $\cT^K(\bx^*,{\btheta^*},\blambda^*)\leq \Lambda(K)$ imply \eqref{eq:rate_result-d-b}. Finally, recall that $\Lambda(K)=\cO(\sum_{k=0}^{K-1}\beta^{q_k}k^4)$; thus, \break $\sup_{K\in\integers_+}\Lambda(K)<\infty$ due to our choice of $\{q_k\}$ and this completes the proof.}
\qed


{\bf Proof of the claims in Remark~\ref{rem:N-complexity}.}
\saa{
Recall that the results of Theorem~\ref{thm:dynamic-error-bounds} are valid for any $B\geq \bar B\triangleq \tilde A_0+\norm{\btheta^*}+\sqrt{\tilde A_0^2+2\tilde A_0\norm{\btheta^*}+\tilde A_1+\dbtilde A_2}$ -- it is worth emphasizing that one can set $B=0$ when \ehh{$\{g_i\}_{i\in\cN}$} are affine. It follows from \eqref{eq:A0}, \eqref{eq:A1} and \eqref{eq:A2} that $\tilde A_0=\cO(\gamma^0N)$ as $\norm{\bx^*-\bx^0}^2=\cO(N)$, $\tilde A_1=\cO((\gamma^0+1)\gamma^0 N+N)$ as $\norm{\btheta^*}_{\bC^2}^2=\cO(N)$, and $\dbtilde A_2=\cO((\gamma^0)^2N^2\Gamma)$. Therefore,
{
$$B=\cO(\sqrt{N}+\gamma^0 N+\sqrt{(\gamma^0+1)\gamma^0 N+N}+N\gamma^0\sqrt{\Gamma}).$$
}%
Recall that the iteration complexity of the proposed method DPDA-TV is $\Lambda(K)/\cW_K$. Note that from Lemma \ref{lem:step-size}, $\cW_K=\Theta(\mu\tilde \tau^0 K^2)$ and $\Lambda(K)=\cO(\tfrac{N}{\gamma^0}+NB+(\gamma^0+1)N^2\Gamma)$. For an undirected time-varying graph $\{\cG_t\}$, $\Gamma=\nsa{\Theta(1/N)}$ and $\log(1/\varsigma)=\Omega(1/N^3)$ --see~\cite{nedic2009distributedquant}. Selecting
$\gamma^0=1/\sqrt{N}$ and $q_k=(5+c)\log_{{1}/{\varsigma}}(k+1)$ for some $(0,1)\ni\varsigma\geq \beta$, implies that the iteration complexity of DPDA-TV is $\cO(N^2/K^2)$ and drops to \ehh{$\cO(N/K^2)$} when \ehh{$\{g_i\}_{i\in\cN}$} are affine functions. Furthermore, the total number of communications to achieve $\epsilon$-suboptimality/infeasibility is $\tilde\cO(N^4/\sqrt{\epsilon})$, and drops to $\tilde\cO(N^{3.5}/\sqrt{\epsilon})$ when \ehh{$\{g_i\}_{i\in\cN}$} are affine functions. }


\section{Supplementary Material: Additional Numerical Experiments}
\label{supp:numeric}
In this section, we illustrate the performance of DPDA and DPDA-TV for solving synthetic C-LASSO problems. We first test the effect of network topology on the performance of proposed algorithms, and then we compare DPDA and DPDA-TV with other distributed primal-dual algorithms, DPDA-S and DPDA-D, proposed in~\cite{aybat2016primal} for solving \eqref{eq:central_problem} -- it is shown in~\cite{aybat2016primal} that both DPDA-S and DPDA-D converge
with $\cO(1/K)$ ergodic rate when $\bar{\varphi}$ is merely convex. In fact, when $\bar{\varphi}$ is strongly convex with modulus $\mu>0$, using the fact that {${\varphi}(\bar{\bx}^K)-{\varphi}(\bx^*)\geq \frac{\mu}{2}\norm{\bar{\bx}^K-\bx^*}^2$}, it immediately follows that $\norm{\bar{\bx}^K-\bx^*}^2\leq \cO(1/K)$.

We consider an isotonic C-LASSO problem over network $\cG^t=(\cN,\cE^t)$ for $t\geq 0$. This problem can be formulated in a centralized form as
$x^*\triangleq\argmin_{x\in\reals^{n}} ~\left\{ \frac{1}{2}\norm{Cx-d}^2+\lambda\norm{x}_1:\ Ax\leq {\bf 0}\right\}$,
where the matrix $C=[C_i]_{i\in\cN}\in\reals^{m|\cN|\times n}$, $d=[d_i]_{i\in \cN}\in\reals^{m|\cN|}$, and $A\in\reals^{n-1\times n}$. In fact, the matrix $A$ captures the isotonic feature of vector $x^*$, and can be written explicitly as, $A(\ell,\ell)=1$ and $A(\ell,\ell+1)=-1$, for $1\leq \ell \leq n-1$, otherwise it is zero.
Each agent $i$ has access to $C_i$, $d_i$, and $A$; hence, by making local copies of $x$, the decentralized formulation 
can be expressed as 
{
\begin{align}\label{prob:lasso-dist}
\min_{\substack{\bx=[x_i]_{i\in\cN}\in\cC, \\ Ax_i\leq {\bf 0}\ i\in\cN}} ~ \tfrac{1}{2}\sum_{i\in\cN}\norm{C_ix_i-d_i}^2+\tfrac{\lambda}{|\cN|}\sum_{i\in\cN}\norm{x_i}_1,
\end{align}}%
where $\cC$ is the consensus set - see \eqref{eq:consensus_set}.

We generated random C-LASSO problems that are strongly convex, i.e., $\min_{i\in\cN}\mu_i\triangleq\underline{\mu}>0$. In particular, we set $n=20$, $m=n+2$, $\lambda=0.05$ and $\cK_i=\reals^{n-1}_{+}$ for $i\in\cN$. Moreover, for each $i\in\cN$, we generate $C_i\in\reals^{m\times n}$ as follows: after $mn$ entries i.i.d. with {standard} Gaussian distribution are sampled, the condition number of $C_i$ is normalized by sampling the singular values from $[1,3]$ uniformly at random which implies that $\ubar{\mu}>0$.

We generate the first 5 and the last 5 components of $x^*$ by sampling from $[-10,0]$ and $[0,10]$ uniformly at random in ascending order, respectively, and the other middle 10 components are set to zero; hence, $[x^*]_j\leq [x^*]_{j+1}$ for $j=1,\ldots,n-1$. Finally, we set $d_i=C_i({x}^*+\epsilon_i)$,  where $\epsilon_i\in\reals^n$ is a random vector with i.i.d. components following Gaussian distribution with zero mean and standard deviation of $10^{-3}$.\\%
\indent{\bf Generating static undirected network:} $\cG=(\cN,\cE)$ is generated as a random small-world network. Given $|\cN|$ and the desired number of edges $|\cE|$, we choose $|\cN|$ edges creating a random cycle over nodes, and then the remaining $|\cE|-|\cN|$ edges are selected uniformly at random.\\%
\indent{\bf Generating time-varying undirected network:} Given $|\cN|$ and the desired number of edges $|\cE_0|$ for the initial graph, we generate a random small-world $\cG_0=(\cN,\cE_0)$ as described above. Given $M\in\integers_+$, and $p\in(0,1)$, for each $k\in\integers_+$, we generate $\cG^t=(\mathcal{N},\mathcal{E}^t)$, the communication network at time $t\in\{(k-1)M,\ldots,kM-2\}$ by sampling $\lceil p |\cE_0|\rceil$ edges of $\cG_0$ uniformly at random and we set $\mathcal{E}^{kM-1}=\cE_0\setminus \bigcup_{t=(k-1)M}^{kM-2}\cE^t$. In all experiments, we set $M=5$, $p=0.8$ and the number of communications per iteration is set to $q_k=10\ln(k+1)$.
\subsection{Effect of Network Topology}\label{sec:numeric-effect-net}
In this section, we test the performance of DPDA and DPDA-TV on \emph{undirected} communication networks. To show the effect of network topology, we consider four scenarios in which the number of nodes $|\cN|\in\{10,~40\}$ and the average number of edges per node $(|\cE|/|\cN|)$ is either $\approx 1.5$ or $\approx 4.5$. For each scenario, we plot 
both the 
relative error, i.e., $\max_{i\in\cN}\norm{x_i^k-x^*}/ \norm{x^*}$ and the infeasibility, i.e., 
$\max_{i\in\cN}\norm{(A\bar{x}_i^k)_+}$,
versus iteration number $k$. {All the plots show the average statistics over 25 replications.}

{\bf Testing DPDA on static undirected networks:} 
We generated the static small-world networks $\cG=(\cN,\cE)$ as described above for $(|\cN|,|\cE|)\in\{(10,15),~(10,45),$ $(40,60),~(40,180)\}$ and solve the saddle-point formulation \eqref{eq:static-saddle} corresponding to \eqref{prob:lasso-dist} using DPDA. 
For DPDA, displayed in Fig.~\ref{alg:PDS}, we chose $\delta=\nsa{L_\max(f)}$, $\gamma^0=\frac{1}{2d_{\max}+\nsa{L_\max(f)}}$, 
$\tilde{\tau}^0=\frac{1}{L_{\max}\nsa{(f)}+4}$, and $\kappa^0=\gamma^0\frac{L_{\max}\nsa{(f)}}{\norm{A}^2}$. In Fig.~\ref{static-net}, we plot $\max_{i\in\cN}\norm{x_i^k-x^*}/ \norm{x^*}$ and $\max_{i\in\cN}\norm{(A\bar{x}_i^k)_+}$ statistics for DPDA versus iteration number $k$. Note that compared to average edge density, the network size has more influence on the convergence rate, i.e., the smaller the network faster the convergence is. However, for fixed size network, as expected, higher the density faster the convergence is.
\begin{figure}[h]
\vspace*{-3mm}
\centering
\includegraphics[scale=0.25]{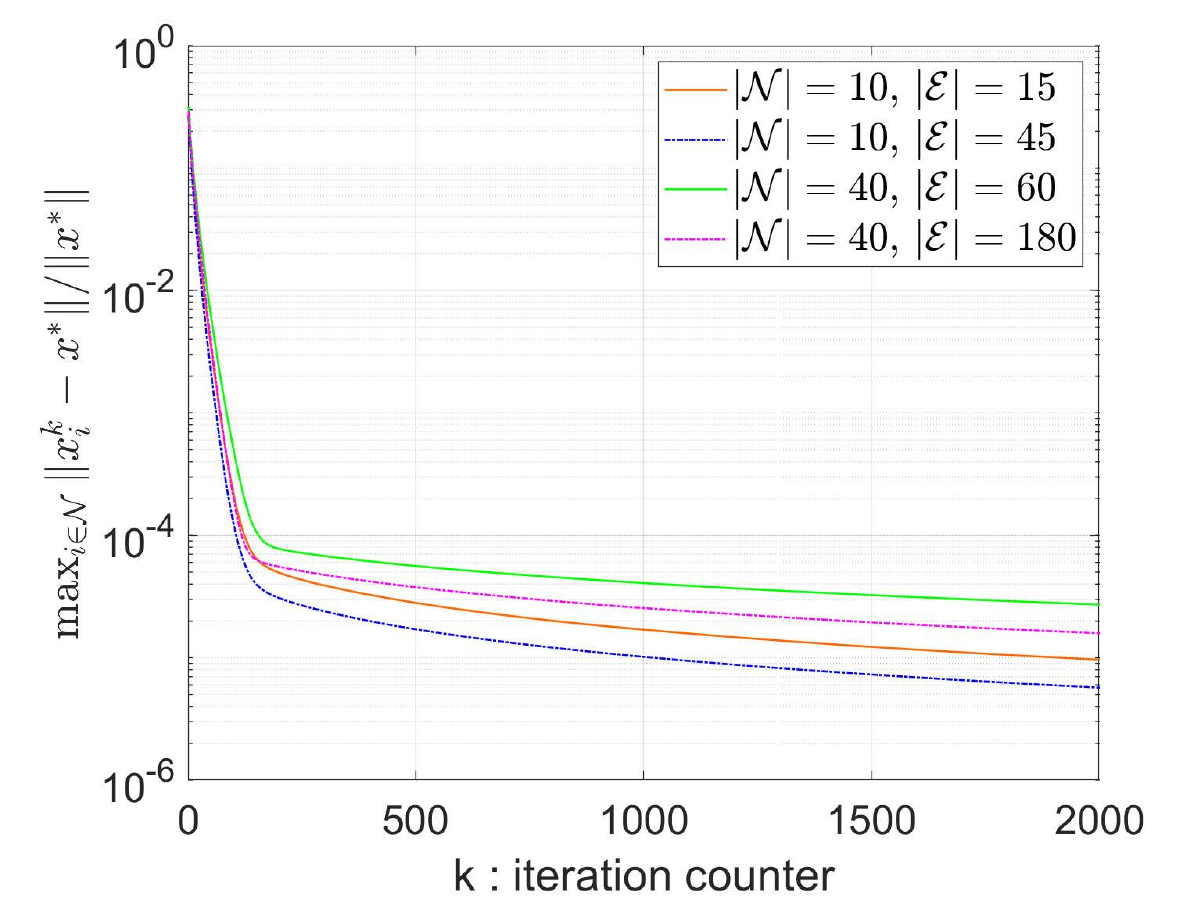}\hspace*{-3mm}
\includegraphics[scale=0.25]{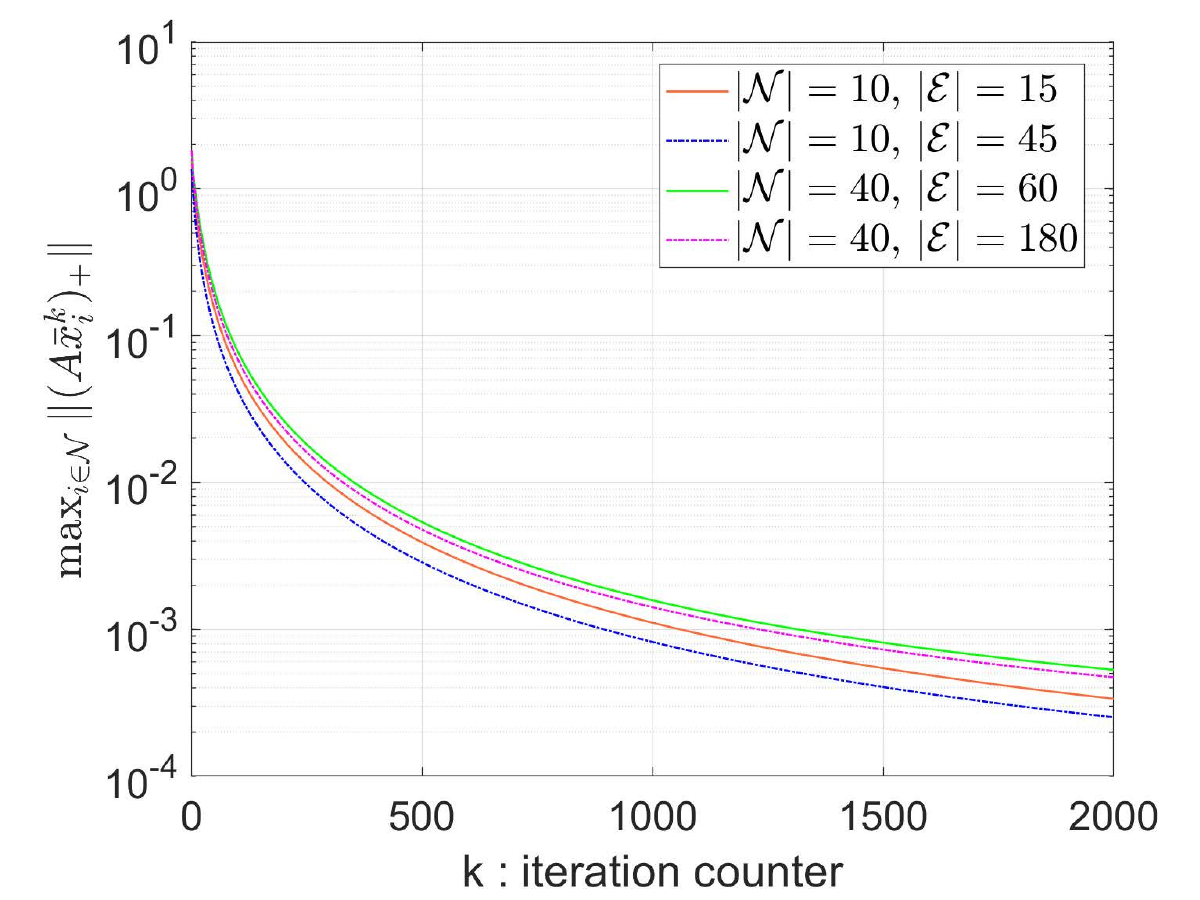}
\caption{Effect of network topology on the convergence rate of DPDA}\vspace*{-3mm}
\label{static-net}
\end{figure}
\begin{figure}[h]
\vspace*{-2mm}
\centering
\includegraphics[scale=0.25]{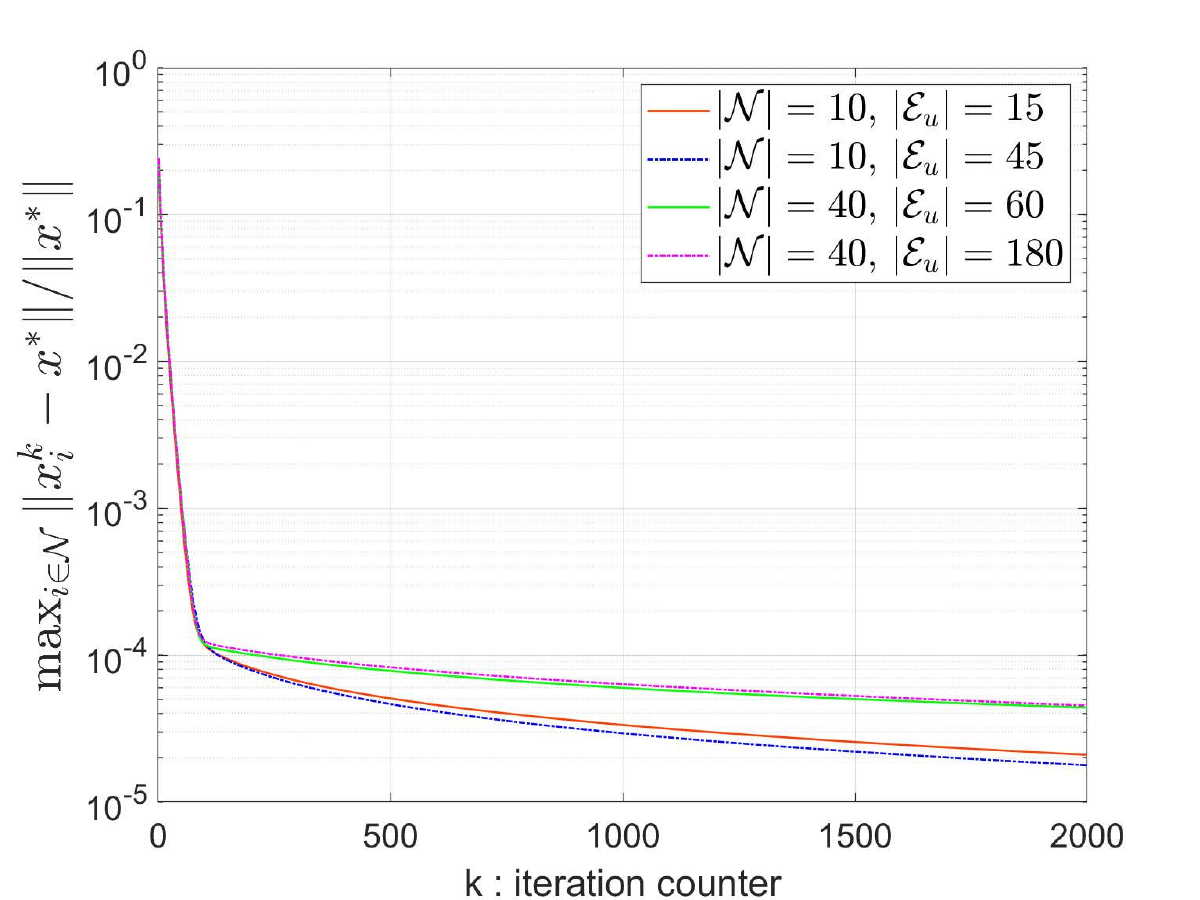}\hspace*{-3mm}
\includegraphics[scale=0.25]{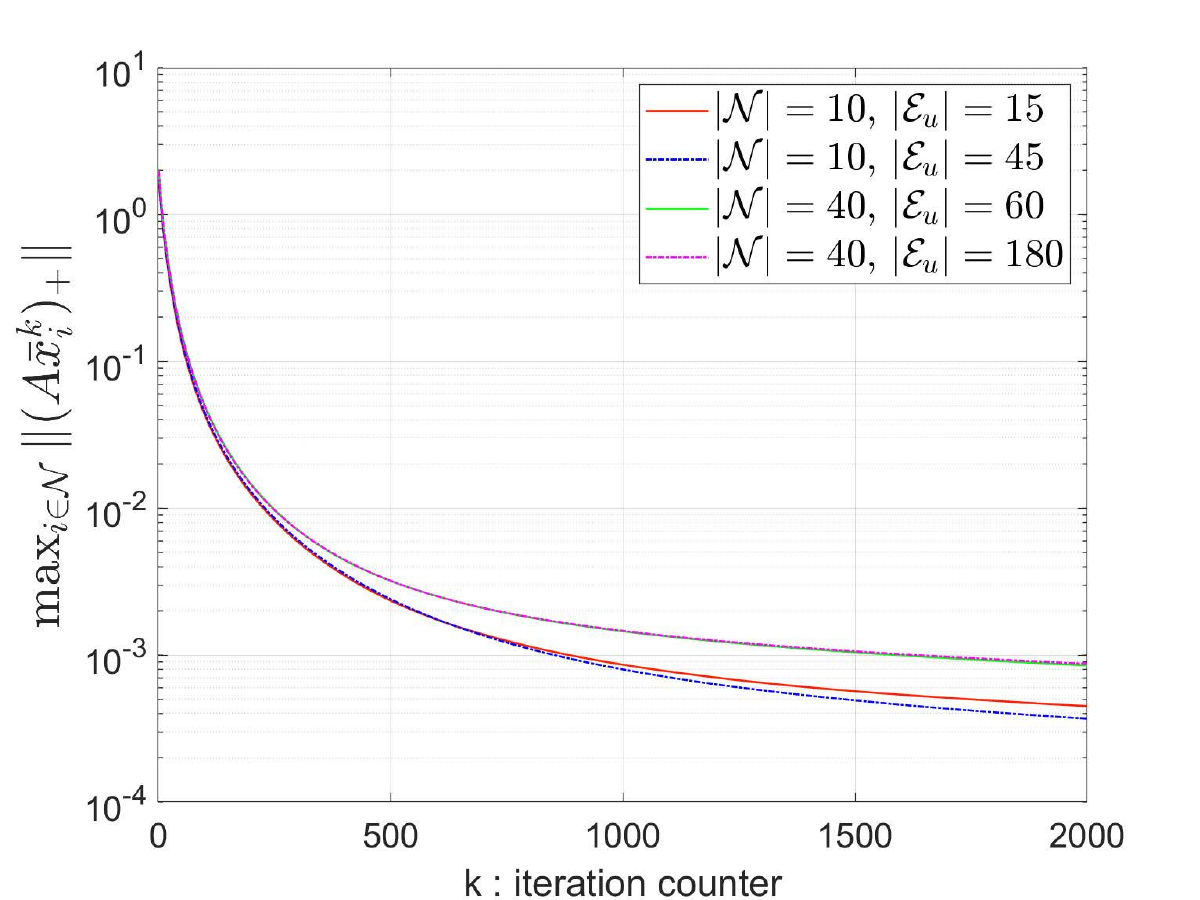}
\caption{Effect of network topology on the convergence rate of DPDA-TV}
\label{dynamic-net}
\vspace*{-2mm}
\end{figure}

{\bf Testing DPDA-TV on time-varying undirected networks:} We first generated an undirected graph $\cG_u=(\cN,\cE_u)$ as in the static case, and let $\cG_0=\cG_u$. Next, we generated $\{\cG^t\}_{t\geq 1}$ as described above by setting $M=5$ and $p=0.8$. For each consensus round $t\geq 1$, $V^t$ is formed according to Metropolis weights, i.e., for each $i\in\cN$, $V^t_{ij}=1/(\max\{d_i,d_j\}+1)$ if $j\in\cN_i^t$, $V^t_{ii}=1-\sum_{i\in\cN_i}V^t_{ij}$, and $V^t_{ij}=0$ otherwise -- see~\eqref{eq:approx-average-dual-undirected} for our choice of $\cR^k$.

For DPDA-TV, displayed in Fig.~\ref{alg:PDD}, we chose $\delta=1$, which lead to the initial step-sizes as {$\gamma^0=\frac{1}{2}$, $\tilde \tau^0=\frac{1}{L_{\max}\nsa{(f)}+2}$, and $\kappa^0=\frac{1}{2\norm{A}^2}$.}
In Fig.~\ref{dynamic-net}, we plot $\max_{i\in\cN}\norm{x_i^k-x^*}/ \norm{x^*}$ and $\max_{i\in\cN}\norm{(A\bar{x}_i^k)_+}$ statistics for DPDA-TV versus iteration number $k$ --  we used $\{\bx^k\}$ to compute the error statistics instead of $\{\bx^k\}$ as $\bx^k$ is never actually computed for DPDA-TV. Note that network size and average edge density have the same impact on the rate as in the static case.
\subsection{Comparison with other methods}
\label{sec:comparison}
We also compared our methods with DPDA-S and DPDA-D, in terms of the relative error and infeasibility of the ergodic iterate sequence, i.e., $ \max_{i\in\cN}\norm{\bar{x}_i^k-x^*}/ \norm{x^*}$ and $\max_{i\in\cN}\norm{(A\bar{x}_i^k)_+}$. We further report the performance of our algorithms in terms of relative error of the actual iterate sequence.
In this section we fix the number of nodes to $|\cN|=10$ and the average edge density to $|\cE|/|\cN|=4.5$ -- we observed the same convergence behavior for the other network scenarios discussed in the previous section.\\%
\indent{\bf Static undirected network:} We generated $\cG=(\cN,\cE)$ and chose the algorithm parameters as in the previous section. Moreover, the step-sizes of DPDA-S are set to the initial steps-sizes of DPDA. As it can be seen in Fig. \ref{static-compare}, DPDA has faster convergence when compared to DPDA-S.
\begin{figure}[h]
\vspace*{-3mm}
\centering
\includegraphics[scale=0.25]{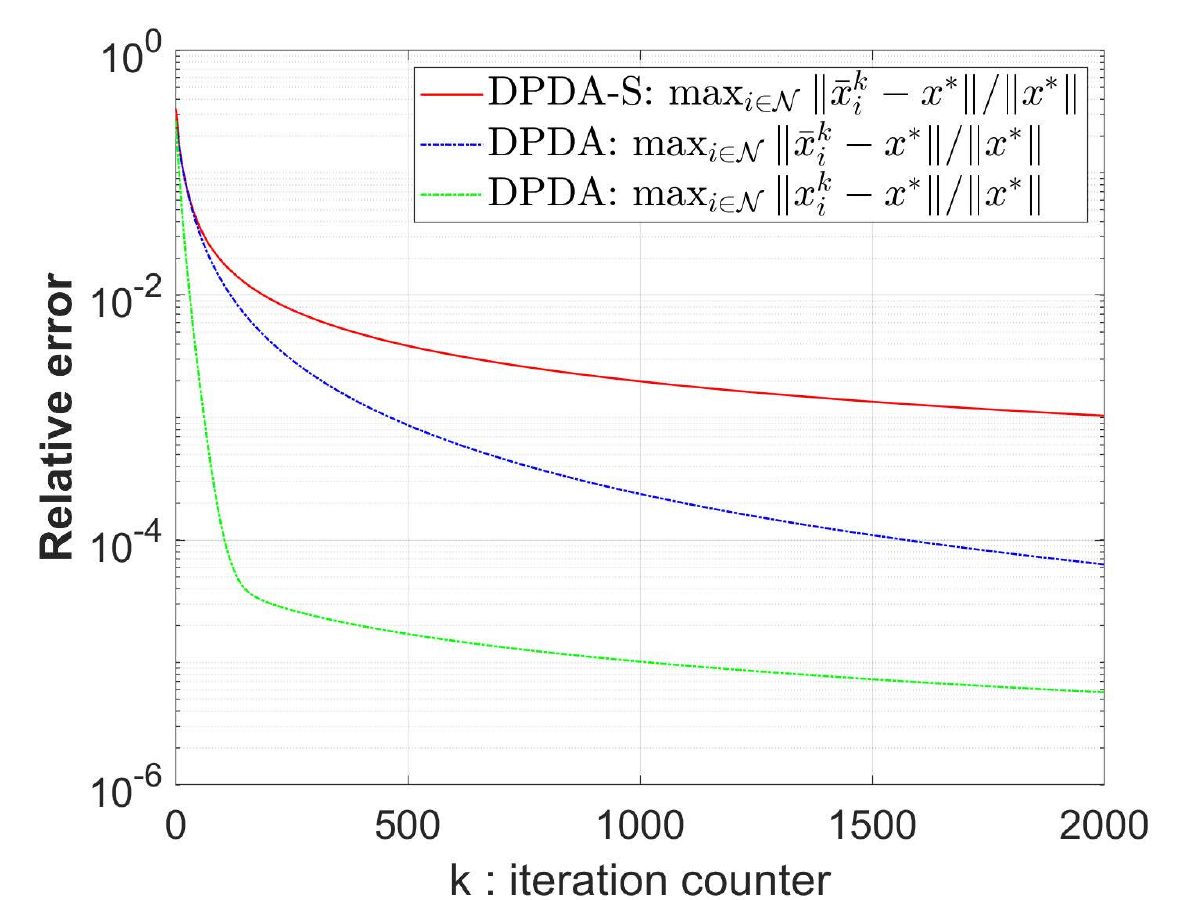}\hspace*{-3mm}
\includegraphics[scale=0.25]{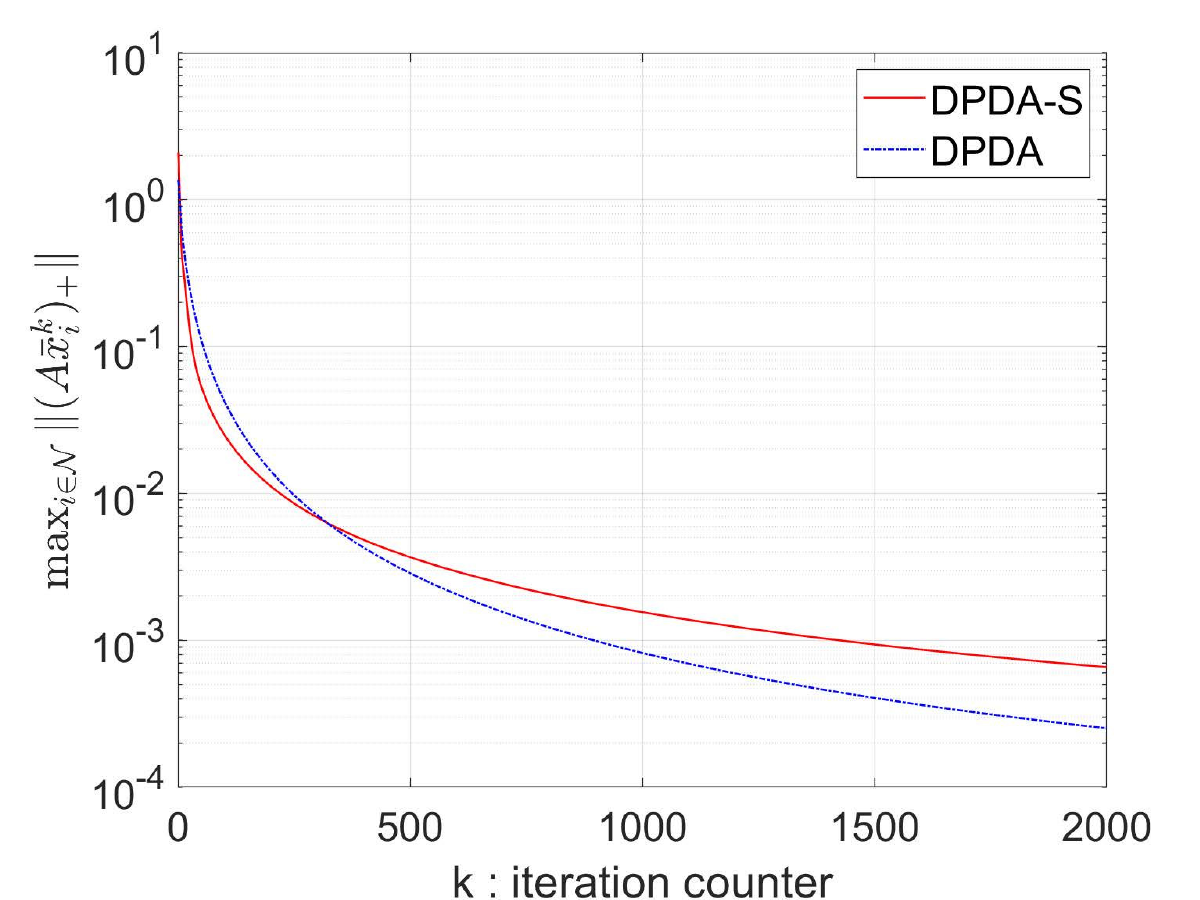}
\caption{DPDA vs DPDA-S over undirected static network}
\label{static-compare}
\end{figure}

{\bf Time-varying undirected network:} We generated the network sequence $\{\cG^t\}_{t\geq 0}$ and chose the parameters as in the \eyh{previous} section. The step-sizes of DPDA-D are also set to the initial steps-sizes of DPDA-TV. Fig. \ref{undirected-compare} shows that DPDA-TV has faster convergence when compared to DPDA-D.
\begin{figure}[h]
\vspace*{-3mm}
\centering
\includegraphics[scale=0.25]{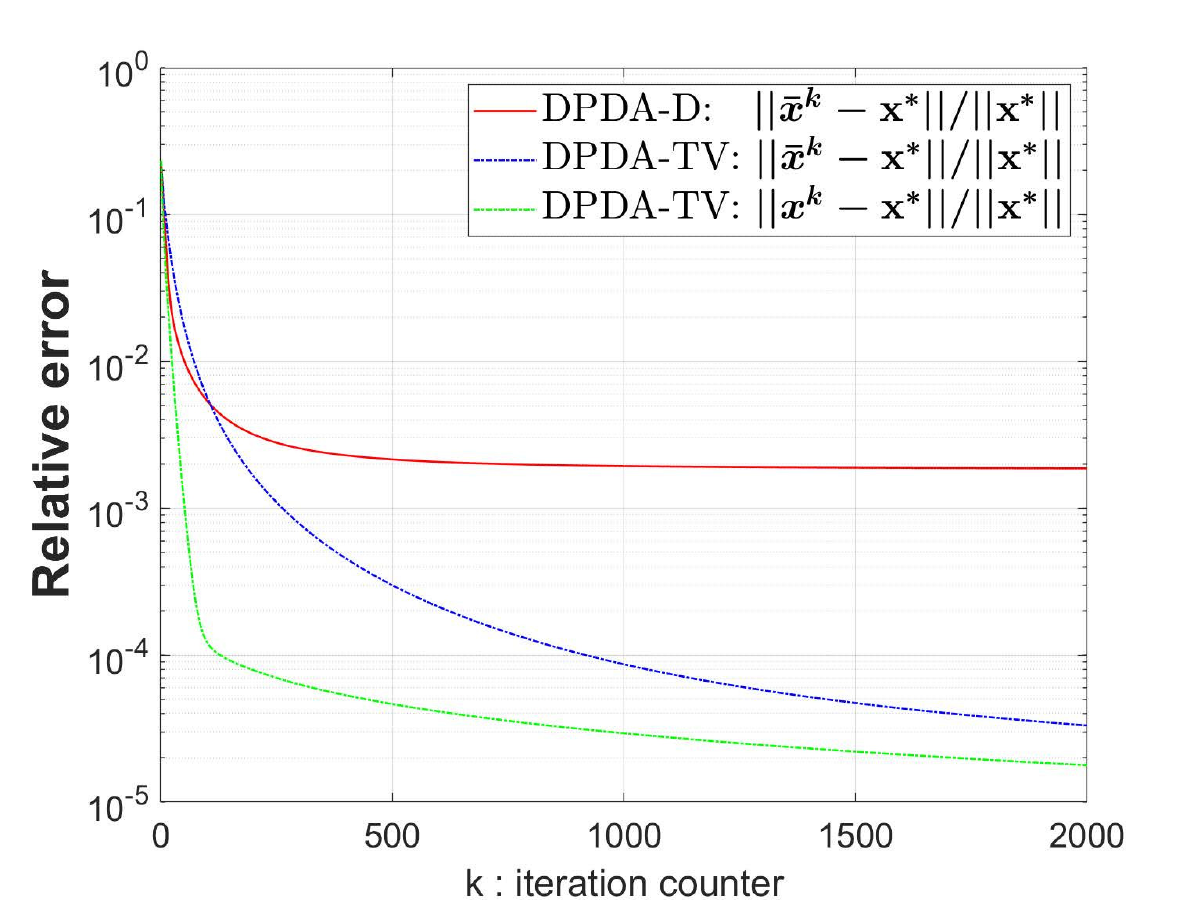}\hspace*{-3mm}
\includegraphics[scale=0.25]{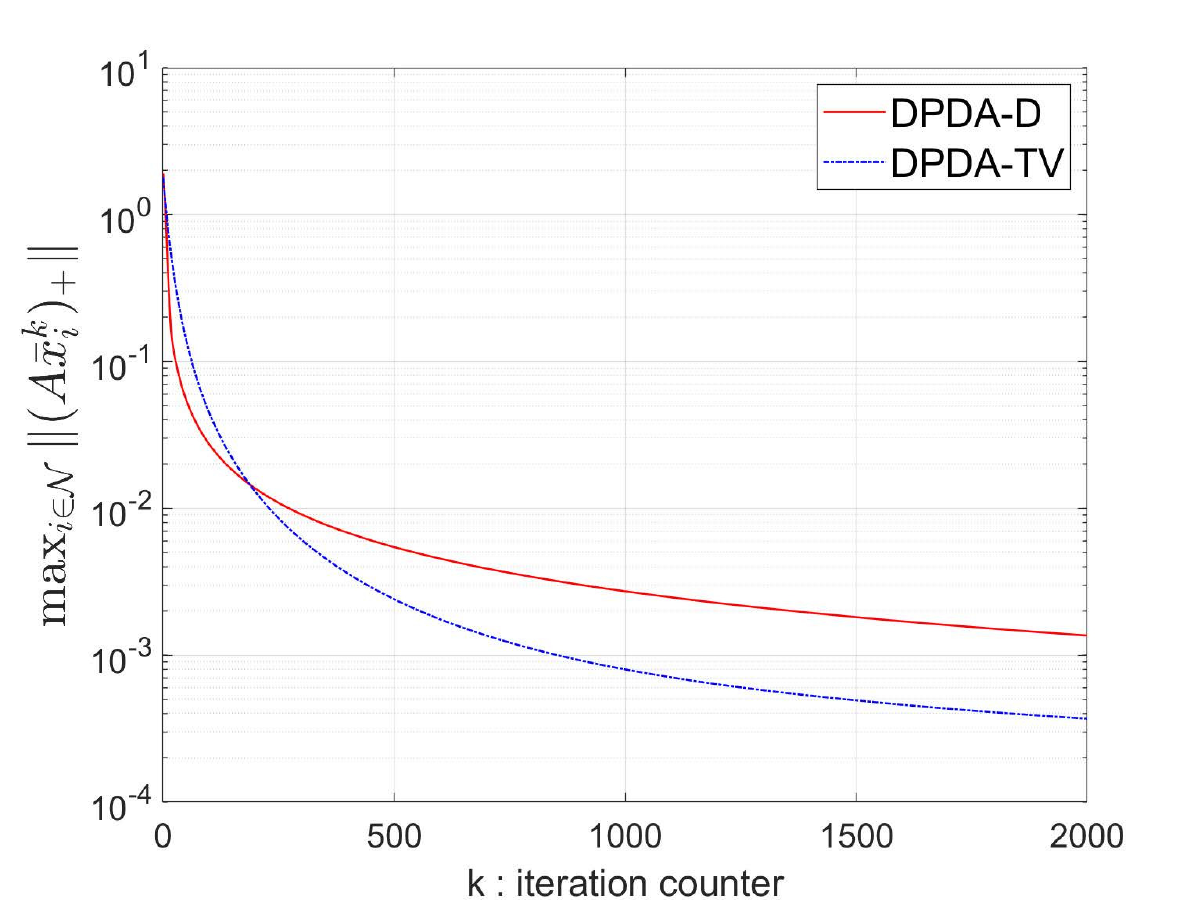}
\caption{DPDA-TV vs DPDA-D over undirected time-varying network}
\label{undirected-compare}
\end{figure}

{\bf Time-varying directed network:} In this scenario, we generated time-varying communication networks similar to \cite{nedich2016achieving}. Let $\cG_d=(\cN,\cE_d)$ be the directed graph shown in Fig.~\ref{fig:Gd} where it has $|\cN|=12$ nodes and $|\cE_d|=24$ directed edges. {We set $\cG_0=\cG_d$, and we generate $\{\cG^t\}_{t\geq 0}$ generated as in the undirected case with parameters $M=5$ and $p=0.8$; hence, $\{\cG^t\}_{t\geq 0}$ is 
$M$-strongly-connected. Moreover, communication weight matrices $V^t$ are formed according to rule \eqref{eq:directed-weights}. We chose the initial step-sizes for DPDA-TV as in the time-varying undirected case, and the constant step-sizes of DPDA-D is set to the initial steps-sizes of DPDA-TV.
In Fig.~\ref{directed-compare} we compare DPDA-TV against DPDA-D. We observe that over time-varying directed networks DPDA-TV again outperforms DPDA-D for both statistics.}
\begin{figure}[h]
\vspace*{-3mm}
\centering
\includegraphics[scale=0.25]{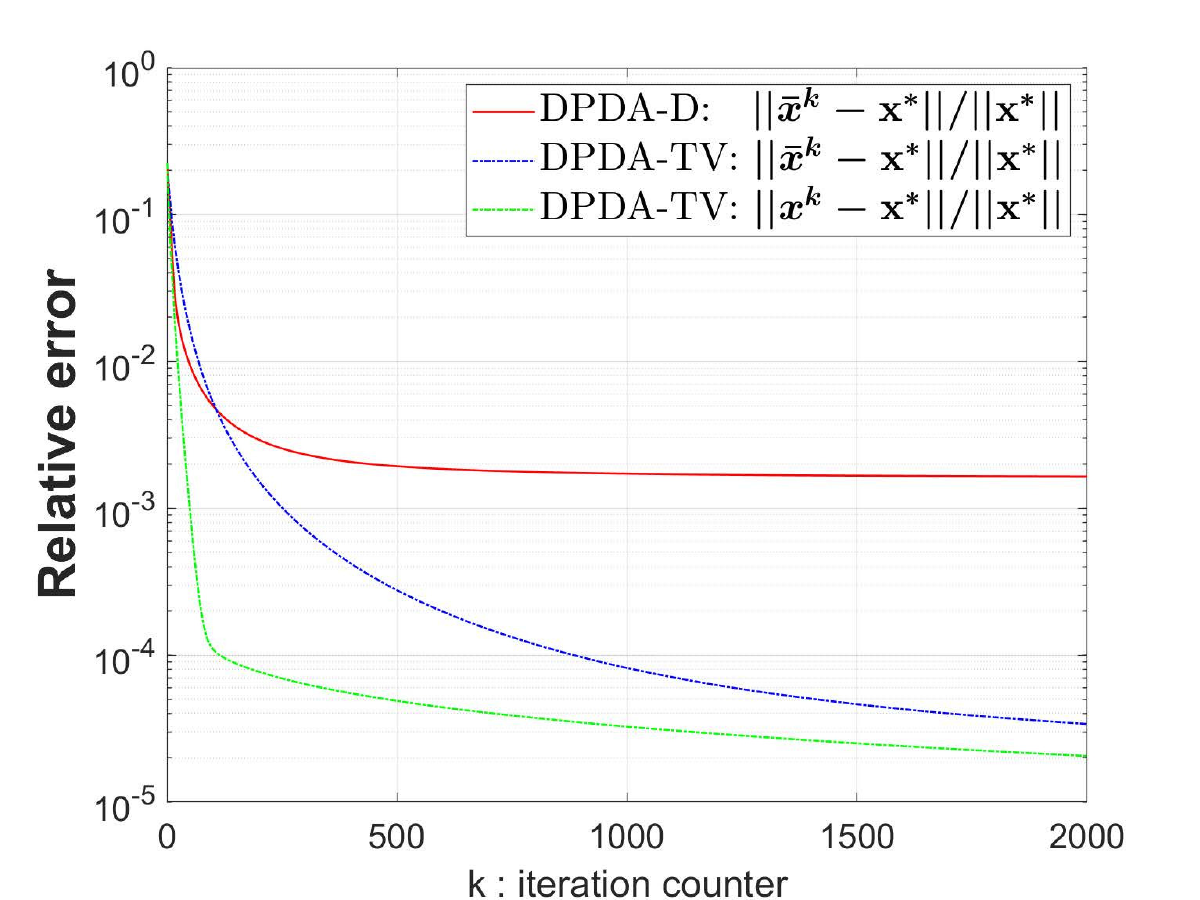}\hspace*{-3mm}
\includegraphics[scale=0.25]{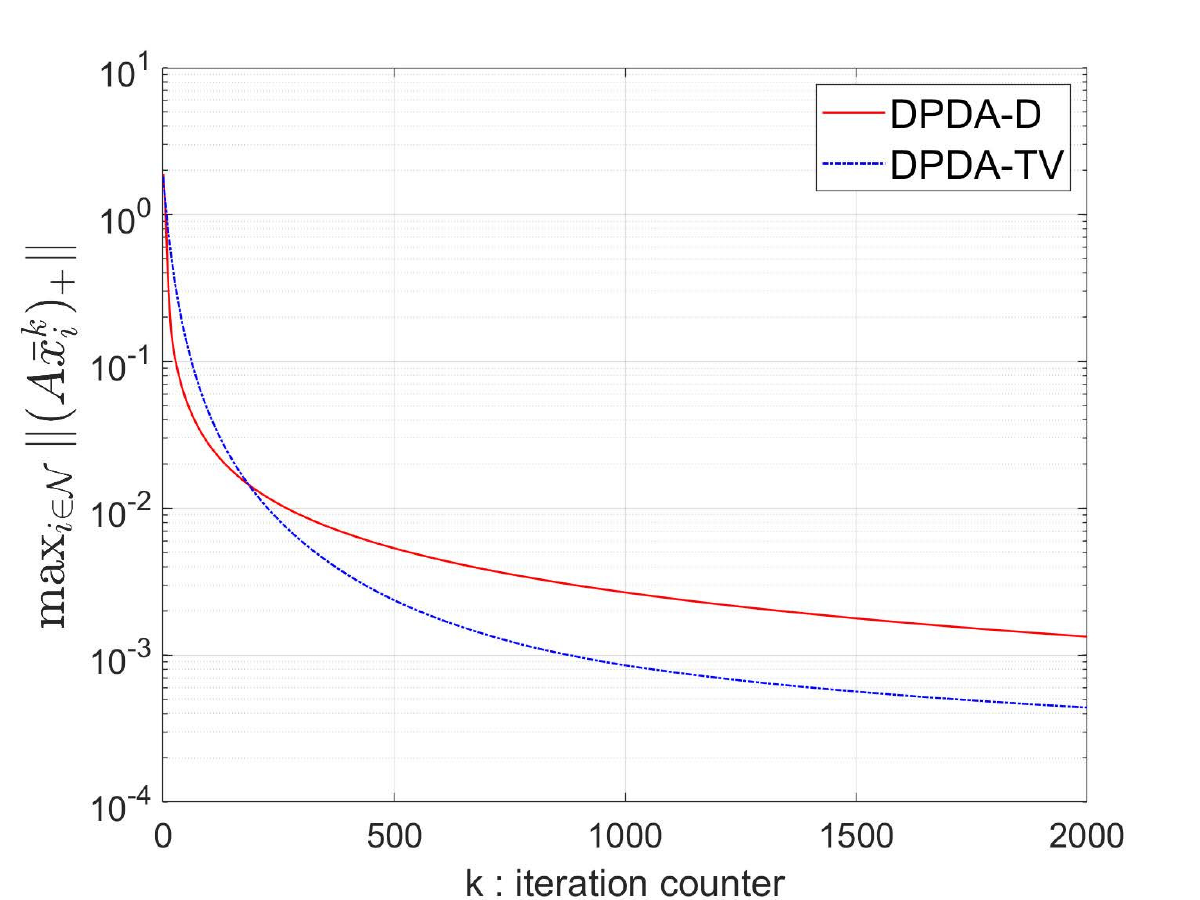}
\caption{DPDA-TV vs DPDA-D over directed time-varying network.}
\label{directed-compare}
\end{figure}

\end{document}